\documentclass[a4paper,english,twoside,leqno,12pt]{smfartVS}

\usepackage{smfenum,smfthm}
\usepackage{amssymb,amsmath}
\usepackage{mathrsfs,color,euscript}
\usepackage[latin1]{inputenc}
\input{xypic}

\numberwithin{equation}{section} 
\theoremstyle{plain}

\newcounter{nonum}

\newtheorem{theon}[nonum]{Theorem}

\def\CC{\mathbb{C}}

\def\ZZ{\mathbb{Z}} 

\def\A{{\rm A}}
\def\B{{\rm B}}
\def\C{{\rm C}}
\def\D{{\rm D}}
\def\E{{\rm E}}
\def\F{{\rm F}}
\def\G{{\rm G}}
\def\H{{\rm H}}
\def\I{{\rm I}}
\def\J{{\rm J}}
\def\K{{\rm K}}
\def\L{{\rm L}}
\def\M{{\rm M}}
\def\N{{\rm N}}
\def\P{{\rm P}}

\def\R{{\rm R}}
\def\SS{{\rm S}}
\def\T{{\rm T}}
\def\U{{\rm U}}
\def\V{{\rm V}}
\def\W{{\rm W}}

\def\Z{{\rm Z}}

\def\Aa{\mathscr{A}}
\def\Bb{\mathscr{B}}
\def\Cc{\EuScript{C}}

\def\Ee{\EuScript{E}}

\def\Ii{\mathscr{I}}

\def\Oo{\EuScript{O}}
\def\Pp{\mathscr{P}}
\def\Qq{\mathscr{Q}}
\def\Qq{\EuScript{Q}}

\def\Ss{\mathfrak{S}}
\def\Tt{\mathfrak{T}}

\def\Ww{\mathscr{W}}

\def\AA{\mathfrak{A}}
\def\BB{\mathfrak{B}}
\def\HH{\mathfrak{H}}
\def\JJ{\mathfrak{J}}
\def\KK{\mathfrak{K}}
\def\PP{\mathfrak{P}}

\def\Ga{\Gamma}
\def\La{\Lambda}
\def\Om{\Omega}

\def\a{\alpha} 
\def\b{\beta}
\def\d{\delta}
\def\e{{\rm e}}

\def\g{\gamma}
\def\h{\varphi}

\def\l{\lambda}

\def\n{\eta}

\def\p{\mathfrak{p}}

\def\s{\sigma}
\def\t{\theta}
\def\v{\upsilon}
\def\w{\varpi}

\def\ie{that is, }

\def\>{\geqslant}
\def\<{\leqslant}

\def\Hom{{\rm Hom}}
\def\End{{\rm End}}
\def\Aut{{\rm Aut}}
\def\Mat{{\rm M}}
\def\GL{{\rm GL}}
\def\Gal{{\rm Gal}}

\def\tr{{\rm tr}}

\def\mult#1{{#1}^{\times}}

\def\CC{\mathfrak{C}}

\def\PB{\mathfrak{Q}}

\def\aa{\mathfrak{a}}

\def\AL{\overline\A{}}
\def\BL{\overline\B{}}
\def\CL{\overline\C{}}

\def\RL{\overline\R{}}

\def\UL{\overline\U{}}
\def\VL{\overline\V{}}

\def\GaL{\overline\Ga{}}
\def\LaL{\overline\La{}}

\def\AH{\A_{\K}}

\def\aa{\mathfrak{P}}

\def\({\left(}
\def\){\right)}

\def\psi{\Psi}
\def\PSI{\boldsymbol{\psi}}

\def\d{f}
\def\c{s}
\def\ee{{\rm c}}
\def\bc{\boldsymbol{b}}
\def\ii{\iota}
\def\can{\varkappa}

\def\flit{\diamond}

\def\Ee{\EuScript{E}}

\def\flop{\prime}

\def\up{0}
\def\ce{{\bf 1}}

\def\ss{\mathfrak{s}}

\def\JL{{\bf JL}}

\def\qq#1{\widehat{#1}}

\author{P. Broussous}
\address{
Universit\'e de Poitiers\\ 
Laboratoire de Math\'ematiques\\ 
T\'el\'eport 2\\ 
Boulevard Marie et Pierre Curie\\ 
BP 30179, F-86962 Futuroscope-Chasseneuil cedex\\
France}
\email{Paul.Broussous@math.univ-poitiers.fr}

\author{V. S\'echerre}
\address{Institut de Math\'ematiques de Luminy\\
CNRS UMR $6206$\\
Universit\'e de la M\'editerran\'ee, 163 avenue de Luminy\\
$13288$ Marseille Cedex $09$\\
France}
\email{secherre@iml.univ-mrs.fr}

\author{S. Stevens}
\address{School of Mathematics\\
University of East Anglia\\
Norwich NR4 7TJ\\
United Kingdom}
\email{Shaun.Stevens@uea.ac.uk}

\title[Smooth representations of $\GL_{m}(\D)$, V]
{Smooth representations of $\GL_{m}(\D)$\\ V: Endo-classes}

\begin{abstract}
Let $\F$ be a locally compact nonarchimedean local field. 
In this article, we extend to any inner form of $\GL_n$ 
over $\F$, with $n\>1$, the notion of endo-class introduced
by Bush\-nell and Henniart for $\GL_n(\F)$. 
We investigate the intertwining relations of simple characters of these groups,
in particular their preservation properties under transfer. 
This allows us to associate to any discrete series representation
of an inner form of $\GL_n(\F)$ an endo-class over $\F$.
We conjecture that this endo-class is invariant under the 
local Jacquet-Langlands correspondence.
\end{abstract}

\begin{altabstract}
Soit $\F$ un corps commutatif localement compact non archim\'edien. 
Dans cet article, nous \'etendons \`a toutes les formes int\'erieures 
de $\GL_n$ sur $\F$, avec $n\>1$, la notion d'endo-classe introduite 
par Bushnell et Henniart pour $\GL_n(\F)$.
Nous \'etudions les pro\-pri\'e\-t\'es des caract\`eres simples de 
ces groupes vis-\`a-vis de l'entrelacement, et \'etablissons en 
par\-ti\-cu\-lier la permanence de ces propri\'et\'es par transfert.
Ceci nous permet d'associer \`a toute re\-pr\'e\-sen\-ta\-tion 
irr\'eductible de la s\'erie discr\`ete d'une forme int\'erieure de 
$\GL_n(\F)$ une endo-classe sur $\F$.
Nous conjecturons que cette endo-classe est invariante par la 
correspondance de Jacquet-Langlands locale.
\end{altabstract}

\thanks{
This work was supported by: EPSRC (grant GR/T21714/01), 
the British Council and the Partenariat Hubert Curien 
in the framework of the Alliance programme (number 19418YK), 
the Agence Nationale de la Recherche (ANR-08-BLAN-0259-01) 
and the Universit\'e de la M\'editerran\'ee Aix-Marseille 2. 
}

\dedicatory{In memory of Martin Grabitz}

\begin{document}

\maketitle

\tableofcontents

\section*{Introduction}

This is the fifth in a series of articles whose objective is a
complete description of the category of smooth complex representations
of $\GL_{m}(\D)$, with $m$ a positive integer and 
$\D$ a division algebra over a locally compact 
nonarchimedean local field.
The longer term aim is an
explicit description, in terms of types, of the local
Jacquet-Langlands cor\-res\-pon\-den\-ce~\cite{DKV,Ba}, as begun by Bushnell
and Henniart~\cite{HJL1,BHJL2,BHLTL3}, and by Silberger and
Zink~\cite{SZ1,SZ2}.

\medskip

The main object of study in this paper is the notion of
\emph{endo-class}, introduced by Bush\-nell and Henniart~\cite{BH}. 
An endo-class (over a locally compact nonarchimedean
local field~$\F$) is an invariant associated to an irreducible
cuspidal representation of~$\GL_n(\F)$, con\-struc\-ted by
explicit methods related to the description of this  
representation as com\-pact\-ly induced from an irreducible 
representation of
a compact-mod-centre subgroup of~$\GL_n(\F)$ (see~\cite{BK,BH}). 
The arithmetic significance of this invariant has been described by
Bushnell and Henniart in~\cite{BHLTL4}, in the case where $\F$ is of
characteristic zero: if we denote by~$\Ww_\F$ the Weil group of~$\F$
(relative to an algebraic closure) and by~$\Pp_\F$ its wild inertia
subgroup, there is a bijection between the set~$\Ee(\F)$ of
endo-classes over~$\F$ and the set of $\Ww_\F$-conjugacy classes of
irreducible representations of~$\Pp_\F$, which is compatible with the
local Langlands correspondence.

\medskip

In this article, we extend the notion of endo-class to any inner form
of~$\GL_n(\F)$, $n\>1$, \ie to any group of the form~$\GL_m(\D)$,
with~$m$ a positive integer 
and~$\D$ an $\F$-central division algebra of dimension
$d^2$ over~$\F$, with $n=md$. If~$\G$ is an inner form
of~$\H=\GL_n(\F)$, and if~$\EuScript{D}(\G)$ denotes the discrete
series of~$\G$ (\ie the set of isomorphism classes of essentially
square-integrable irreducible representations of~$\G$), we define a map:
\begin{equation*}
\boldsymbol{\Theta}_{\G}:\EuScript{D}(\G)\to\Ee(\F)
\end{equation*}
(see paragraph \ref{daga}) 
which associates an endo-class over~$\F$ to any discrete series
representation of~$\G$. This map should play an important role in an
explicit description of the local Jacquet--Langlands correspondence:
\begin{equation*}
\label{JLC}
{\bf JL}:\EuScript{D}(\G)\to\EuScript{D}(\H).
\end{equation*}
In particular, we expect that ${\bf JL}$ preserves the endo-class (see
Conjecture~\ref{EndoClassJL1}), that is:
\begin{equation*}
\boldsymbol{\Theta}_{\H}\circ{\bf JL}=\boldsymbol{\Theta}_{\G}.
\end{equation*}
This conjectural property can be seen as a generalization of the fact
that the correspondence ${\bf JL}$ preserves the representations of
level zero (see~\cite{SZ1}). The notion of endo-class also plays a
central role in the construction of semisimple types (in an article in
preparation~\cite{BSS2}), which will give a complete description of
the structure of the category of smooth complex representations of~$\G$.

\medskip

One of the objectives of~\cite{VS1}, completed in~\cite{SeSt}, is the
construction of \emph{simple characters}, which are certain special 
characters of particular compact open subgroups of~$\G$. These simple
characters are attached to data called \emph{simple strata}, and are
a fundamental part of the construction of more elaborate objects
called \emph{simple types} (see~\cite{VS2,VS3}). 
One knows from~\cite{VS3,SeSt} that every irreducible discrete series
representation~$\pi$ of~$\G$ contains a simple character~$\t$ attached
to a simple stratum. 
Neither the simple stratum nor the simple
cha\-rac\-ter are unique, but every other simple character~$\t'$ contained
in~$\pi$ intertwines~$\t$, \ie there is an element~$g\in\G$ such that
$\t'$ and the conjugate character $\t^g$ coincide on the intersection
of the compact open subgroups where they are defined. It is this
observation which leads to the notion of endo-class.

\medskip

An endo-class is an equivalence class of objects called
\emph{potential simple characters} (or \emph{ps-characters} for short), 
for a relation called \emph{endo-equivalence}. 
A ps-character $\Theta$ 
is characterized by giving a simple stratum~$[\La,n,m,\b]$ 
in an~$\F$-central simple algebra $\A$ and a simple character~$\t$ attached 
to this simple stratum. 
The pair $([\La,n,m,\b],\t)$ is called a {\it real\-iza\-tion} of $\Theta$.
Another simple stratum~$[\La',n',m',\b]$ in another~$\F$-central simple 
algebra $\A'$ (note that $\b$ is unchanged) and a
simple character~$\t'$ for this stratum define the same ps-character
precisely when~$\t$ and~$\t'$ are linked by the transfer map
defined in~\cite{VS1}. 
Two ps-characters~$\Theta_1$ and~$\Theta_2$ are
said to be endo-equivalent (see Definition~\ref{Petrone2}) if they can
be characterized by giving realizations 
$([\La,n_i,m_i,\b_i],\t_i)$ in an
$\F$-central simple algebra~$\A$, for $i=1,2$ (note that $\A$ and $\La$
do not depend on $i$), of the same degree and normalized level, and
such that the simple characters~$\t_1$ and~$\t_2$ intertwine in~$\mult\A$. 

\medskip

The properties of endo-equivalence depend on important intertwining
properties of simple characters, notably the preservation of these
properties under the transfer map. This article centres on two
important technical results: the property of ``preservation of
inter\-twi\-ning'' (Theorem~\ref{EndoClasChar}) and the ``intertwining
implies conjugacy'' property (Theorem~\ref{TrelawneyIIC}). Partial
results on these questions were already given by Grabitz~\cite{Gr},
notably a proof of  ``intertwining implies conjugacy'', but these
results are proved under unnecessarily restrictive hypotheses:
that the simple strata underlying the construction are \emph{sound} in
the sense of Definition~\ref{StrateSoundDef}. We have sought to
develop the notion of endo-class in as general a situation as
possible, emphasizing the functorial properties of the objects
in\-vol\-ved. 
However, rather than starting again from scratch, we decided
to use the work of Grabitz as much as possible. 
We note that, as well 
as~\cite{Gr}, our proofs rely heavily on the results of Bushnell,
Henniart and Kutzko~\cite{BK,BH} in the split case.

\medskip

Let us now describe in more detail the results, and the techniques used, 
in this article. 
For~$i=1,2$, let~$\Theta_i$ be a ps-character defined by a
simple stratum $[\La,n_i,m_i,\b_i]$ in an $\F$-cen\-tral simple algebra 
$\A$ and a simple character $\t_i\in\Cc(\La,m_i,\b_i)$ 
attached to this stratum (see paragraph \ref{Dubol} for the notation).
Suppose from now on that the ps-characters~$\Theta_1$ 
and~$\Theta_2$ are endo-equivalent so 
that, in particular, we may assume the characters~$\t_1$ and~$\t_2$ 
inter\-twi\-ne in~$\mult\A$. 
The ``preservation of intertwining'' property can be stated as follows: 

\begin{theon}[see Theorem \ref{EndoClasChar}]
For $i=1,2$, let $[\La',n'_i,m'_i,\b_i^{}]$ be a simple stratum in a simple
central $\F$-algebra $\A'$ and $\t'_i\in\Cc(\La',m'_i,\b_i^{})$ 
defining the ps-character 
$\Theta_i$, that is, $\t'_i$ is the transfer of $\t_i$. 
Then the characters $\t_1'$ and $\t_2'$ intertwine in $\A^{\prime\times}$.
\end{theon}

This means that the property that two simple characters intertwine is 
invariant under transfer. The statement above is the same as its 
analogue~\cite[Theorem 8.7]{BH} in the case that~$\A$ is split and~$\La$ 
is strict. 
However, we will see that the proof requires new ideas. 

\medskip

One of the important results in~\cite{BK} is the ``intertwining implies 
conjugacy'' property for simple characters, which expresses the fact that 
intertwining of simple characters is a very stringent relation. It is this 
property which allows a classification ``up to con\-ju\-ga\-cy'' of the 
irreducible cuspidal representations of~$\GL_n(\F)$.
This property no longer holds in the general case, as was already observed 
in~\cite{BG} for simple strata. 
To remedy the situation, we introduce the notion of \emph{embedding type} 
of a simple stratum (see Definition~\ref{Sidonie}):
two simple strata~$[\La,n_i,m_i,\b_i]$ have the same embedding type if the 
maximal unramified sub\-extensions of~$\F(\b_i)/\F$ are conjugate under the 
normalizer of~$\La$ in~$\mult\A$. 
With the same no\-ta\-tion and hypotheses as above, we prove the following: 

\begin{theon}[see Theorem \ref{TrelawneyIIC}]
Suppose that~$n_1=n_2$,~$m_1=m_2$, and the simple strata~$[\La,n_i,m_i,\b_i]$ 
have the same embedding type. 
Write $\K_i$ for the maximal un\-ramified extension of $\F$ contained in 
$\F(\b_i)\subseteq\A$. 
Then there is an element of the normalizer of~$\La$ in~$\A^{\times}$ which 
simultaneously conjugates~$\K_1$ to~$\K_2$ and~$\t_1$ to~$\t_2$.
\end{theon}

This result was proved by Grabitz~\cite[Corollary 10.15]{Gr} with the 
additional assumption that the sim\-ple strata~$[\La,n,m,\b_i]$ are sound. 
We prove it here without this hypothesis. 

\medskip

Once one has proved that endo-equivalence preserves certain numerical 
invariants (see Lemma~\ref{Hectare1}), it is not hard to see that the proofs 
of these two Theorems can be reduced to the following: 

\begin{theon}[see Theorem \ref{Gata}]
For $i=1,2$, let $[\La',n',m',\b_i^{}]$ be a simple stratum in a simple
central $\F$-algebra $\A'$ and $\t'_i\in\Cc(\La',m',\b_i^{})$ defining 
the ps-character $\Theta_i$, that is, $\t'_i$ is the transfer of $\t_i$. 
Assume the simple strata have the same embedding type and write $\K_i$ 
for the maximal un\-ramified extension of $\F$ contained in 
$\F(\b_i)\subseteq\A'$. 
Then there is an element of the normalizer of~$\La'$ 
in~$\A^{\prime\times}$ which simultaneously conjugates~$\K_1$ to~$\K_2$ 
and~$\t'_1$ to~$\t'_2$. 
\end{theon}

Now let us describe the scheme of the proof. 
We begin with our endo-equivalent ps-cha\-racters $\Theta_1$ and $\Theta_2$, 
together with realizations $([\La,n_i,m_i,\b_i],\t_i)$ in $\A$ such that 
the sim\-ple characters $\t_1$ and $\t_2$ intertwine in $\A^\times$. 
In order to use the results of Grabitz, we need first to produce sound 
realizations of the ps-characters $\Theta_i$ with the same embedding type, 
which intertwine. For sound strata, the embedding type 
is determined by a single integer, the \emph{Fr\"ohlich invariant}, 
which can also be defined for arbitrary strata  (see
Definition~\ref{def:Frinv}).
One can then realize $\Theta_i$ on the lattice sequence $\La\oplus\La$ in such 
a way that the Fr\"ohlich invariant is $1$ and the simple characters still 
intertwine (see Lemma~\ref{Blaise3}). In particular, replacing our original 
realizations of $\Theta_1$ and $\Theta_2$ with these new ones, we can assume 
the simple strata $[\La,n_i,m_i,\b_i]$ have the same Fr\"ohlich invariant. 
Now we define a process:
\begin{equation*}
([\La,n,m,\b],\t)\mapsto([\La^{\ddag},n,m,\b],\t^\ddag)
\end{equation*}
from arbitrary realizations to sound realizations, with $\t^\ddag$ the 
transfer of $\t$, which preserves intertwining and the Fr\"ohlich invariant
(see paragraph \ref{PreDague}). 
In particular, from $\t_1$ and $\t_2$ one obtains simple characters 
$\t^\ddag_1$ and $\t^\ddag_2$ on sound simple strata with the same Fr\"ohlich 
invariant (so same embedding type) which intertwine. 
Thus we can apply Grabitz's results, together with a reduction to the case 
$m_1=m_2$, to deduce that $\t^\ddag_1$ and $\t^\ddag_2$ are conjugate under 
$\A^{\ddag\times}$ (where $\A^\ddag$ is the simple central $\F$-algebra with 
respect to which the stratum $[\La^\ddag,n,m,\b]$ is defined).
Changing again our realizations of $\Theta_1$ and $\Theta_2$ we can suppose we 
have an equality $\t_1=\t_2$ of simple characters. 
This is given in Proposition~\ref{Hengist6}, the culmination of the first 
stage of the proof. 

\medskip

To show that other realizations $\t'_1$ and $\t'_2$ on simple 
strata in $\A'$ with 
the same embedding type are conjugate, we would like to reduce to the split 
case so that we can use results from~\cite{BK,BH}. 
For this we define an \emph{interior lifting} (see section \ref{PILSC}):
\begin{equation*}
([\La,n,m,\b],\t) \mapsto ([\Ga,n,m,\b],\t^\K)
\end{equation*}
relative to $\K/\F$, the maximal unramified subextension of $\F(\b)/\F$, 
where $[\Ga,n,m,\b]$ is a simple stratum in the centralizer $\C$ of $\K$
in the simple central $\F$-algebra $\A$ with 
respect to which $[\La,n,m,\b]$ is defined.
Then we make a \emph{base change} (see section \ref{BCASDFSC}):
\begin{equation*}
([\Ga,n,m,\b],\t^\K) \mapsto ([\overline\Ga,n,m,\b],\overline{\t}{}^\K)
\end{equation*}
relative to $\L/\K$, a finite 
unramified extension which is sufficiently large so 
that the al\-ge\-bra $\C\otimes_\K\L$ is split. 
The definition of the base change used here is somewhat subtle: 
indeed, it is not clear how to make a good definition 
which will preserve intertwining and, when applied to our characters $\t_i$, 
will be independent of $i$. 
Moreover, it is necessary to begin with the interior 
lift or else the base change process would produce
\emph{quasi-simple} characters (see \cite{VS1}), 
rather than simple characters.

\medskip

In order to apply these processes, we note that the maximal unramified 
subextension~$\K$ of $\F(\b_i)/\F$ in $\A$ can be assumed to be 
independent of $i$ since the simple strata have the same embedding type. 
Combining now interior lifting and base change, we get a process:
\begin{equation*}
([\La,n,m,\b],\t)\mapsto([\overline\Ga,n,m,\b],\overline{\t}{}^\K)
\end{equation*}
denoted here $\t\mapsto\widetilde{\t}$ 
for simplicity, which is both injective and 
equivariant, so it is enough to show that 
$\widetilde{\t}'_1$ and $\widetilde{\t}'_2$
are conjugate under $\A^{\prime\times}$.
Now the hypothesis $\t_1=\t_2$ implies 
$\widetilde{\t}_1=\widetilde{\t}_2$ 
(see Propositions~\ref{MauiToho} 
and~\ref{Maui35}), so that the ps-characters 
$\widetilde{\Theta}_1$ and $\widetilde{\Theta}_2$
defined by $\widetilde{\t}_1$ and $\widetilde{\t}_2$ 
are endo-equivalent. 
Moreover, for each $i$, the simple cha\-rac\-ter 
$\widetilde{\t}'_i$ is the transfer of $\widetilde{\t}_i$
(see Theorem~\ref{ResComWithTransAndAEFit}), so 
it is another realization of the ps-character 
$\widetilde\Theta_i$. 
We are now in the split case so, 
modulo a finesse in the case that we do not have strict lattice sequences, 
we deduce from endo-equivalence~\cite{BH} that the characters 
$\widetilde{\t}'_i$ intertwine. 
Thus, from the ``intertwining implies conjugacy'' property~\cite{BK}, 
the characters $\widetilde{\t}'_1$ and $\widetilde{\t}'_2$ are conjugate
under $(\C'\otimes_\K\L)^{\times}$, where $\C'$ denotes 
the centralizer of $\K$ in $\A'$.
Thanks to the invariance property of the base change under the action 
of the Galois group $\Gal(\L/\K)$ (see Proposition~\ref{GalInv}), 
a cohomological argument (see Lemma~\ref{CohoArg}) 
allows us to show that they 
are actually conjugate under $\C^{\prime\times}$.
This completes the proof. 

\section*{Notation}

Let $\F$ be a nonarchimedean locally compact field.
All $\F$-algebras are supposed to be fi\-ni\-te-dimensional 
with a unit. 
By an \textit{$\F$-division algebra} we mean a central 
$\F$-al\-gebra which is a division algebra.

For $\K$ a finite extension of $\F$, or more generally a division algebra over
a finite extension of $\F$, we denote by $\Oo_{\K}$ its ring of integers, 
by $\p_{\K}$ the maximal ideal of $\Oo_{\K}$ and by $\mathfrak{k}_\K$
its residue field.

For $\A$ a simple central algebra over a finite extension $\K$ of $\F$, we 
denote by $\N_{\A/\K}$ and $\tr_{\A/\K}$ respectively the reduced norm and 
trace of $\A$ over $\K$.

For $u$ a real number, we denote by $\lceil{u}\rceil$ the smallest integer 
which is greater than or equal to $u$, and by $\lfloor{u}\rfloor$ the greatest 
integer which is smaller than or equal to $u$, that is, its integer part.

A \textit{character} of a topological group $\G$ is a continuous homomorphism 
from $\G$ to the group $\mathbb{C}^{\times}$ of non-zero complex numbers.

All representations are supposed to be smooth with complex coefficients.

%%%%%%%%%%%%%%%%%%%%%%%%%%%%%%%%%%%%%%%%%%%%%%%%%%%%%%%%%%%%%%%%%%%%%%%%%%%
%%%%%%%%%%%%%%%%%%%%%%%%%%%%%%%%%%%%%%%%%%%%%%%%%%%%%%%%%%%%%%%%%%%%%%%%%%%

\section{Statement of the main results}

In this section, we recall some well known facts about lattice 
sequences, simple strata and simple characters in a simple central 
$\F$-algebra (see \cite{Br1,BK,BK3,VS1,SeSt} for more details),
and we state the main results of this article.

%%%%%%%%%%%%%%%%%%%%%%%%%%%%%%%%%%%%%%%%%%%%%%%%%%%%%%%%%%%%%%%%%%%%%%%%%%%

\subsection{}
\label{Dubol}

Let $\A$ be a simple central $\F$-algebra, and let $\V$ be a simple left 
$\A$-module. 
The algebra $\End_{\A}(\V)$ is an $\F$-division algebra, the opposite of 
which we denote by $\D$.
Considering $\V$ as a right $\D$-vector space, we have a canonical 
isomorphism of $\F$-algebras between $\A$ and $\End_{\D}(\V)$.

\begin{defi}
An {\it $\Oo_{\D}$-lattice sequence} on $\V$ is a sequence 
$\La=(\La_k)_{k\in\ZZ}$ of $\Oo_\D$-lattices of $\V$ such that %we have 
$\La_{k}\supseteq\La_{k+1}$ for all $k\in\ZZ$, and such that there exists 
a positive integer $e$ satisfy\-ing $\La_{k+e}=\La_k\p_\D$ for all
$k\in\ZZ$. 
This integer is called the {\it period} of $\La$ over $\Oo_\D$.
\end{defi}

If $\La_k\supsetneq\La_{k+1}$ for all $k\in\ZZ$, then the lattice
sequence $\La$ is said to be {\it strict}. 

\medskip

Associated with an $\Oo_{\D}$-lattice sequence $\La$ on $\V$, we have an
$\Oo_{\F}$-lattice sequence on $\A$ defined by:
\begin{equation*}
\aa_{k}(\La)=\{a\in\A\ |\ a\La_{i}\subseteq\La_{i+k},\ i\in\ZZ\},
\quad k\in\ZZ.
\end{equation*}
The lattice $\AA(\La)=\aa_0(\La)$ is a hereditary $\Oo_\F$-order in $\A$, 
and $\mathfrak{P}(\La)=\aa_1(\La)$ is its Jacob\-son radical.
They depend only on the set $\{\La_k\ |\ k\in\ZZ\}$.

\medskip

We denote by $\KK(\La)$ the $\mult\A$-{\it normalizer} of $\La$, 
that is the subgroup of $\mult\A$ made of all elements 
$g\in\mult\A$ for which there is an integer $n\in\ZZ$ such that 
$g(\La_{k})=\La_{k+n}$ for all $k\in\ZZ$.
Given  $g\in\KK(\La)$, such an integer 
is unique: it is de\-noted 
$\v_{\La}(g)$ and called the $\La$-{\it valuation} of $g$. 
This defines a group homo\-morphism $\v_{\La}$ from $\KK(\La)$ to $\ZZ$.
Its kernel, denoted $\U(\La)$, is the group of invertible elements 
of $\AA(\La)$. 
We set $\U_0(\La)=\U(\La)$ and, for $k\>1$, we set 
$\U_k(\La)=1+\aa_k(\La)$.

\medskip

Let $\F'$ be a finite extension of $\F$ contained in $\A$.
An $\Oo_\D$-lattice sequence $\La$ on $\V$ is said to be 
{\it $\F'$-pure} if it is norm\-alized by $\F^{\prime\times}$.
The centralizer of $\F'$ in $\A$, denoted $\A'$, is a
sim\-ple central $\F'$-algebra.
We fix a simple left $\A'$-module $\V'$ and write $\D'$ 
for the algebra oppos\-ite to $\End_{\A'}(\V')$.
By \cite[Th\'eor\`eme 1.4]{SeSt} 
(see also \cite[Theorem 1.3]{Br1}), 
given an $\F'$-pure $\Oo_{\D}$-lattice sequence on $\V$, 
there is an $\Oo_{\D'}$-lattice sequence $\La'$ 
on $\V'$ such that:
\begin{equation}
\label{FPrimedescen}
\aa_{k}(\La)\cap\A'=\aa_{k}(\La'), \quad k\in\ZZ.
\end{equation}
It is unique up to translation of indices, and its
$\A^{\prime\times}$-normalizer is $\KK(\La)\cap\A^{\prime\times}$.

\begin{defi}
A \emph{stratum} in $\A$ is a quadruple $[\La,n,m,\b]$ made of an 
$\Oo_\D$-lattice sequence $\La$ on $\V$, 
two integers $m,n$ such that $0\<m\<n-1$ and an element 
$\b\in\aa_{-m}(\La)$.
\end{defi} 

For $i=1,2$, let $[\La,n,m,\b_i]$ be a stratum in $\A$.
We say these two strata are \emph{equivalent}
if $\b_2-\b_1\in\aa_{-m}(\La)$. 

\medskip

Given a stratum $[\La,n,m,\b]$ in $\A$, we denote by $\E$ the 
$\F$-algebra generated by $\b$. 
This stratum is said to be {\it pure} if $\E$ is a field, 
if $\La$ is $\E$-pure and if $\v_{\La}(\b)=-n$.
In this situation, we denote by:
\begin{equation*}
e_{\b}(\La)
\end{equation*}
the period of $\La$ as an $\Oo_{\E}$-lattice sequence. 
Given a pure stratum $[\La,n,m,\b]$, we denote by $\B$ the centralizer 
of $\E$ in $\A$.
For $k\in\ZZ$, we set:
\begin{equation*}
\mathfrak{n}_k(\b,\La)=\{x\in\AA(\La)\ |\ \b x-x\b\in\aa_k(\La)\}.
\end{equation*}
The smallest integer $k\>\v_{\La}(\b)$ such that $\mathfrak{n}_{k+1}(\b,\La)$ 
is contained in $\AA(\La)\cap\B+\PP(\La)$ is called the 
\emph{critical exponent} of the stratum $[\La,n,m,\b]$, denoted 
$k_0(\b,\La)$.

\begin{defi}
\label{stratepure}
The stratum $[\La,n,m,\b]$ is said to be \emph{simple} if it is pure
and if we have $m\<-k_0(\b,\La)-1$. 
\end{defi}

Let $[\La,n,m,\b]$ be a simple stratum in $\A$. 
In \cite{SeSt} (see paragraph 2.4), one attaches to this simple stratum
a compact open subgroup 
$\H^{m+1}(\b,\La)$ of $\mult\A$ and a finite set $\Cc(\La,m,\b)$ 
of characters of $\H^{m+1}(\b,\La)$, called simple 
characters of level $m$, depending on the choice of
an additive character:
\begin{equation}
\label{AddCharJ}
\psi:\F\to\mathbb{C}^{\times}
\end{equation}
which is trivial on $\p_{\F}$ but not on $\Oo_{\F}$, and which will 
be fixed once and for all throughout this paper.
If $\lfloor n/2\rfloor\<m$, then $\H^{m+1}(\b,\La)=\U_{m+1}(\La)$, 
and the set $\Cc(\La,m,\b)$ reduces to a single character 
$\psi_\b^{\A}$ of $\U_{m+1}(\La)$ defined by:
\begin{equation}
\label{LesLapinsNAimentPasLesCarottes}
\psi_\b^{\A}:x\mapsto\psi\circ\tr_{\A/\F}(\b(x-1)),
\end{equation}
which depends only on the equivalence class of $[\La,n,m,\b]$.
More generally, for any possible value of $m$, the subgroup
$\H^{m+1}(\b,\La)$ and the set $\Cc(\La,m,\b)$ depend only on the
equivalence class of $[\La,n,m,\b]$.

%%%%%%%%%%%%%%%%%%%%%%%%%%%%%%%%%%%%%%%%%%%%%%%%%%%%%%%%%%%%%%%%%%%%%%%%%%%

\subsection{}
\label{Jessica}

Let $\b$ be a non-zero element of some finite extension of $\F$. 
We set $\E=\F(\b)$ and:
\begin{align*}
\label{Pitot}
n_{\F}(\b)&=-\v_{\E}(\b),\\
e_{\F}(\b)&=e(\E:\F),\\
f_{\F}(\b)&=f(\E:\F),
\end{align*}
where $e(\E:\F)$ and $f(\E:\F)$ stand for the ramification index 
and the residue class degree of $\E$ over $\F$ respectively, and 
$\v_{\E}$ for the valuation map of the field $\E$ 
giving the value $1$ to any uniformizer of $\E$.
The lattice sequence $i\mapsto\p_{\E}^{i}$, denoted $\La(\E)$, 
is the unique (up to transl\-ation) $\E$-pure strict 
$\Oo_{\F}$-lattice sequence on the $\F$-vector space $\E$,
and its valuation map coincide with $\v_{\E}$ on $\mult\E$.
To any integer $0\<k\<n_{\F}(\b)-1$ we can attach the pure stratum
$[\La(\E),n_{\F}(\b),k,\b]$ of the split $\F$-algebra
$\A(\E)=\End_{\F}(\E)$, the critical exp\-onent of which we denote by:
\begin{equation*}
\label{Potit}
k_{\F}(\b)=k_0(\b,\La(\E)).
\end{equation*}
This is an integer greater than or equal to $-n_{\F}(\b)$.
Let us recall the definition of a simple pair over $\F$ 
(see \cite[Definition 1.5]{BH}).

\begin{defi}
A \textit{simple pair} over $\F$ is a pair $(k,\b)$ consisting of a
non-zero element $\b$ of some finite extension of $\F$ and an integer
$0\<k\<-k_{\F}(\b)-1$. 
\end{defi}

Associated with a simple pair $(k,\b)$ over $\F$ is the simple stratum 
$[\La(\E),n_{\F}(\b),k,\b]$ in $\A(\E)$
together with a compact open subgroup of $\A(\E)^{\times}$ and a
set of simple char\-ac\-ters: 
\begin{equation*}
\H^{k+1}_{\F}(\b)=\H^{k+1}(\b,\La(\E)),
\quad
\Cc_{\F}(k,\b)=\Cc(\La(\E),k,\b).
\end{equation*}
Now let $\A$ be a simple central 
$\F$-algebra and $\V$ be a simple left $\A$-module.
A \textit{realization} of the simple pair $(k,\b)$ in $\A$
is a stratum in $\A$ of the form $[\La,n,m,\h(\b)]$ made of:
\begin{enumerate}
\item
a homomorphism $\h$ of $\F$-algebra from $\F(\b)$ to $\A$;
\item
an $\Oo_{\D}$-lattice sequence $\La$ on $\V$ 
normalized by the image of $\F(\b)^{\times}$ 
under $\h$;
\item
an integer $m$ such that
$\left\lfloor m/e_{\h(\b)}(\La)\right\rfloor=k$.
\end{enumerate}
The integer $-n$ is then the $\La$-valuation of $\h(\b)$.
By \cite[Proposition 2.25]{VS1} we have:
\begin{equation}
\label{Lakme}
k_0(\h(\b),\La)=e_{\h(\b)}(\La)k_{\F}(\b),
\end{equation}
which implies that any realization of a simple pair is a simple stratum.
According to \cite{VS1} again ({\it ibid.}, paragraph 3.3), 
for such a realization there is a canonical bijective map:
\begin{equation}
\label{JungerJournauxDeGuerre}
\boldsymbol{\tau}_{\La,m,\h}:\Cc_{\F}(k,\b)\to\Cc(\La,m,\h(\b))
\end{equation}
called the {\it transfer} map.
Some of its properties have been studied in \cite{SeSt} and some further
properties will be given in sections \ref{ILT} and \ref{BCASDFSC}
of the present article.
Given another realization $[\La',n',m',\h'(\b)]$ of the pair $(k,\b)$
in some simple central $\F$-algebra $\A'$, we have a transfer map from 
$\Cc(\La,m,\h(\b))$ to $\Cc(\La',m',\h'(\b))$ by composing
$\boldsymbol{\tau}_{\La,m,\h}^{-1}$ with 
$\boldsymbol{\tau}_{\La',m',\h'}^{}$.

\medskip

Given a simple pair $(k,\b)$ over $\F$, we denote by 
$\boldsymbol{\EuScript{C}}_{(k,\b)}$ the set of pairs 
$([\La,n,m,\h(\b)],\t)$ made of a realization $[\La,n,m,\h(\b)]$ 
of $(k,\b)$
in a simple central $\F$-algebra and a simple character
$\t\in\Cc(\La,m,\h(\b))$.
Hence the surjective map:
\begin{equation*}
([\La,n,m,\h(\b)],\t)\mapsto
\boldsymbol{\tau}_{\La,m,\h}^{-1}(\t)\in\Cc_{\F}(k,\b)
\end{equation*}
is well defined on $\boldsymbol{\EuScript{C}}_{(k,\b)}$ and 
induces, by its fibers, an equi\-valence relation on it.

\begin{defi}
\label{DefPSC}
A \textit{potential simple character} over $\F$ (or {\it ps-character}
for short) is a triple $(\Theta,k,\b)$ made of a simple pair $(k,\b)$
over $\F$ and an equivalence class $\Theta$ in 
$\boldsymbol{\EuScript{C}}_{(k,\b)}$. 
\end{defi}

When the context is clear, we will often denote by $\Theta$ the 
ps-character $(\Theta,k,\b)$.
Given a realization $[\La,n,m,\h(\b)]$ of $(k,\b)$, we will denote 
by $\Theta(\La,m,\h)$ 
the simple character $\t$ such that the pair 
$([\La,n,m,\h(\b)],\t)$ belongs to $\Theta$.

%%%%%%%%%%%%%%%%%%%%%%%%%%%%%%%%%%%%%%%%%%%%%%%%%%%%%%%%%%%%%%%%%%%%%%%%%%%

\subsection{}
\label{MainStreet}

We now state the main results which are proved in this article.
Our first task is to extend the notion of endo-equivalence of 
simple pairs developed by Bushnell and Henniart in \cite{BH}.
More precisely, we extend it to realizations in non-necessarily 
split simple central $\F$-algebras with non-necessarily strict 
lattice sequences.

\begin{defi}
\label{Petrone}
For $i=1,2$, let $(k_{i},\b_{i})$ be a simple pair over $\F$.
We say that these pairs are \textit{endo-equivalent}, denoted:
\begin{equation*}
(k_{1},\b_1)\thickapprox(k_{2},\b_2),
\end{equation*}
if $k_1=k_2$ and $[\F(\b_1):\F]=[\F(\b_2):\F]$, and if
there exists a simple central $\F$-algebra $\A$ together with 
realizations $[\La,n_i,m_i,\h_i(\b_i)]$ of $(k_i,\b_i)$ 
in $\A$, with $i=1,2$, which intertwine in $\A$.
\end{defi}

Recall that two strata $[\La,n_i,m_i,\b_i]$ in $\A$, 
with $i\in\{1,2\}$, \emph{intertwine} in $\A$ if there 
exists $g\in\mult\A$ such that: 
\begin{equation}
\label{Cinq-Mars}
(\b_1+\aa_{-m_1}(\La))\cap g(\b_2+\aa_{-m_2}(\La))g^{-1}\neq\varnothing.
\end{equation}

As we will see in paragraph \ref{Dague} (see Corollary \ref{Epaminondas}), 
this definition of 
endo-equivalence of simple pairs is equivalent to \cite[Definition 1.14]{BH}, 
although more general in appearance. 

\medskip

We now investigate the intertwining relations among various
realizations of given sim\-ple pairs, and in particular their 
preservation properties.
Our first result is the following pro\-position, which generalizes 
\cite[Proposition 1.10]{BH} and is proved in paragraph \ref{ProvaDAmore}.

\begin{prop}
\label{SimonLeMagicienStrong}
For $i=1,2$, let $(k,\b_{i})$ be a simple pair over $\F$, and suppose 
these pairs are endo-equivalent.
Let $\A$ be a simple central $\F$-algebra and let 
$[\La,n_i,m_i,\h_i(\b_i)]$ be a real\-ization of 
$(k,\b_i)$ in $\A$, for $i=1,2$.
These strata then intertwine in $\A$.
\end{prop}

Broussous and Grabitz remarked in \cite{BG} that two simple strata 
$[\La,n,m,\b_{i}]$, $i=1,2$, in $\A$ which intertwine in $\A$ may 
be not conjugate under $\mult\A$, unlike the case where $\A$ is split
(see \cite[Theorem 2.6.1]{BK} for the case where $\A$ is split and 
$\La$ is strict).
In order to remedy this, they introduced the notion of an embedding 
type (see also Fr\"ohlich \cite{Fr}).
Here we extend this notion to non-necessarily strict lattice sequences.

\medskip

We fix a simple central $\F$-algebra $\A$ and a simple left $\A$-module 
$\V$ as in paragraph \ref{Dubol}.
Associated with it, we have an $\F$-division algebra $\D$.
An {\it embedding} in $\A$ is a pair $(\E,\La)$ made of a finite extension 
$\E$ of $\F$ contained in $\A$ and an $\E$-pure $\Oo_\D$-lattice sequence 
$\La$ on $\V$.
Given such a pair, we denote by $\E^{\flit}$ the maximal 
finite unramified extension of $\F$ which is contained in $\E$ and 
whose degree divides the reduced degree of $\D$ over $\F$.

\medskip

Two em\-bed\-dings $(\E_i,\La_i)$, $i=1,2$, in $\A$ are said 
to be {\it equivalent} in $\A$ if there exists an element 
$g\in\mult\A$ such 
that $\La_1$ is in the translation class of $g\La_2$ and 
$\E_1^{\flit}=g\E_2^{\flit}g^{-1}$.
This defines an equivalence relation on the set of embeddings in 
$\A$, and an equivalence class for this relation is called an 
{\it embedding type} in $\A$.

\begin{defi}
\label{Sidonie}
The {\it embedding type} of a pure stratum $[\La,n,m,\b]$ is the 
embedding type of the pair $(\F(\b),\La)$ in $\A$.
\end{defi}

This allows us to state the following ``intertwining implies 
conjugacy'' theorem, which generalizes \cite[Theorem 2.6.1]{BK} and 
\cite[Proposition 4.1.2]{BG} and is proved in paragraph \ref{Barnabooth}. 

\begin{prop}
\label{IreneeDeLyon}
For $i=1,2$, let $[\La,n,m,\b_i]$ be a simple stratum in $\A$.
Assume that they intertwine in $\A$ and have the same embedding type.
Write $\K_i$ for the maximal un\-ramified extension of $\F$ contained
in $\F(\b_i)$.
Then there is $u\in\KK(\La)$ such that $\K_1=u\K_2u^{-1}$ and 
$\b_1-u\b_2u^{-1}\in\aa_{-m}(\La)$.
\end{prop}

%%%%%%%%%%%%%%%%%%%%%%%%%%%%%%%%%%%%%%%%%%%%%%%%%%%%%%%%%%%%%%%%%%%%%%%%%%%

\subsection{}
\label{MainStreet2}

We now extend the notion of endo-equivalence of simple 
characters developed by Bushnell and Henniart in \cite{BH}.
As for simple pairs, we extend it to realizations in 
non-ne\-cessarily split 
sim\-ple central $\F$-algebras with non-necessarily strict lattice 
sequences.

\begin{defi}
\label{Petrone2}
For $i=1,2$, let $(\Theta_{i},k_{i},\b_{i})$ be a ps-character over $\F$.
We say that these ps-characters are \textit{endo-equivalent}, denoted:
\begin{equation*}
\Theta_1\thickapprox\Theta_2,
\end{equation*}
if $k_1=k_2$ and 
$[\F(\b_1):\F]=[\F(\b_2):\F]$, and if there exists a 
simple central $\F$-algebra $\A$ together with realizations
$[\La,n_i,m_i,\h_i(\b_i)]$ of $(k_i,\b_i)$ in $\A$, 
with $i=1,2$, such that the simple characters 
$\Theta_1(\La,m_1,\h_1)$ and $\Theta_2(\La,m_2,\h_2)$ inter\-twine in 
$\mult\A$.
\end{defi}

Recall that two simple characters $\t_i\in\Cc(\La,m_i,\b_i)$, $i=1,2$, 
\emph{intertwine} in $\mult\A$ if there exists $g\in\mult\A$ such that: 
\begin{equation}
\label{Eaque}
\t_2(x)=\t_1(gxg^{-1}), 
\quad
x\in\H^{m_2+1}(\b_2,\La)\cap g^{-1}\H^{m_1+1}(\b_1,\La)g.
\end{equation}

As we will see at the end of this article 
(see Corollary \ref{Vendries2}), 
this definition of endo-equi\-valence of simple characters is 
equivalent to \cite[Definition 8.6]{BH}.

\medskip

We now state the main results of this article concerning properties of 
simple characters with respect to intertwining and conjugacy.
The following generalizes \cite[Theorem 8.7]{BH}. 

\begin{theo}
\label{EndoClasChar}
For $i=1,2$, let $(\Theta_{i},k_{i},\b_{i})$ be a ps-character over $\F$, 
and suppose that $\Theta_{1}\thickapprox\Theta_{2}$. 
Let $\A$ be a simple central $\F$-algebra and let 
$[\La,n_i,m_i,\h_i(\b_i)]$ be realizations of 
$(k_{i},\b_i)$ in $\A$, for $i=1,2$. 
Then $\Theta_1(\La,m_1,\h_1)$ and 
$\Theta_2(\La,m_2,\h_2)$ intertwine in $\A^{\times}$.
\end{theo}

The following ``intertwining implies conjugacy'' theorem 
for simple characters generalizes \cite[Theorem 3.5.11]{BK} and 
\cite[Corollary 10.15]{Gr} to simple characters in non-necessarily split 
simple central $\F$-algebras with non-necessarily strict lattice sequences. 

\begin{theo}
\label{TrelawneyIIC}
Let $\A$ be a simple central $\F$-algebra.
For $i=1,2$, let $[\La,n,m,\b_i]$ be a simple stratum in $\A$, 
and let $\t_{i}\in\Cc(\La,m,\b_i)$ be a simple character.
Write $\K_i$ for the maximal un\-ramified extension of $\F$ contained
in $\F(\b_i)$.
Assume that $\t_1$ and $\t_2$ inter\-twine in $\mult\A$ and that the strata 
$[\La,n,m,\b_i]$ have the same embedding type.
Then there is an element $u\in\KK(\La)$ such that: 
\begin{enumerate}
\item 
$\K_1=u\K_2u^{-1}$; 
\item
$\Cc(\La,m,\b_1)=\Cc(\La,m,u\b_2u^{-1})$;
\item 
$\t_2(x)=\t_1(uxu^{-1})$, 
for all $x\in\H^{m+1}(\b_2,\La)=u^{-1}\H^{m+1}(\b_1,\La)u$.
\end{enumerate}
\end{theo}

We will see in section \ref{Sec4} (see Corollary \ref{Psaume}) 
that the proofs of these 
two theorems can be reduced to that of the following statement, 
which will be proved in section \ref{Sec8}.

\begin{theo}
\label{Gata}
For $i=1,2$, let $(\Theta_{i},k_{i},\b_{i})$ be a ps-character over $\F$, 
and suppose that $\Theta_{1}\thickapprox\Theta_{2}$. 
Let $\A$ be a simple central $\F$-algebra, and let 
$[\La,n,m,\h_i(\b_i)]$ be realizations of 
$(k_{i},\b_i)$ in $\A$, for $i=1,2$. 
Write $\K_i$ for the maximal un\-ramified extension of $\F$ 
contained in $\F(\b_i)$ and $\t_i$ for the simple character 
$\Theta_i(\La,m,\h_i)$.
Assume these strata have the same embedding type. 
Then there is an element $u\in\KK(\La)$ such that: 
\begin{enumerate}
\item 
$\h_1(\K_1)=u\h_2(\K_2)u^{-1}$; 
\item
$\Cc(\La,m,\h_1(\b_1))=\Cc(\La,m,u\h_2(\b_2)u^{-1})$;
\item 
$\t_2(x)=\t_1(uxu^{-1})$, 
for all $x\in\H^{m+1}(\h_2(\b_2),\La)=u^{-1}\H^{m+1}(\h_1(\b_1),\La)u$.
\end{enumerate}
\end{theo}

The main ingredient in this reduction step is Lemma \ref{Hectare1}, 
which states that the endo-equivalence relation preserves certain 
numerical invariants attached to a ps-character.

%%%%%%%%%%%%%%%%%%%%%%%%%%%%%%%%%%%%%%%%%%%%%%%%%%%%%%%%%%%%%%%%%%%%%%%%%%%

\subsection{}

As has been explained in the introduction, 
this article makes a large use of the results of 
Bushnell, Henniart and Kutzko in the split case 
\cite{BH,BK} (see paragraphs \ref{MainStreet} and 
\ref{MainStreet2}), 
as well as results of Grabitz \cite{Gr} which are 
based on the following definition. 

\begin{defi}
\label{StrateSoundDef}
A simple stratum $[\La,n,m,\b]$ in $\A$ is {\it sound} 
if $\La$ is strict, $\AA\cap\B$ is principal and 
$\KK(\AA)\cap\mult\B=\KK(\AA\cap\B)$, where $\AA$ is 
the hereditary $\Oo_\F$-order defined by $\La$.
\end{defi}

More generally, an embedding $(\E,\La)$ in $\A$ is {\it sound} 
if the conditions of Definition \ref{StrateSoundDef}
are fulfilled with $\B$ the centralizer of $\E$ in $\A$. 

\begin{rema}
Note that the condition on $\AA\cap\B$ forces $\AA$ to 
be a principal $\Oo_\F$-order.
In the split case, a simple stratum $[\La,n,m,\b]$ is sound 
if and only if $\La$ is strict and $\AA$ is principal.
\end{rema}

When $\La$ is strict, its translation class is entirely determined
by the hereditary $\Oo_\F$-order $\AA=\AA(\La)$.
In this case, we will sometimes write $(\E,\AA)$ and $[\AA,n,m,\b]$
rather than $(\E,\La)$ and $[\La,n,m,\b]$.

\medskip

In the case where 
the simple strata $[\La,n,m,\b_i]$, $i=1,2$, are sound, 
Grabitz has proved in \cite{Gr} the ``intertwining implies 
conjugacy'' theorem for simple characters 
(see \textit{ibid.}, Theorem 10.3 and Corollary 10.15). 
More precisely, he has proved the following result. 

\medskip

Given $\K/\F$ an un\-ramified extension contained in $\A$,
a sound simple stratum $[\La,n,m,\b]$ in $\A$ is 
$\K$-{\it special} (see \cite[Definition 3.1]{Gr})
if it is $\K$-pure in the sense of Definition 
\ref{KPure} and if $(\K(\b),\AA(\La)\cap\C)$ is a sound embedding 
in $\C$, where $\C$ is the centralizer of $\K$ in $\A$. 

\begin{theo}[\cite{Gr}, Theorem 10.3]
\label{Grabinoulor}
For $i=1,2$, let $[\La,n,m,\b_i]$ be a sound simple stratum in a 
simple central $\F$-algebra $\A$ and let $\t_{i}\in\Cc(\La,m,\b_i)$ 
be a simple character.
Let $f$ be a multiple of the greatest common divisor of $f_\F(\b_1)$ 
and $f_\F(\b_2)$, and let $\K_i$ be an un\-ramified extension of $\F$
of degree $f$ contained in $\A$ such that 
$[\La,n,m,\b_i]$ is $\K_i$-special. 
Assume $(\K_1,\La)$ and $(\K_2,\La)$ are equivalent embeddings 
in $\A$, and that $\t_1$ and $\t_2$ inter\-twine in $\mult\A$.
Then:
\begin{enumerate}
\item 
$e_\F(\b_1)=e_\F(\b_2)$ and $f_\F(\b_1)=f_\F(\b_2)$;
\item
$\K_i$ contains the maximal un\-ramified extension of $\F$ contained
in $\F[\b_i]$.
\end{enumerate}
Moreover, there exists $u\in\KK(\La)$ such that: 
\begin{enumerate}
\item[(3)]
$\K_1=u\K_2u^{-1}$; 
\item[(4)]
$\Cc(\La,m,\b_1)=\Cc(\La,m,u\b_2u^{-1})$;
\item[(5)]
$\t_2(x)=\t_1(uxu^{-1})$, 
for all $x\in\H^{m+1}(\b_2,\La)=u^{-1}\H^{m+1}(\b_1,\La)u$.
\end{enumerate}
\end{theo}

We will also need the following result.

\begin{prop}[\cite{Gr}, Propositions 9.1 and 9.9]
\label{Gr9199}
For $i=1,2$, let $[\La,n_,m_,\b_{i}]$ be a sound simple stratum
in $\A$.
Assume that $\Cc(\La,m,\b_1)\cap\Cc(\La,m,\b_2)$ is not empty.
Then $e_\F(\b_1)=e_\F(\b_2)$, $f_\F(\b_1)=f_\F(\b_2)$, 
$k_\F(\b_1)=k_\F(\b_2)$
and $\Cc(\La,m,\b_{1})=\Cc(\La,m,\b_{2})$.
\end{prop}

Note that \cite[Proposition 9.1]{Gr} gives us an equality between 
$[\F(\b_1):\F]$ and $[\F(\b_2):\F]$, but the two finer equalities 
between the ramification indexes and residue class degrees come 
from Theorem \ref{Grabinoulor}.

\medskip 

Our proof of Theorem \ref{Gata} in section \ref{Sec8}
is decomposed into two steps. 
The first step con\-sists of treating the case where the extensions 
$\F(\b_i)/\F$ are totally ramified, and the second step consists 
of reducing to the totally ramified case.
In section \ref{PILSC}, we develop an interior 
lifting process for simple strata and simple characters with respect 
to a finite un\-ra\-mified extension $\K$ of $\F$, in a way similar 
to \cite{BH} and \cite{Gr}.
Its compatibility with transfer is explored in section \ref{ILT}. 
This interior lifting is enough to reduce to the totally ramified case. 
The totally ramified case is more subtle.
For this we develop an `exterior lifting' or un\-ramified base change
in section \ref{BCASDFSC}. 

%%%%%%%%%%%%%%%%%%%%%%%%%%%%%%%%%%%%%%%%%%%%%%%%%%%%%%%%%%%%%%%%%%%%%%%%%%%
%%%%%%%%%%%%%%%%%%%%%%%%%%%%%%%%%%%%%%%%%%%%%%%%%%%%%%%%%%%%%%%%%%%%%%%%%%%

\section{Realizations and intertwining for simple strata}
\label{SSa}

In this section, we introduce various constructions 
which will be used throughout the paper.
More precisely, we describe various processes, 
preserving intertwining, which associate to a 
realization of a simple pair in some simple central $\F$-algebra a 
realization in a (possibly different) simple central 
$\F$-algebra, with additional properties.
This allows us to prove that Definition \ref{Petrone} 
and the definition of endo-equivalence of simple pairs 
given in \cite{BH} 
are equivalent (see Corollary \ref{Epaminondas}), and 
to prove Proposition \ref{SimonLeMagicienStrong}.

%%%%%%%%%%%%%%%%%%%%%%%%%%%%%%%%%%%%%%%%%%%%%%%%%%%%%%%%%%%%%%%%%%%%%%%%%%%

\subsection{}
\label{Split}

We fix a simple central $\F$-algebra $\A$ and a simple left $\A$-module 
$\V$.
We set: 
\begin{equation*}
\widetilde\A=\End_{\F}(\V),
\end{equation*}
which is a split simple central $\F$-algebra in which the 
algebra $\A$ embeds naturally.
To any stratum $[\La,n,m,\b]$ in $\A$ we can attach a stratum
$[\widetilde\La,n,m,\b]$ in $\widetilde\A$, where $\widetilde\La$ denotes 
the $\Oo_\F$-lattice sequence defined by $\La$.
By \cite[Th\'eor\`eme 2.23]{VS1},
this latter stratum is simple if and only the first one is, and in this 
case they are realizations of the same simple pair over $\F$.
More\-over, we have the following result.

\begin{prop}
\label{LongJohnSilverTildeVariante}
For $i=1,2$, let $[\La,n_i,m_i,\b_i]$ be a simple stratum in $\A$.
Assume they intertwine in $\A$.
Then the strata $[\widetilde\La,n_i,m_i,\b_i]$ intertwine in 
$\widetilde{\A}$. 
\end{prop}

\begin{proof}
This follows immediately from the definition of intertwining and the fact 
that the $\Oo_\F$-module $\PP_k(\La)$ is contained in $\PP_k(\widetilde\La)$ 
for all $k\in\ZZ$. 
\end{proof}

%%%%%%%%%%%%%%%%%%%%%%%%%%%%%%%%%%%%%%%%%%%%%%%%%%%%%%%%%%%%%%%%%%%%%%%%%%%

\subsection{}
\label{DerivedStratum}

Let $[\La,n,m,\b]$ be a simple stratum in $\A$, which is a realization
of a simple pair $(k,\b)$ over $\F$.
The {\it affine class} of $\La$ is the set of all $\Oo_{\D}$-lattice 
sequences on $\V$ of the form:
\begin{equation}
\label{AffCla}
a\La+b:k\mapsto\La_{\lceil(k-b)/a\rceil},
\end{equation}
with $a,b\in\ZZ$ and $a\>1$.
The period of (\ref{AffCla}) is 
$a$ times the period $e(\La)$ of $\La$. 
Given an integer $l\>1$, we set $\V'=\V\oplus\dots\oplus\V$ ($l$ times)
and $\A'=\End_{\D}(\V')$, and 
embed $\A$ in $\A'$ diagonally.
For each $j\in\{1,\dots,l\}$, we choose a lattice sequence $\La^{j}$ 
in the 
affine class of $\La$, and assume the periods of the $\La^{j}$'s 
are all equal to a common integer $ae(\La)$ with $a\>1$.
We now form the $\Oo_{\D}$-lattice sequence $\La'$ on $\V'$ 
defined by: 
\begin{equation}
\label{SpadassinVariante}
\La'=\La^{1}\oplus\La^{2}\oplus\cdots\oplus\La^{l},
\end{equation}
and fix a non-negative integer $m'$ such that:
\begin{equation*}
\label{Seide}
\lfloor{m'/a}\rfloor=m.
\end{equation*}
If we set $n'=an$, this gives us a simple stratum $[\La',n',m',\b]$ 
in $\A'$, which is a realization of the simple pair $(k,\b)$.
In the particular case where $l=1$, we have the following result.

\begin{lemm}
\label{Bachir}
Assume that $l=1$, so that $\La'$ is in the affine class of $\La$.
Then we have $\H^{m'+1}(\b,\La')=\H^{m+1}(\b,\La)$ and 
$\Cc(\La',m',\b)=\Cc(\La,m,\b)$.
Moreover, the transfer map from $\Cc(\La',m',\b)$ to $\Cc(\La,m,\b)$ 
is the identity map.
\end{lemm}

\begin{proof}
The first assertion is straightforward by induction on $\b$.
For the second one, see \cite[Th\'eor\`eme 2.13]{SeSt}.
\end{proof}

%%%%%%%%%%%%%%%%%%%%%%%%%%%%%%%%%%%%%%%%%%%%%%%%%%%%%%%%%%%%%%%%%%%%%%%%%%%

\subsection{}
\label{Ecarlate}

Assume now we are given two simple strata $[\La,n_i,m_i,\b_i]$, 
$i=1,2$, in $\A$.
For each $i$, we set $n_i'=an^{}_i$ and fix a non-negative 
integer $m_i'$ such that $\lfloor{m_i'/a}\rfloor=m^{}_i$, so that we
have a simple stratum $[\La',n_i',m_i',\b^{}_i]$ in ${\A}'$.

\begin{prop}
\label{LongJohnSilverGeneralVariante}
Assume that the strata $[\La,n_i,m_i,\b_i]$, $i\in\{1,2\}$, 
intertwine in $\A$.
Then the strata $[\La',n_i',m_i',\b^{}_i]$, $i\in\{1,2\}$, 
intertwine in $\A'$.
\end{prop}

\begin{proof}
We start with an element $g\in\A^{\times}$ which intertwines the two
strata $[\La,n_i,m_i,\b_i]$, that is, which satisfies the condition 
(\ref{Cinq-Mars}),
and we let $\iota$ denote the diagonal embedding of $\A$ in $\A'$
(which we omit from the notation when the context is clear).
For $j\in\{1,\dots,l\}$, write $\V^{j}$ for the $j$th copy 
of $\V$ in $\V'=\V\oplus\dots\oplus\V$.
Then for each $i$, we have:
\begin{equation*}
\PP_{-m_i'}(\La')\cap\End_{\D}(\V^j)=\PP_{-m_i'}(\La^{j}),
\quad
j\in\{1,\dots,l\},
\end{equation*}
which is equal to $\PP_{-m_i}(\La)$
as can be seen by a direct computation in the case $l=1$.
This implies that $\iota$ induces an $\Oo_\F$-module embedding of 
$\PP_{-m_i}(\La)$ in $\PP_{-m_i'}(\La')$, from which we deduce that 
$g'=\iota(g)\in\A^{\prime\times}$ intertwines the strata 
$[\La',n_i',m_{i}',\b^{}_i]$.
\end{proof}

\begin{rema}
\label{blioblieris}
Note that $\iota$ induces a group homomorphism of $\KK(\La)$ into 
$\KK(\La')$.  
Therefore, if $g\in\KK(\La)$ intertwines two simple strata 
$[\La,n,m,\b_i]$, $i=1,2$, that is, if we have:
\begin{equation*}
\b_2-g\b_1g^{-1}\in\PP_{-m}(\La),
\end{equation*}
and if we set $n'=an$ and fix a non-negative integer $m'$ such that 
$\lfloor{m'/a}\rfloor=m$, then the element 
$\iota(g)\in\KK(\La')$ intertwines the strata $[\La',n',m',\b_i]$.
\end{rema}

\begin{prop}
\label{JeromeBauche2}
Assume that the strata $[\La,n_i,m_i,\b_i]$, for $i\in\{1,2\}$, 
have the same em\-bed\-ding type. 
Then the strata $[\La',n_i',m_i',\b^{}_i]$, $i\in\{1,2\}$, have the
same embedding type.
\end{prop}

\begin{proof}
Given $g\in\KK(\La)$ which conjugates 
the unramified extensions $\F(\b_i)^{\flit}$, $i\in\{1,2\}$, 
the element $\iota(g)\in\KK(\La')$ conjugates the extensions 
$\F(\iota(\b_i))^{\flit}$, $i\in\{1,2\}$.
\end{proof}

%%%%%%%%%%%%%%%%%%%%%%%%%%%%%%%%%%%%%%%%%%%%%%%%%%%%%%%%%%%%%%%%%%%%%%%%%%%

\subsection{}
\label{Miskatonic}

For $i=1,2$, let $\t_i$ be a simple character in 
$\Cc(\La,m_i,\b_i)$, and let $\t_i'$ be 
its trans\-fer in $\Cc(\La',m_i',\b^{}_i)$. 
The following result is an analogue of Proposition 
\ref{LongJohnSilverGeneralVariante} for simple characters. 

\begin{prop}
\label{LongJohnSilverGeneralCarSimVariante}
Assume that $\t_1$ and $\t_2$ intertwine in $\A^{\times}$.
Then $\t_1'$ and $\t_2'$ intertwine in $\A^{\prime\times}$.
\end{prop}

\begin{proof}
The decomposition of $\V'$ into a sum of copies of $\V$
defines a Levi subgroup:
\begin{equation}
\label{Arendt}
\M=\A^{\times}\times\dots\times\A^{\times}
\end{equation}
of $\A^{\prime\times}$. 
We fix a parabolic subgroup $\P$ of $\A^{\times}$ with Levi factor $\M$ 
and unipotent radical $\N$, and we write $\N^{-}$ for the unipotent radical 
of the parabolic subgroup of $\A^{\times}$ opposite to $\P$ with respect to 
$\M$. 
According to \cite[Th\'eor\`eme 2.17]{SeSt}, we have an Iwahori
decomposition: 
\begin{eqnarray*}
\H^{m_{i}'+1}(\b_i,\La^{\prime})&=&
\big(\H^{m_{i}'+1}(\b_i,\La^{\prime})\cap\N^{-}\big)%\cdot
\big(\H^{m_{i}'+1}(\b_i,\La^{\prime})\cap\M\big)%\cdot
\big(\H^{m_{i}'+1}(\b_i,\La^{\prime})\cap\N\big),\\
\H^{m_{i}'+1}(\b_i,\La^{\prime})\cap\M&=&
\H^{m_{i}+1}(\b_i,\La)\times\dots\times\H^{m_{i}+1}(\b_i,\La),
\end{eqnarray*}
for each integer $i=1,2$. 
We have the following result.

\begin{lemm}
\label{JacquesVerges}
The simple character $\t^{\prime}_{i}$ is trivial 
on the subgroups $\H^{m_{i}'+1}(\b^{}_i,\La^{\prime})\cap\N$ and
$\H^{m_{i}'+1}(\b^{}_i,\La^{\prime})\cap\N^{-}$, and we have:
\begin{equation*}
\t^{\prime}_i\ |\ \H^{m_{i}'+1}(\b^{}_i,\La^{\prime})\cap\M=
\t_i\otimes\dots\otimes\t_i.
\end{equation*}
\end{lemm}

\begin{proof}
This derives from \cite[Th\'eor\`eme 2.17]{SeSt}.
Indeed, for $j\in\{1,\dots,l\}$, the restriction of 
$\t^{\prime}_i$ to 
$\H^{m_{i}'+1}(\b^{}_i,\La^{\prime})\cap\Aut_{\D}(\V^j)$
is the transfer of $\t^{\prime}_i$ to $\Cc(\La^j,m'_i,\b^{}_i)$,
which is equal to $\t_i$ by Lemma \ref{Bachir}.
\end{proof}

Now let $g\in\A^{\times}$ intertwine $\t_1$ and $\t_2$ as in 
the identity (\ref{Eaque}), 
and set $g'=\iota(g)\in\M$. 
If we write $\H_{i}=\H^{m_{i}+1}(\b_i,\La)$ and 
$\H'_{i}=\H^{m_{i}'+1}(\b_i,\La^{\prime})$ for each integer $i\in\{1,2\}$, 
we get an Iwahori decomposition: 
\begin{eqnarray*}
\H'_2\cap g^{\prime -1}\H'_1g'&=&
\big(\H'_2\cap g^{\prime-1}\H'_1g'\cap\N^{-}\big)%\cdot
\big(\H'_2\cap g^{\prime-1}\H'_1g'\cap\M\big)%\cdot
\big(\H'_2\cap g^{\prime-1}\H'_1g'\cap\N\big),\\
\H'_2\cap g^{\prime-1}\H'_1g'\cap\M&=&
\big(\H_2\cap g^{-1}\H_1g\big)\times
\dots\times\big(\H_2\cap g^{-1}\H_1g\big).
\end{eqnarray*}
According to Lemma \ref{JacquesVerges}, the simple characters 
$\t^{\prime}_{1}$ and $\t^{\prime}_{2}$ are trivial
on the two sub\-groups $\H'_2\cap g^{\prime-1}\H'_1g'\cap\N$ and 
$\H'_2\cap g^{\prime-1}\H'_1g'\cap\N^{-}$, and we have:
\begin{equation*}
\t^{\prime}_i\ |\ \H'_2\cap g^{\prime-1}\H'_1g'\cap\M=
\big(\t_i\ |\ \H_2\cap g^{-1}\H_1g\big)
\otimes\dots\otimes\big(\t_i\ |\ \H_2\cap g^{-1}\H_1g\big)
\end{equation*}
for each $i\in\{1,2\}$.
This ensures that $g'$ intertwines the simple characters $\t'_1$ and 
$\t'_2$. 
\end{proof}

%%%%%%%%%%%%%%%%%%%%%%%%%%%%%%%%%%%%%%%%%%%%%%%%%%%%%%%%%%%%%%%%%%%%%%%%%%%

\subsection{}
\label{Dague}

We give an example which will be of particular interest for us.
Let $[\La,n,m,\b]$ be a simple stratum in $\A$, which is a realization
of a simple pair $(k,\b)$ over $\F$, and let $e$
denote the period of $\La$ over $\Oo_{\D}$. 
We set:
\begin{equation}
\label{Spadassin}
\La^{\dag}:k\mapsto\La_{k}\oplus\La_{k+1}\oplus\cdots\oplus\La_{k+e-1},
\end{equation}
which is a strict $\Oo_{\D}$-lattice sequence on 
$\V^{\dag}=\V\oplus\dots\oplus\V$ ($e$ times)
of the form (\ref{SpadassinVariante}).
Thus we can form the simple 
stratum $[\La^{\dag},n,m,\b]$ in $\A^{\dag}=\End_{\D}(\V^{\dag})$,
which is a realization of $(k,\b)$.
Moreover, the hereditary $\Oo_\F$-order $\AA^{\dag}$ defined by
$\La^{\dag}$ is principal, and we have the following result, 
which derives from Propositions \ref{LongJohnSilverGeneralVariante}
and \ref{JeromeBauche2}.

\begin{prop}
\label{LongJohnSilverVariante}
For $i=1,2$, let $[\La,n_i,m_i,\b_i]$ be a simple stratum in $\A$.
Assume they intertwine in $\A$ (resp. have the same embedding type).
Then the strata $[\La^{\dag},n_i,m_i,\b_i]$ 
intertwine in ${\A}^{\dag}$ (resp. have the same embedding type).
\end{prop}

Note that the operations $\La\mapsto\widetilde\La$ (see paragraph 
\ref{Split}) and $\La\mapsto\La^{\dag}$ 
commute, so that there is no ambiguity in writing $\widetilde\La^{\dag}$ 
for the strict $\Oo_\F$-lattice sequence defined by $\La^{\dag}$.

\begin{coro}
\label{Epaminondas}
Definition \ref{Petrone} is equivalent to Definition \cite[1.14]{BH}.
\end{coro}

\begin{proof}
Assume we are given two simple pairs $(k,\b_i)$, $i=1,2$, which are 
endo-equivalent in the sense of Definition \ref{Petrone}. 
Then we have $[\F(\b_1):\F]=[\F(\b_2):\F]$, and 
there exists a simple central $\F$-algebra $\A$ together 
with realizations $[\La,n_i,m_i,\h_i(\b_i)]$ of $(k,\b_i)$ 
in $\A$, with $i=1,2$, which intertwine in $\A$.
By replacing $\A$ and $\La$ by $\widetilde\A^{\dag}$ and
$\widetilde\La^{\dag}$, we have realizations 
$[\widetilde\La^{\dag},n_i,m_i,\h_i(\b_i)]$ of $(k,\b_i)$ in 
$\widetilde\A^{\dag}$, with $i=1,2$, and these realizations 
intertwine in $\widetilde\A^{\dag}$ according to Propositions 
\ref{LongJohnSilverTildeVariante} and \ref{LongJohnSilverVariante}.
Thus the simple pairs $(k,\b_1)$ and $(k,\b_2)$ are endo-equivalent 
in the sense of \cite[Definition 1.14]{BH}. 
Conversely, two simple pairs which are endo-equivalent in this 
sense are clearly endo-equivalent in the sense of Definition 
\ref{Petrone}. 
\end{proof}

%%%%%%%%%%%%%%%%%%%%%%%%%%%%%%%%%%%%%%%%%%%%%%%%%%%%%%%%%%

\subsection{}
\label{ProvaDAmore}

We now prove the preservation property of 
intertwining for sim\-ple strata, that is, 
Proposition \ref{SimonLeMagicienStrong}.
We first prove that the endo-equivalence relation 
preserves certain numerical invariants attached to simple pairs. 
Compare the following proposition with \cite{BH}, Property (1.15).
See paragraph \ref{Jessica} for the notation.

\begin{prop}
\label{Eponine}
For $i=1,2$, let $(k,\b_{i})$ be a simple pair over $\F$, and suppose that 
$(k,\b_1)$ and $(k,\b_2)$ are endo-equivalent.
Then we have $n_{\F}(\b_1)=n_{\F}(\b_2)$, $e_{\F}(\b_1)=e_{\F}(\b_2)$, 
$f_{\F}(\b_1)=f_{\F}(\b_2)$ and $k_{\F}(\b_1)=k_{\F}(\b_2)$.
\end{prop}

\begin{proof}
By Corollary \ref{Epaminondas}, we may assume that the pairs 
$(k,\b_{i})$ are endo-equivalent in the sense of \cite{BH}. 
The result then follows from \cite[Proposition 1.10]{BH}. 
\end{proof}

For $i=1,2$, let $(k,\b_{i})$ be a simple pair over $\F$, and suppose that 
$(k,\b_1)\thickapprox(k,\b_2)$.

\medskip

Let $\A$ be a simple central $\F$-algebra and, for $i=1,2$, 
let $[\La,n_i,m_i,\h_i(\b_i)]$ be a real\-ization of 
$(k,\b_i)$ in $\A$.
Let $\V$ denote the simple left $\A$-module on which $\La$ is a lattice 
sequence and write $\D$ for the $\F$-algebra opposite to $\End_{\A}(\V)$.
For $i=1,2$, let $\E_i$ denote the $\F$-algebra $\F(\b_i)$.
We fix a simple right $\E_1\otimes_{\F}\D$-module $\SS$ and set 
$\A(\SS)=\End_{\D}(\SS)$, and we denote by $\rho_1$ the natural 
$\F$-algebra homomorphism $\E_1\to\A(\SS)$.
Let $\Ss$ denote the unique (up to translation) $\E_1$-pure 
strict $\Oo_{\D}$-lattice sequence on $\SS$. 

\begin{lemm}
\label{Merimee}
There is a homomorphism of $\F$-algebras $\rho_2:\E_2\to\A(\SS)$ such 
that $\Ss$ is $\rho_2(\E_2)$-pure, 
and such that the pairs $(\rho_1(\E_1),\Ss)$ and $(\rho_2(\E_2),\Ss)$
have the same embedding type in $\A(\SS)$ (see paragraph \ref{MainStreet}). 
\end{lemm}

\begin{proof}
As $(k,\b_1)$ and $(k,\b_2)$ are endo-equivalent, Proposition \ref{Eponine} 
gives us the equalities $e_\F(\b_1)=e_\F(\b_2)$ and $f_\F(\b_1)=f_\F(\b_2)$.
The result follows from \cite[Corollary 3.16]{BG}.
\end{proof}

\begin{rema}
We actually have a stronger result: for any $\F$-algebra 
homomorphism $\rho_2$ such that $\Ss$ is 
$\rho_2(\E_2)$-pure, the pairs $(\rho_1(\E_1),\Ss)$ and 
$(\rho_2(\E_2),\Ss)$ have the same embedding type in $\A(\SS)$. 
Indeed, if $\rho_2$ is such a homomorphism and if $\n_{2}$
is an $\F$-algebra homo\-morphism as in Lemma \ref{Merimee}, the 
Skolem-Noether theorem gives us 
$g\in\A(\SS)^{\times}$ which conjugates these
$\F$-algebra homomorphisms $\rho_{2}$ and $\n_{2}$.
As $\E_1$ and $\E_2$ have the same degree over $\F$, the 
lattice sequence $\Ss$ is the unique (up to translation) 
$\rho_2(\E_2)$-pure strict $\Oo_{\D}$-lattice sequence | 
and also the unique (up to translation) $\n_2(\E_2)$-pure 
strict $\Oo_{\D}$-lattice sequence | on $\SS$. 
It follows that $g$ normalizes the lattice sequence 
$\Ss$ and that the pairs 
$(\rho_2(\E_2),\Ss)$ and $(\n_2(\E_2),\Ss)$ have the 
same embedding type in $\A(\SS)$. 
\end{rema}

Let us fix an $\F$-algebra homomorphism $\rho_2$ as in Lemma \ref{Merimee}.
As $(k,\b_1)$ and $(k,\b_2)$ are endo-equivalent, we have
$n_{\F}(\b_1)=n_{\F}(\b_2)$ and $e_{\F}(\b_1)=e_{\F}(\b_2)$,
so that the $\Ss$-valuation of $\rho_i(\b_i)$, denoted $n_0$, 
and the period $e_{\rho_i(\b_i)}(\Ss)$ do not depend on $i\in\{1,2\}$.
We set:
\begin{equation*}
m_0={e_{\rho_i(\b_i)}(\Ss)}k.
\end{equation*}
For each $i\in\{1,2\}$, we have a stratum $[\Ss,n_0,m_0,\rho_i(\b_i)]$, 
which is a realization of  $(k,\b_i)$ in $\A(\SS)$. 
By paragraph \ref{Split}, we have a realization 
$[\widetilde\Ss,n_0,m_0,\rho_i(\b_i)]$ of $(k,\b_i)$ in the split simple 
central $\F$-algebra $\End_{\F}(\SS)$, and the $\Oo_{\F}$-lattice sequence
$\widetilde\Ss$ is strict.
Hence we can apply \cite[Proposition 1.10]{BH}, 
which implies that these realizations, for $i=1,2$, intertwine in 
$\End_{\F}(\SS)$.
By our assumption (see Lemma \ref{Merimee}), 
the strata $[\Ss,n_0,m_0,\rho_i(\b_i)]$, for $i=1,2$, 
have the same embedding type. 
Here we need to recall the following statement, due to Broussous and 
Grabitz. 

\begin{prop}[\cite{BG}, Proposition 4.1.3]
\label{BG413}
For $i=1,2$, let $[\Sigma,n,m,\g_i]$ be a simple stratum in a simple 
central $\F$-algebra $\U$, where $\Sigma$ is strict.
Assume that they have the same embedding type, and that the strata
$[\widetilde{\Sigma},n,m,\g_i]$ intertwine in $\widetilde{\U}$.
Then there exists an element $u\in\KK(\Sigma)$ such that 
$\g_1-u\g_2u^{-1}\in\aa_{-m}(\Sigma)$.

Moreover, $u$ can be chosen such that the maximal unramified extension 
of $\F$ contained in $\F(\g_1)$ is equal to that of $\F(u\g_2u^{-1})$.
\end{prop}

We deduce from Proposition \ref{BG413}
that there exists an element $g\in\KK(\Ss)$ such that:
\begin{equation}
\label{Laodice}
\rho_1(\b_1)-g\rho_2(\b_2)g^{-1}\in\PP_{-m_0}(\Ss).
\end{equation}
We now fix a decomposition:
\begin{equation*}
\V=\V^{1}\oplus\cdots\oplus\V^{l}
\end{equation*}
of $\V$ into simple right $\E_{1}\otimes_{\F}\D$-modules (which all are 
copies of $\SS$) such that the lattice sequence $\La$ 
decomposes into the direct sum of the 
$\La^{j}=\La\cap\V^{j}$, for $j\in\{1,\dots,l\}$.

\begin{lemm}
\label{Merimee2}
There are isomorphisms of $\E_{1}\otimes_{\F}\D$-modules 
$\V^{j}\to\SS$, $j\in\{1,\dots,l\}$, such that 
the resulting $\F$-algebra homomorphism $\iota:\A(\SS)\to\A$
satisfies $\iota\circ\rho_{1}=\h_{1}$.
\end{lemm}

\begin{proof}
Since each $\V^{j}$, for $j\in\{1,\dots,l\}$, 
is an $\E_1$-vector subspace of $\V$,
the $\F$-algebra homomorphism $\h_1$ has the form
$x\mapsto(\omega_{1}(x),\dots,\omega_{l}(x))$, where 
$\omega_{j}$ is an $\F$-algebra homo\-mor\-phism from $\E_1$ to 
$\End_{\D}(\V^{j})$.
By the Skolem-Noether theorem, one can choose, 
for each integer $j$, a suitable $\E_{1}\otimes_{\F}\D$-module 
isomorphism between $\V^{j}$ and $\SS$ such that the 
resulting $\F$-algebra homomorphism $\pi_{j}$ between 
$\End_{\D}(\V^{j})$ and $\A(\SS)$ satisfies the condi\-tion 
$\pi_{j}\circ\omega_{j}=\rho_{1}$.
Then the $\F$-algebra homomorphism $\iota$ defined by 
$\iota(x)=(\pi_1^{-1}(x),\dots,\pi_l^{-1}(x))$ for $x\in\A(\SS)$
satisfies the required condition.
\end{proof}

We now fix isomorphisms of $\E_{1}\otimes_{\F}\D$-modules 
$\V^{j}\to\SS$, $j\in\{1,\dots,l\}$, as in Lemma \ref{Merimee2}. 
Then each $\La^{j}$ is in the affine class of $\Ss$ 
(see (\ref{AffCla}) and \cite[\S1.4.8]{VS3}), and these 
lattice sequences all have the same 
period, equal to that of $\La$.
Therefore, we are in the situation of paragraph \ref{DerivedStratum}.
We set:
\begin{equation*}
n=n_i, \quad m={e_{\h_i(\b_i)}(\La)}k, 
\end{equation*}
which both do not depend on $i\in\{1,2\}$.
By (\ref{Laodice}) and Remark \ref{blioblieris}, the element 
$\iota(g)$ normalizes $\La$ and conjugates 
$[\La,n,m,\iota(\rho_{2}(\b_2))]$ into a simple 
stratum in $\A$ which is equivalent to $[\La,n,m,\h_1(\b_1)]$.
By the Skolem-Noether theorem, there is an element 
$x\in\A^{\times}$ which 
conjugates the $\F$-algebra homomorphisms $\iota\circ\rho_{2}$
and $\h_2$, and thus intertwines the simple strata 
$[\La,n,m,\iota(\rho_{2}(\b_2))]$ and 
$[\La,n,m,\h_{2}(\b_2)]$.
Therefore the strata $[\La,n,m,\h_i(\b_i)]$ intertwine.
As $m\<m_1,m_2$, the strata 
$[\La,n_i,m_i,\h_i(\b_i)]$ intertwine, which 
ends the proof of Proposition \ref{SimonLeMagicienStrong}.

%%%%%%%%%%%%%%%%%%%%%%%%%%%%%%%%%%%%%%%%%%%%%%%%%%%%%%%%%%%%%%%%%%%%%%%%%%%

\begin{rema}
\label{gaptransfer}
There is a gap in the proof of the existence of the transfer map given
in \cite[Th\'eor\`eme 3.53]{VS1}, in the case where $\La$ is a strict lattice 
sequence. 
To complete this proof, one has to prove that, given a non-minimal simple 
pair $(k,\b)$ over $\F$ together with 
a realization $[\La,n,m,\h(\b)]$ of this pair 
in a simple central $\F$-algebra $\A$, 
there is a simple pair $(k',\g)$ over $\F$ having realizations in $\A(\E)$ 
and $\A$ which are approximations of $\b$ and $\h(\b)$, respectively.
More precisely, set $q=-k_0(\b,\La)$, and start with a stratum 
$[\La,n,q,\g]$ in $\A$ which is simple and equivalent to $[\La,n,q,\h(\b)]$.
If we denote by $(k',\g)$ the simple pair of which this stratum 
is a realization, and if we set $n_0=n_{\F}(\b)$ and $q_0=-k_\F(\b)$,
then we search for a realization $[\La(\E),n_{0},q_0,\h_0(\g)]$
of $(k',\g)$ in $\A(\E)$ which is equivalent to the pure stratum 
$[\La(\E),n_{0},q_0,\b]$ (see paragraph \ref{Jessica}).
Let us remark that, 
when passing to $\widetilde{\A}$ 
(see paragraph \ref{Split}), we get a stratum 
$[\widetilde\La,n,q,\g]$ which is simple and equivalent to 
$[\widetilde\La,n,q,\h(\b)]$.
Now let $[\La(\E),n_{0},q_0,\delta]$ be a stratum in $\A(\E)$ which is 
simple and equi\-va\-lent to $[\La(\E),n_{0},q_0,\b]$.
By choosing a suitable decomposition of the $\F$-vector space $\V$ into 
a direct sum of copies of $\E$, we get an $\F$-embedding:
\begin{equation*}
\iota:\A(\E)\to\widetilde\A,
\end{equation*}
thus a stratum 
$[\widetilde\La,n,q,\iota(\delta)]$ in $\widetilde\A$
which is simple and equivalent to 
$[\widetilde\La,n,q,\iota(\b)]$.
By the Skolem-Noether theorem, there exists an element 
$g\in\widetilde\A^{\times}$
which conjugates $\iota(\b)$ and $\h(\b)$, thus inter\-twines 
the strata $[\widetilde\La,n,q,\g]$ and $[\widetilde\La,n,q,\iota(\delta)]$. 
The simple pairs $(k',\g)$ and $(k',\delta)$ are thus endo-equivalent.
Now let $[\La(\E),n_{0},q_0,\jmath(\g)]$ be a realization of $(k',\g)$ in 
$\A(\E)$ which intertwines with $[\La(\E),n_{0},q_0,\delta]$.
By the ``intertwining implies conjugacy'' theorem \cite[Theorem 3.5.11]{BK} 
in the split simple central $\F$-algebra $\A(\E)$, 
there is $g\in\U(\La(\E))$ such that 
$g\jmath(\g)g^{-1}-\delta\in\PP(\La(\E))^{-q_0}$.
The homomorphism of $\F$-algebras $\h_0:x\mapsto g\jmath(x)g^{-1}$ has the 
required property.
\end{rema}

%%%%%%%%%%%%%%%%%%%%%%%%%%%%%%%%%%%%%%%%%%%%%%%%%%%%%%%%%%%%%%%%%%%%%%%%%

\subsection{}
\label{PreDague}

Before closing this section, we give a more elaborate example 
than that of paragraph \ref{Dague}, which will be very useful 
in the sequel. 
As in paragraph \ref{Dague}, let $[\La,n,m,\b]$ be a simple stratum in
$\A$, which is a realization of a simple pair $(k,\b)$ over $\F$, and
let $e$ denote the period of $\La$ over $\Oo_{\D}$.
Write $\B$ for the centralizer of the field $\E=\F(\b)$ in $\A$,
fix a simple left $\B$-module $\V_{\b}$ and write $\D_{\b}$ for the 
$\E$-algebra opposite to the algebra of $\B$-endomorphisms of $\V_{\b}$.
Let $\Sigma$ denote an $\Oo_{\D_{\b}}$-lattice sequence on $\V_{\b}$ 
corresponding to $\La$ by (\ref{FPrimedescen}), and let $e'$ denote
its period over $\Oo_{\D_{\b}}$. 
We fix an integer $l$ which is a multiple of $e$ and $e'$ and set: 
\begin{equation}
\label{Francisque}
\La^{\ddag}:k\mapsto\La_{k}\oplus\La_{k+1}\oplus\cdots\oplus\La_{k+l-1},
\end{equation}
which is a strict $\Oo_{\D}$-lattice sequence on
$\V^{\ddag}=\V\oplus\dots\oplus\V$ ($l$ times) of the form
(\ref{SpadassinVariante}).
Thus we can form the simple stratum $[\La^{\ddag},n,m,\b]$
in $\A^{\ddag}=\End_{\D}(\V^{\ddag})$, which is a realization 
of $(k,\b)$.
Moreover, the hereditary $\Oo_\F$-order $\AA^{\ddag}$ defined 
by $\La^{\ddag}$ is principal, and we have the following result.

\begin{lemm}
\label{EveryBodyIsSound}
The stratum $[\La^{\ddag},n,m,\b]$ is sound
(see Definition \ref{StrateSoundDef}).
\end{lemm}

\begin{proof}
Write $\B^{\ddag}$ for the centralizer of $\E$ in $\A^{\ddag}$ and 
$\Sigma^{\ddag}$ for the $\Oo_{\D_{\b}}$-lattice sequence on 
$\V_{\b}\times\dots\times\V_{\b}$ ($l$ times) defined by: 
\begin{equation*}
\Sigma^{\ddag}:k\mapsto\Sigma_{k}\oplus\Sigma_{k+1}
\oplus\cdots\oplus\Sigma_{k+l-1}.
\end{equation*}
This is a strict lattice sequence, which defines a principal order 
of $\B^{\ddag}$.
By direct computation of each block, we get for all $k\in\ZZ$: 
\begin{equation}
\label{DuSeigneur}
\aa_{k}(\La^{\ddag})\cap\B^{\ddag}=\aa_{k}(\Sigma^{\ddag}),
\end{equation}
which amounts to saying that $\Sigma^{\ddag}$ is an
$\Oo_{\D_{\b}}$-lattice sequence which corresponds 
to $\La^{\ddag}$ by (\ref{FPrimedescen}).
In particular, its $\B^{\ddag\times}$-normalizer is 
$\KK(\La^{\ddag})\cap\B^{\ddag\times}$. 
As $\La^{\ddag}$ is strict, its normalizer is equal to $\KK(\AA^{\ddag})$, 
and a similar statement holds for the lattice sequence $\Sigma^{\ddag}$, 
so that we have
$\KK(\AA^{\ddag})\cap\B^{\ddag\times}=\KK(\AA^{\ddag}\cap\B^{\ddag})$. 
Finally, if we choose $k=0$ in (\ref{DuSeigneur}), we deduce that
$\AA^{\ddag}\cap\B^{\ddag}$ is principal.
\end{proof}

Note that, unlike (\ref{Spadassin}), the process def\-ined by 
(\ref{Francisque}) depends on $\E$ and $l$, and not only on 
the lattice sequence $\La$.

\medskip

Now let $[\La,n_i,m_i,\b_{i}]$, for $i=1,2$, be simple strata 
in $\A$.
Let $e$ denote the period of $\La$ over $\Oo_{\D}$, and
write $e'_i$ for the period of the $\Oo_{\D_{\b_i}}$-lattice sequence 
associated with $\La$ as above. 

\begin{prop}
\label{LongJohnSilverVarianteSound}
Let $l\>1$ be a multiple of $e'_1,e'_2$ and $e$, and assume that the simple 
strata $[\La,n_i,m_i,\b_i]$, $i=1,2$, intertwine in $\A$ (resp. have the 
same embedding type).
Then the simple strata $[\La^{\ddag},n_i,m_i,\b_i]$, $i=1,2$, are sound, 
and intertwine in ${\A}^{\ddag}$ (resp. have the same embedding type). 
\end{prop}

\begin{proof}
This derives from Propositions 
\ref{LongJohnSilverGeneralVariante} and \ref{JeromeBauche2},
and Lemma \ref{EveryBodyIsSound}. 
\end{proof}

%%%%%%%%%%%%%%%%%%%%%%%%%%%%%%%%%%%%%%%%%%%%%%%%%%%%%%%%%%%%%%%%%%%%%%%%%%%
%%%%%%%%%%%%%%%%%%%%%%%%%%%%%%%%%%%%%%%%%%%%%%%%%%%%%%%%%%%%%%%%%%%%%%%%%%%

\section{Intertwining implies conjugacy for simple strata}
\label{EESP}

In this section, we prove the 
``intertwining implies conjugacy'' property for simple 
strata, that is Proposition \ref{IreneeDeLyon}.
We fix a simple central $\F$-algebra $\A$ and a simple left $\A$-module 
$\V$ as in paragraph \ref{Dubol}. 
Associated with it, we have an $\F$-division algebra $\D$.

%%%%%%%%%%%%%%%%%%%%%%%%%%%%%%%%%%%%%%%%%%%%%%%%%%%%%%%%%%%%%%%%%%%%%%%%%%%

\subsection{}
\label{LanceQuiSaigne}

We will need the following general lemma on embed\-ding types.
Let $\Bb$ be a $\D$-basis of $\V$, and let $\L$ be a 
max\-imal unrami\-fied extension of $\F$ contained in $\D$. 
The choice of $\Bb$ defines an isomorphism of $\F$-algebras 
between $\A$ and $\Mat_r(\D)$ 
for some integer $r\>1$, which allows us to identify these 
$\F$-algebras. 
In particular, we will consider $\L$ as an extension of $\F$ 
contained in $\A$. 
We write $\I_r$ for the identity matrix.

\medskip

An embedding $(\K,\La)$ in $\A$ is said to be {\it standard} with respect to 
the pair $(\Bb,\L)$ if $\K$ is a subfield of $\L$ and if $\La$ is split by the 
basis $\Bb$ in the sense of \cite{BL}. 

\begin{lemm}
\label{ConjPiD}
Let $(\Bb,\L)$ be a pair as above. 
\begin{enumerate}
\item 
Any embedding in $\A$ is equivalent to an embedding 
%in $\A$ 
which is standard with res\-pect to the pair $(\Bb,\L)$.
\item
Let $(\K,\La)$ be standard with respect to $(\Bb,\L)$, and let $\w$ be a 
uniformizer of $\D$ nor\-malizing $\L$. 
Then conjugation by the diagonal matrix $\w\cdot\I_r$ 
norma\-li\-zes $\K$ and $\La$, and any element of $\Gal(\K/\F)$ 
is induced by conjugation by a power of $\w\cdot\I_r$.
\end{enumerate}
\end{lemm} 

\begin{proof} 
Assertion (2) follows from the fact that the map $x\mapsto\w x\w^{-1}$, 
for $x\in\L$, is a generator of the group $\Gal(\L/\F)$.
To prove (1), let $(\E,\La)$ be an embedding in $\A$, and set $\K=\E^{\flit}$
(see paragraph \ref{MainStreet} for the notation). 
One first notices that one can 
con\-jugate the pair $(\K,\La)$ so that $\K\subseteq\L$, 
which we will assume. 
Let $\Ii$ be the non-enlarged Bruhat-Tits building of 
$\mult\A$ and $\Ii'$ be that of the centralizer $\mult\C$ of 
$\K^{\times}$ in $\mult\A$. 
Since the group $\mult\C$ identifies with $\GL_r(\D')$, where $\D'$ 
is the centralizer of $\K$ in $\D$, the two buildings $\Ii$ 
and $\Ii'$ have same dimension $r-1$. 
Recall (see \cite[Th\'eor\`eme II.1.1]{BL}) 
that there exists a unique mapping:
\begin{equation*}
\boldsymbol{j}=\boldsymbol{j}_{\K/\F}:\Ii'\to\Ii
\end{equation*}
which is affine and $\mult\C$-equivariant. 
Its image is the set of $\K^{\times}$-fixed points in $\Ii$.  
The basis $\Bb$ gives rise to an apartment ${\Aa}$ of $\Ii$ 
(see e.g. \cite[\S0]{BL}), 
and points in that apartment are fixed by diagonal matrices 
of $\mult\A$ of the form $x\cdot\I_r$, with $x\in\mult\D$.
In particular, they are fixed by $\K^{\times}$. 
It easily follows that there is some apartment $\Aa'$ in $\Ii'$
such that we have $\Aa=\boldsymbol{j}(\Aa')$.

The affine class of $\La$
determines a point $y$ of the building $\Ii$ 
(see \cite[I.7]{BL}). 
Since $\K^{\times}$ nor\-ma\-li\-zes $\La$, 
this point writes $\boldsymbol{j}(x)$, 
for some $x\in\Ii'$.  
Since $\mult\C$ acts transitively on the set of all 
apartments of $\Ii'$, and since any point of $\Ii'$ 
is contained in some apartment, there exists an element 
$h\in\mult\C$ such that $h\cdot x\in\Aa'$. 
Its follows that $h\cdot y=\boldsymbol{j}(h\cdot x)$ lies in $\Aa$. 
By \cite[Proposition I.2.7]{BL}, 
this means that the lattice sequence $h\La$ 
is split by the basis $\Bb$, i.e. that 
$(h\K h^{-1},h\La)=(\K,h\La)$ 
is standard with respect to the pair $(\Bb,\L)$, as required.
\end{proof}

\begin{rema}
We can rephrase Assertion (1) of the above lemma by saying that, 
for any embedding $(\E,\La)$ in $\A$, 
there is $g\in\mult\A$ such that 
$(\E^\flit,\La)$ is standard with respect to the 
pair $(g\Bb,g\L g^{-1})$.

If one writes $\N_{\mult\A}(\K)$ for the normalizer of $\K$ in $\mult\A$,
Assertion (2) can also be rephrased by saying that conjugation induces 
a surjective group homomorphism from the intersection 
$\KK(\La)\cap\N_{\mult\A}(\K)$ onto $\Gal(\K/\F)$. 
With the notation of the proof of Lemma \ref{ConjPiD}, 
the kernel of this homomorphism is $\KK(\La)\cap\mult\C$.
\end{rema}

%%%%%%%%%%%%%%%%%%%%%%%%%%%%%%%%%%%%%%%%%%%%%%%%%%%%%%%%%%%%%%%%%%%%%%%%%%%

\subsection{}
\label{IICSS}

We will also need the following result, which generalizes \cite[Lemma 1.6]{BH}. 

\begin{prop}
\label{Tauride} 
Let $\La$ be an $\Oo_{\D}$-lattice sequence on $\V$ 
and $\E/\F$ a finite extension.
Suppose that there are two homomorphisms $\h_i:\E\to\A$ of
$\F$-algebras, $i=1,2$, such that the pairs $(\h_1(\E),\La)$
and $(\h_2(\E),\La)$ are two equivalent embeddings in $\A$.
Then there is an element $u\in\KK(\La)$ such that: 
\begin{equation}
\label{Iphigenie}
\h_1(x)=u\h_2(x)u^{-1}, \quad x\in\E.
\end{equation}
\end{prop}

\begin{rema}
\label{Pisot}
In particular, if $\K$ denotes the maximal un\-ramified extension 
of $\F$ contained in $\E$, then $u$ conjugates $\h_2(\K)$ to $\h_1(\K)$. 
\end{rema}

\begin{proof}
Since the embeddings 
$(\h_1(\E),\La)$ and $(\h_2(\E),\La)$ are equivalent, 
there exists an element $g\in\KK(\La)$ such that 
$\h_1(\E^{\flit})=g\h_2(\E^{\flit} )g^{-1}$. 
Then the mapping: 
\begin{equation}
\label{Meria}
x\mapsto g\h^{}_2(\h_{1}^{-1}(x))g^{-1}
\end{equation} 
is an $\F$-automorphism of $\h_1(\E^{\flit})$. 
By Lemma \ref{ConjPiD}(2), there is $h\in\KK(\La)$ such that 
this $\F$-auto\-morphism is equal to $x\mapsto hxh^{-1}$. 
We thus have 
$\h_1(x)=w\h_2(x)w^{-1}$, for all $x\in\E^{\flit}$, where $w=h^{-1}g$.
So replacing $\h_2$ by a $\KK (\La)$-conjugate, one may reduce to the
case where $\h_1$ and $\h_2$ coincide on $\E^{\flit} $.
Assume now that we are is this case, and put $\K=\h_2(\E^{\flit})$.
Let $\C$ be the centralizer of $\K$ in $\A$, and write $\CC$ for 
the intersection of $\AA=\AA(\La)$ with $\C$. 

\begin{lemm}
\label{TotRam}
There is $u\in\U(\CC)$ such that {\rm (\ref{Iphigenie})} holds. 
\end{lemm}

\begin{proof}
We fix an unramified extension $\L$ of $\K$ such that the degree 
of $\L/\F$ is equal to the reduced degree of $\D$ over $\F$, denoted $d$.
The $\L$-algebra $\CL=\C\otimes_\K\L$ is thus split and, 
as $\E/\K$ has residue class degree prime to $d$, 
the $\L$-algebra $\E\otimes_\K\L$ is an extension 
of $\L$, denoted $\overline{\E}$. 
For each $i$, the $\K$-algebra homo\-morphism $\h_i$ extends 
to a homo\-morphism of $\L$-algebras $\overline{\E}\to\CL$, still 
denoted $\h_i$. 
By applying \cite[Lemma 1.6]{BH} with the $\Oo_{\L}$-order 
$\overline{\CC}=\CC\otimes_{\Oo_{\K}}\Oo_{\L}$ and the homo\-morphisms 
of $\L$-algebras $\h_1$ and $\h_2$, we get 
$u\in\U(\overline{\CC})$ satisfying (\ref{Iphigenie}).  
If we write $\overline{\B}$ for the centralizer of $\h_{2}(\E)$ in 
$\CL$, then the $1$-cocycle $\s\mapsto u^{-1}\s(u)$ 
defines a class in the Galois cohomology set:
\begin{equation*}
\H^{1}(\Gal(\L/\K),\U(\overline{\CC})\cap\overline{\B}{}^{\times}).
\end{equation*}
This cohomology set is trivial by a standard filtration argument.
(For more detail, see e.g. \cite[\S6]{BG}.)
Hence we actually may choose $u$ in $\U(\CC)$, which ends the proof of the 
lemma.  
\end{proof}

Proposition \ref{Tauride} follows immediately from Lemma \ref{TotRam}.
\end{proof}

\begin{rema}
The conclusion of Proposition \ref{Tauride} does not hold if the pairs
$(\h_1(\E),\La)$ and $(\h_2(\E),\La)$ are not assumed to be 
equivalent in $\A$. 
For instance, take $\A=\Mat_2(\D)$ where $\D$ is a quaternionic algebra
over $\F$, and let $\E/\F$ be an unramified quadratic extension. 
%of $\F$. 
One may embed $\E$ in $\Mat_2(\F)$ so that the multiplicative group of the 
image normalizes the order $\Mat_2(\Oo_\F)$. 
This gives an embedding $\h_1$ of $\E$ in 
$\A=\Mat_2(\D)=\Mat_2(\F)\otimes_{\F}\D$, such that $\h_1(\E^{\times})$ 
normalizes $\Mat_2(\Oo_{\D})=\Mat_2(\Oo_\F )\otimes_{\Oo_\F}\Oo_\D$. 
One also may embed $\E$ in $\D$. 
The diagonal embedding of $\D$ in $\A$ gives rise to a second embedding 
$\h_2$ such that $\h_2 (\E^{\times})$ normalizes $\Mat_2(\Oo_{\D})$. 
Take $\La$ to be a strict lattice sequence 
in $\D\times\D$ defining the order 
$\AA=\Mat_2(\Oo_\D )$, so that:
\begin{equation*}
\KK(\La)=\KK(\AA)=\langle\w\rangle\cdot\U(\AA),
\end{equation*}
where $\w$ denotes a uniformizer of $\D$ and $\langle\w\rangle$ the subgroup 
generated by $\w$.
One can check that the pairs $(\h_i(\E),\La)$, $i=1,2$, are 
inequivalent. 
Assume for a contradiction that there is an element
$u\in\KK(\AA)$ such 
that $\h_1(\E)=u\h_2(\E)u^{-1}$, and write $\PP$ for the radical 
of $\AA$.
For $i=1,2$, the map $\h_i$ induces an embedding of the residue 
field $\mathfrak{k}_{\E}$ in the $\mathfrak{k}_{\F}$-algebra 
$\AA/\PP$, which is isomorphic to $\Mat_2(\mathfrak{k}_{\D})$, 
and the images $\h_i(\mathfrak{k}_{\E})$, $i=1,2$, are conjugate 
under the action of $u$ on the quotient $\AA/\PP$. 
But this action stabilizes the centre of $\Mat_2(\mathfrak{k}_{\D})$ 
and $\h_2(\mathfrak{k}_{\E})$ lies in this centre. 
This implies that $\h_1(\mathfrak{k}_{\E})$ is central: 
a contradiction.
\end{rema}

%%%%%%%%%%%%%%%%%%%%%%%%%%%%%%%%%%%%%%%%%%%%%%%%%%%%%%%%%%%%%%%%%%%%%%%%%%%

\subsection{}
\label{Barnabooth}

We now prove the ``intertwining implies conjugacy'' property for sim\-ple 
strata, that is, Proposition \ref{IreneeDeLyon}.
For $i=1,2$, let $[\La,n,m,\b_i]$ a simple stratum in $\A$.
Assume that they intertwine in $\A$ and have the same embedding type, 
and write $\K_i$ for the maximal un\-ramified extension of $\F$ 
contained in $\E_i=\F(\b_i)$.
By Remark \ref{Pisot}, 
we may replace $\b_2$ by some $\KK(\La)$-conjugate and assume that $\K_1$ 
and $\K_2$ are equal to a common extension $\K$ of $\F$.
We write $\N_{\mult\A}(\K)$ for the normalizer of $\K$ in $\mult\A$.
Therefore, we are reduced to proving that there is an element
$u\in\KK(\La)\cap\N_{\mult\A}(\K)$ such that we have
$\b_1-u\b_2u^{-1}\in\aa_{-m}(\La)$. 

\medskip

We proceed as in the proof of proposition \ref{SimonLeMagicienStrong}
(see paragraph \ref{ProvaDAmore}).
Let us fix a simple right $\E_1\otimes_{\F}\D$-module $\SS$ and set 
$\A(\SS)=\End_{\D}(\SS)$.
Let us denote by $\rho_1$ the natural 
$\F$-algebra homomorphism $\E_1\to\A(\SS)$.
We write $\Ss$ for the unique (up to translation) $\E_1$-pure strict 
$\Oo_{\D}$-lattice sequence on $\SS$ and fix
an $\F$-algebra homomorphism $\rho_2:\E_2\to\A(\SS)$ such 
that $\Ss$ is $\rho_2(\E_2)$-pure, and such that 
$(\rho_1(\E_1),\Ss)$ and $(\rho_2(\E_2),\Ss)$ have the 
same embedding type in $\A(\SS)$ (see Lemma \ref{Merimee}).
We also fix a decomposition: 
\begin{equation}
\label{Aron224}
\V=\V^{1}\oplus\cdots\oplus\V^{l}
\end{equation}
of $\V$ into simple right $\K(\b)\otimes_{\F}\D$-modules (which all are 
copies of $\SS$) such that $\La$ is de\-com\-posed by 
(\ref{Aron224}) in the sense of \cite[D\'efinition 1.13]{VS3}, that is, 
$\La$ is the direct sum of the lattice sequences 
$\La^{j}=\La\cap\V^{j}$, for $j\in\{1,\dots,l\}$.
By choosing, for each $j$, an iso\-mor\-phism of 
$\K(\b)\otimes_{\F}\D$-modules between $\SS$ and $\V^{j}$, 
this decomposition gives us an $\F$-algebra homomorphism:
\begin{equation*}
\iota:\A(\SS)\to\A. 
\end{equation*}
Using Lemma \ref{Merimee2},
we may assume that this homomorphism 
satisfies $\iota(\rho_1(\b_1))=\b_1$.

\medskip

For each $i\in\{1,2\}$, let $(k,\b_i)$ be the simple pair of which 
$[\La,n,m,\b_i]$ is a realization. 
By putting $n_0=n_{\F}(\b_i)$ and $m_0={e_{\rho_i(\b_i)}(\Ss)}k$, 
which do not depend on $i$ by Proposition \ref{Eponine},
we get a simple stratum $[\Ss,n_0,m_0,\rho_i(\b_i)]$ which is a
realization of $(k,\b_i)$ in $\A(\SS)$.
The proof of \cite[Theorem 4.1.2]{BG} 
(see also \cite[Lemma 10.5]{Gr}) gives us
an element $v\in\KK(\Ss)$ such that 
$\rho_1(\K)=v \rho_2(\K)v^{-1}$
and $\b_1-v\b_2v^{-1}\in\aa_{-m_0}(\Ss)$. 
By Proposition \ref{Tauride}, there is %an element 
$w\in\KK(\La)$ such that $\iota(\rho_2(x))=wxw^{-1}$ 
for all $x\in\E_2$, and, by Remark \ref{Pisot}, this element 
satisfies $w\K w^{-1}=\iota(\rho_2(\K))$.
Thus $u=\iota(v)w$ normalizes $\K$ and $\La$ and 
satisfies the required condition:
\begin{equation*}
\b_1-u\b_2u^{-1}\in\aa_{-e_{\b_i}(\La)k}(\La)\subseteq\aa_{-m}(\La),
\end{equation*}
which ends the proof of Proposition \ref{IreneeDeLyon}.

%%%%%%%%%%%%%%%%%%%%%%%%%%%%%%%%%%%%%%%%%%%%%%%%%%%%%%%%%%
%%%%%%%%%%%%%%%%%%%%%%%%%%%%%%%%%%%%%%%%%%%%%%%%%%%%%%%%%%

\section{Realizations and intertwining for simple characters}
\label{Sec4}

The two main results of this section are Propositions 
\ref{Hengist6} and \ref{Dolomites}.
The first one asserts that two endo-equivalent ps-characters have 
realizations with very special properties, allowing us to 
use the results of \cite{Gr}.
The second one leads to the rigidity theo\-rem \ref{Paraclet}, 
and will also give us an important property of the base change 
map in paragraph \ref{EhOui}. 

%%%%%%%%%%%%%%%%%%%%%%%%%%%%%%%%%%%%%%%%%%%%%%%%%%%%%%%%%%

\subsection{}

In this paragraph, we generalize the construction given in para\-graph 
\ref{PreDague} by in\-cor\-po\-ra\-ting the notion of embedding type.
For this, we will need the following definition.

\medskip

Let $[\La,n,m,\b]$ be a simple stratum in $\A$, which is a realization 
of a simple pair $(k,\b)$ over $\F$, and set $\E=\F(\b)$. 
The containment of $\Oo_\E$ in $\AA(\La)$ allows us to identify the 
residue field $\mathfrak{k}=\mathfrak{k}_{\E^{\flit}}$ with its canonical 
image in the $\mathfrak{k}_\F$-algebra $\overline{\AA}=\AA(\La)/\PP(\La)$.

\begin{defi}\label{def:Frinv}
The {\it Fr\"ohlich invariant} of $[\La,n,m,\b]$ is 
the degree over $\mathfrak{k}_{\F}$ of the intersection of $\mathfrak{k}$ 
with the centre of $\overline{\AA}$.
\end{defi}

Recall that this invariant has been introduced by Fr\"ohlich (see \cite{Fr}) 
for sound strata. 
In this case, we have the following important property. 

\begin{theo}[\cite{Fr}, Theorem 2]
\label{FROH}
For $i=1,2$, let $(\K_i,\La)$ be a sound embedding in $\A$ where $\K_i/\F$ 
is an unramified extension contained in $\A$. 
These embeddings are equivalent if and only if 
$[\K_1^{\flit}:\F]=[\K_2^{\flit}:\F]$ and they 
have the same Fr\"ohlich invariant.
\end{theo}

We will need the two following lemmas.

\begin{lemm}
\label{Thomas}
Let us fix an integer $l\>1$, an 
$\Oo_{\D}$-lattice sequence $\La'$ and an integer $m'$
as in paragraph \ref{DerivedStratum}, and 
let us form the simple stratum $[\La',n',m',\b]$ in $\A'$.
The simple 
strata $[\La,n,m,\b]$ and $[\La',n',m',\b]$ have the same 
Fr\"ohlich invariant.
\end{lemm}

\begin{proof}
Let us identify $\A'$ with the matrix algebra $\Mat_l(\A)$, 
and write $j$ for the $\mathfrak{k}_\F$-algebra homomorphism 
$\mathfrak{k}\to\overline{\AA}{}'=\AA(\La')/\PP(\La')$ 
induced by the embedding of $\Oo_\E$ in $\AA(\La')$ (which is
the restriction to $\Oo_\E$ of the diagonal embedding of $\E$
in $\A'$).
By a direct computation, we see that the diagonal blocks of 
$\AA(\La^{\prime})$ are equal to $\AA(\La)$, 
and that of its radical $\PP(\La^{\prime})$ are equal
to $\PP(\La)$.
This is enough to prove that
$j(x)$ is central in $\overline{\AA}{}'$ if and only if 
$x$ is central in $\overline{\AA}$.
Thus the strata 
$[\La,n,m,\b]$ and $[\La',n',m',\b]$ have the same Fr\"ohlich 
invariant.
\end{proof}

\begin{lemm}
\label{Blaise3}
We set $\La'=\La\oplus\La$ and $m'=m$ (thus $l=2$).
There exists an element $u\in\A^{\prime\times}$ such 
that $\La'$ is $u\F(\b)u^{-1}$-pure and the simple 
stratum $[\La',n,m,u\b u^{-1}]$ in $\A'$ 
has Fr\"ohlich invariant $1$.
\end{lemm}

\begin{proof}
We fix a $\D$-basis $\Bb$ of $\V$, a max\-imal unramified extension $\L$ 
of $\F$ contained in $\D$ and a uniformizer $\w$ of $\D$ normalizing $\L$
(see paragraph \ref{LanceQuiSaigne}).
According to Lemma \ref{ConjPiD}, 
we may identify $\A$ with $\Mat_r(\D)$ and assume
that the embedding $(\E^{\flit},\La)$ is in standard form with respect to 
$(\Bb,\L)$.
The map $\h:x\mapsto\w x\w^{-1}$
defines a generator of $\Gal(\E^\flit/\F)$, 
and thus induces on the residue field 
$\mathfrak{k}=\mathfrak{k}_{\E^{\flit}}$ a generator
of $\Gal(\mathfrak{k}/\mathfrak{k}_\F)$, denoted $\s$.  
We write $j$ for the $\mathfrak{k}_\F$-algebra homomorphism 
from $\mathfrak{k}$ to $\overline{\AA}$
induced by $\h$, which is the composite of 
$\s$ with the canonical embedding of $\mathfrak{k}$ 
in $\overline{\AA}$. 
Thus, one has $j(x)=x$ if and only if $x\in\mathfrak{k}_\F$. 
We now set:
\begin{equation*}
u=
\begin{pmatrix}
\I_r&0\\
0&\w\cdot\I_r\\
\end{pmatrix}\in\Mat_2(\A)=\A'.
\end{equation*}
If one identifies the $\mathfrak{k}_\F$-algebra
$\overline{\AA}{}^{\flop}=\AA(\La^{\flop})/\PP(\La^{\flop})$
with $\Mat_{2}(\overline{\AA})$,
then the $\mathfrak{k}_\F$-algebra homomorphism 
$j'$ from $\mathfrak{k}$ to $\overline{\AA}{}^{\flop}$ 
induced by $x\mapsto uxu^{-1}$ is given by:
\begin{equation*}
x\mapsto
\begin{pmatrix}
x&0\\
0&j(x)\\
\end{pmatrix}.
\end{equation*}
Therefore, $j^{\flop}(x)$ is central in $\overline{\AA}{}^{\flop}$ 
if and only if $x=j(x)$ is central in $\overline{\AA}$, that is, 
if and only if $x\in\mathfrak{k}_\F$.
\end{proof}

This leads us to the following result.
For $i=1,2$, let $(k,\b_i)$ be a simple pair over $\F$,
let $[\La,n_i,m_i,\h_i(\b_i)]$ be a realization of $(k,\b_i)$
in $\A$ and let $\t_i\in\Cc(\La,m_i,\h_i(\b_i))$ be a simple 
character.

\begin{prop}
\label{Hengist}
Assume $\t_1$ and $\t_2$ intertwine in $\mult\A$. 
Then there is a simple central $\F$-algebra $\A'$ together with 
realizations $[\La',n^{}_i,m^{}_i,\h'_i(\b^{}_i)]$ of $(k,\b_i)$ 
in $\A'$ (with the same $n_i$ and $m_i$), with 
$i=1,2$, which are sound and have the same embedding type, and such 
that $\t_1'$ and $\t_2'$ intertwine in $\A^{\prime\times}$, where
$\t_i'\in\Cc(\La',m^{}_i,\h_i'(\b_i^{}))$ denotes the transfer of 
$\t_i$.
\end{prop}

\begin{proof}
First, we reduce to the case where the strata $[\La,n_i,m_i,\h_i(\b_i)]$ 
have Fr\"ohlich in\-variant $1$.
Let $g\in\A^{\times}$ intertwine the characters $\t_1$ and $\t_2$ 
as in (\ref{Eaque}).
We set $\La'=\La\oplus\La$ and $\A'=\Mat_2(\A)$ and, 
for each $i$, we fix an element $u_i\in\A^{\prime\times}$ as in Lemma 
\ref{Blaise3} so that the simple stratum 
$[\La',n_i^{},m_i^{},u_i^{}\h_i^{}(\b_i^{})u_i^{-1}]$
has Fr\"ohlich invariant $1$.
For each $i$, let $\t_i'$ be the trans\-fer of $\t_i$ in 
$\Cc(\La',m_i,\h_i(\b_i))$, and let $\t_i''$ be that of $\t_i$ in 
$\Cc(\La',m_i,u_i^{}\h_i^{}(\b_i^{})u_i^{-1})$, 
which is equal to the 
conjugate character $x\mapsto\t_i'(u_i^{-1}xu_i^{})$.
By the proof of Proposition 
\ref{LongJohnSilverGeneralCarSimVariante}, 
the element $g'=\iota(g)\in\A^{\prime\times}$ 
intertwines $\t_1'$ and $\t_2'$,
where $\iota$ denotes the diagonal em\-bed\-ding of $\A$ in $\A'$, 
and it follows that 
$g''=u_1^{-1}g'u_2^{}$ intertwines 
$\t_1''$ and $\t_2''$.
Thus we can assume that the strata $[\La,n_i,m_i,\h_i(\b_i)]$ have 
Fr\"ohlich invariant $1$.
Using Proposition \ref{LongJohnSilverVarianteSound} (with some 
suitable integer $l\>1$) and 
Lemma \ref{Thomas} together, we see that the simple strata 
$[\La^{\ddag},n_i,m_i,\h_i(\b_i)]$
are sound with Fr\"ohlich invariant $1$. 
By Theorem \ref{FROH}, they have the same embedding type. 
Let $\t_i^{\ddag}$ be the trans\-fer of $\t_i$ in 
$\Cc(\La^{\ddag},m_i,\h_i(\b_i))$. 
The fact that $\t_1^{\ddag}$ and $\t_2^{\ddag}$ 
intertwine in $\A^{\ddag\times}$ follows from Proposition 
\ref{LongJohnSilverGeneralCarSimVariante}.
\end{proof}

\begin{rema}
\label{ColorBlue}
The assumption $[\F(\b_1):\F]=[\F(\b_2):\F]$ is not needed in the proof.
\end{rema}

%%%%%%%%%%%%%%%%%%%%%%%%%%%%%%%%%%%%%%%%%%%%%%%%%%%%%%%%%%%%%%%%%%%%%%%%%%%

\subsection{}

Before proving the first main result of this section, that is Proposition 
\ref{Hengist6}, we will need the following lemmas.
Compare the first one with Proposition \ref{Eponine}. 

\begin{lemm}
\label{Hectare1}
For $i=1,2$, let $(\Theta_i,k,\b_{i})$ be a ps-character over $\F$, 
and suppose that $\Theta_1$ and $\Theta_2$ are endo-equivalent.
Then $n_{\F}(\b_1)=n_{\F}(\b_2)$, $e_{\F}(\b_1)=e_{\F}(\b_2)$, 
$f_{\F}(\b_1)=f_{\F}(\b_2)$ and $k_{\F}(\b_1)=k_{\F}(\b_2)$.
\end{lemm}

\begin{proof}
By assumption, we have $[\F(\b_1):\F]=[\F(\b_2):\F]$ and
there is a simple central $\F$-algebra $\A$ together with 
realizations $[\La,n_i,m_i,\b_i]$ of $(k,\b_i)$, for 
$i=1,2$, such that the corresponding simple characters 
$\t_1$ and $\t_2$ inter\-twi\-ne in $\mult\A$.
By Proposition \ref{Hengist}, we can assume that these 
realizations are sound and have the same embedding type.
We now follow the proof of \cite[Pro\-po\-sition 8.4]{BH}.  
An argument similar to the first part of this proof 
(which we do not reproduce) gives us $n_1=n_2$, denoted $n$. 
Now consider the integers $m_1,m_2$. 
By symmetry, we can assume that $m_1\>m_2$.
Let us choose 
a simple stratum $[\La,n,m_1,\g]$ in $\A$ which is equivalent to 
$[\La,n,m_1,\b_2]$ and let $\t_0$ denote the restriction of $\t_2$ to 
$\H^{m_1+1}(\g,\La)$. 
The characters $\t_0$ and $\t_1$ still intertwine, which implies, 
by the ``inter\-twi\-ning implies conjugacy'' theorem
\cite[Corollary 10.15]{Gr}, the existence of $u\in\KK(\La)$ such that
$\Cc(\La,m_1,\b_1)=\Cc(\La,m_1,u\g u^{-1})$.
By Proposition \ref{Gr9199}, we get:
\begin{equation}
\label{Galaad}
k_\F(\b_1)=k_\F(\g),
\quad
[\F(\b_1):\F]=[\F(\g):\F].
\end{equation}
By \cite[Theorem 5.1(ii)]{BG}, the equality
$[\F(\b_2):\F]=[\F(\g):\F]$
implies that $[\La,n,m_1,\b_2]$ is a simple stratum in $\A$.
By Theorem \ref{Grabinoulor}, we get $e_{\F}(\b_1)=e_{\F}(\b_2)$ and 
$f_{\F}(\b_1)=f_{\F}(\b_2)$, 
and (\ref{Galaad}) gives us $k_\F(\b_1)=k_\F(\b_2)$.
The remaining equality is a consequence of the identity 
$n_i=e_{\b_i}(\La)n_{\F}(\b_i)$. 
\end{proof}

\begin{coro}
\label{Psaume} 
Theorem \ref{Gata} implies Theorems \ref{EndoClasChar} 
and \ref{TrelawneyIIC}.
\end{coro}

\begin{proof}
For $i=1,2$, let $(\Theta_{i},k,\b_{i})$ be a ps-character over $\F$, 
and suppose that $\Theta_{1}$ and $\Theta_{2}$ are endo-equivalent. 
Let $\A$ be a simple central $\F$-algebra.
For each $i$, let $[\La,n_i,m_i,\h_i(\b_i)]$ be a realization of $(k,\b_i)$ 
in $\A$, and put $\t_i=\Theta_i(\La,m_i,\h_i)$.
Write $n=n_i$ and:
\begin{equation*}
m={e_{\h_i(\b_i)}(\La)}k,
\end{equation*}
which do not depend on $i$ by Lemma \ref{Hectare1}. 
As $m_1,m_2\>m$, 
we may assume without loss of generality that $m_1=m_2=m$.
Let us fix an $\F$-algebra homomorphism $\h_3:\F(\b_2)\to\A$ 
such that the simple strata $[\La,n,m,\h_1(\b_1)]$ and 
$[\La,n,m,\h_3(\b_2)]$ have the same embedding type, 
and let $\t_3$ denote the transfer of $\t_2$ in $\Cc(\La,m,\h_3(\b_2))$.
According to Theorem \ref{Gata}, there is an element $u\in\KK(\La)$ 
such that 
$\t_3(x)=\t_1(uxu^{-1})$ for all $x\in\H^{m+1}(\h_3(\b_2),\La)$ and, 
by the Skolem-Noether 
theorem, there is an element $g\in\A^{\times}$
such that $\h_3(x)=g\h_2(x)g^{-1}$. 
Thus $g$ intertwines $\t_3$ and $\t_2$,
which proves that $\t_1$ and $\t_2$ intertwine in $\A^{\times}$ and ends 
the proof of Theorem \ref{EndoClasChar}. 

Assume now that the strata $[\La,n,m,\h_i(\b_i)]$, $i=1,2$ have the 
same embedding type. 
Then applying Theorem \ref{Gata} gives immediately Theorem \ref{TrelawneyIIC}.  
\end{proof}

We are thus reduced to proving 
Theorem \ref{Gata}, which will be done in section \ref{Sec8}. 
For this we will have to develop base change methods (see sections 
\ref{PILSC}, \ref{ILT} and \ref{BCASDFSC}).
We now state and prove the first main result of this section. 

\begin{prop}
\label{Hengist6}
For $i=1,2$, let $(\Theta_{i},k,\b_{i})$ be a 
ps-character over $\F$, and suppose that 
$\Theta_{1}$ and $\Theta_{2}$ are endo-equivalent. 
Write $\K_i$ for the maximal un\-ramified extension of $\F$ contained
in $\F(\b_i)$.
Then there exists a simple central $\F$-algebra $\A$ together with 
reali\-za\-tions $[\La,n,m,\h_i(\b_i)]$ of $(k,\b_i)$, for $i=1,2$, 
which are sound and have the same embedding type, and such that:
\begin{enumerate}
\item 
$m$ is a multiple of $k$;
\item
$\h_1(\K_1)=\h_2(\K_2)$;
\item
$\Theta_1(\La,m,\h_1)=\Theta_2(\La,m,\h_2)$.
\end{enumerate}
\end{prop}

\begin{proof}
By Proposition \ref{Hengist}, 
there is a simple central $\F$-algebra $\A$ together with 
realizations $[\La,n_i,m_i,\h_i(\b_i)]$ of $(k,\b_i)$, for 
$i=1,2$, sound and having the same embedding type, 
such that $\t_1=\Theta_1(\La,m_1,\h_1)$ and 
$\t_2=\Theta_2(\La,m_2,\h_2)$ inter\-twi\-ne in $\mult\A$.
By Lemma \ref{Hectare1}, we have $n_1=n_2$,
and the integer $m={e_{\h_i(\b_i)}(\La)}k$ does not depend on $i$.

\begin{lemm}
\label{Hectare2}
For each $i$, there exists a unique 
$\vartheta_i\in\Cc(\La,m,\h_i(\b_i))$ extending $\t_i$,
and the characters $\vartheta_1$ and $\vartheta_2$ intertwine 
in $\mult\A$.
\end{lemm}

\begin{proof}
The proof is similar to that of \cite[Lemma 3.6.7]{BK} 
and \cite[Lemma 8.5]{BH} together.
One just has to replace Corollary $3.3.21$ of \cite{BK} by 
Proposition $2.16$ of \cite{SeSt}, and Theorems $3.5.8$, 
$3.5.9$ and $3.5.11$ of \cite{BK} by Corollary $10.15$ and Propositions 
$9.9$ and $9.10$ of \cite{Gr}.
\end{proof}

Therefore we can assume that $m_1,m_2$ are both equal to $m$.
The result now follows from 
the ``inter\-twi\-ning implies conjugacy'' theorem
\cite[Corollary $10.15$]{Gr}.
\end{proof}

%%%%%%%%%%%%%%%%%%%%%%%%%%%%%%%%%%%%%%%%%%%%%%%%%%%%%%%%%%%%%%%%%%%%%%%%%%%

\subsection{}
\label{GraalFiction}

We now assume that we are in the situation of paragraph \ref{Miskatonic}. 
Let us fix two simple strata $[\La,n,m,\b_i]$, $i=1,2$, in $\A$.
We set $n'=an$ and fix a non-negative integer $m'$ such that 
$\lfloor{m'/a}\rfloor=m$, so that we have simple strata 
$[\La',n',m',\b_i]$, $i=1,2$, in ${\A}'$, where $\La'$ is defined 
by (\ref{SpadassinVariante}).
We fix a simple character $\t_i$ in $\Cc(\La,m,\b_i)$ 
and write $\t_i'$ for its transfer in $\Cc(\La',m',\b_i)$. 
The aim of this paragraph is to prove the following proposition,
which is the second main result of this section. 

\begin{prop}
\label{Dolomites}
Assume that $\t_1$ and $\t_2$ are equal. 
Then $\t^{\prime}_1$ and $\t^{\prime}_2$ are equal.
\end{prop}

\begin{proof}
We first prove the following lemma, which generalizes 
\cite[Theorem~3.5.9]{BK} and \cite[Proposition~9.10]{Gr} 
(see also \cite[Lemme~7.9]{Dat}, which gives a similar result 
in the split case for \emph{semisimple characters} and whose 
proof we follow). 

\begin{lemm}
\label{BK359}
Assume that $m\>1$, and that $\Cc(\La,m,\b_1)\cap\Cc(\La,m,\b_2)$ 
is not empty. 
Then we have $\H^m(\b_1,\La)=\H^m(\b_2,\La)$.
\end{lemm}

\begin{proof}
We put $\nu=2m-1$ and, for $i=1,2$, we choose a simple stratum 
$[\La,n,\nu,\g_i]$ equivalent to $[\La,n,\nu,\b_i]$ in $\A$. 
Then, for each $i=1,2$, we have $\Cc(\La,\nu,\b_i)=\Cc(\La,\nu,\g_i)$ 
and, from~\cite[Proposition~2.15]{SeSt}, 
we have $\H^m(\b_i,\La)=\H^m(\g_i,\La)$. 
Since the restriction of a simple character to
$\H^{\nu+1}(\b_1,\La)=\H^{\nu+1}(\b_2,\La)$ is still a simple character, 
the intersection $\Cc(\La,\nu,\g_1)\cap\Cc(\La,\nu,\g_2)$ is not empty. 
By computing the intertwining of an element of this intersection 
via the formula
of~\cite[Th\'eor\`eme~2.23]{SeSt}, we get:
\begin{equation*}
\Om_{q_1-\nu}(\g_1,\La)\B_{\g_1}^\times\Om_{q_1-\nu}(\g_1,\La) = 
\Om_{q_2-\nu}(\g_2,\La)\B_{\g_2}^\times\Om_{q_2-\nu}(\g_2,\La)
\end{equation*}
with the notations of \emph{loc. cit.} and where, for each 
$i=1,2$, we write 
$\B_{\g_i}$ for the centralizer of $\F(\g_i)$ in $\A$ and 
$q_i=-k_0(\g_i,\La)$.
Taking the intersection with $\PP_m(\La)$ and then its ad\-di\-tive 
closure, we find that the following set:
\begin{equation} 
\label{horribleset}
\PB_m^i+
\left(\PP_{q_i-\nu}(\La)\cap\mathfrak n_{-\nu}(\g_i,\La)\right)\PB_m^i+
\PB_m^i \JJ^{\lceil q_i/2\rceil}(\g_i,\La),
\end{equation}
is independent of $i$, 
where we have put $\PB_m^i=\PP_m(\La)\cap\B_{\g_i}$
and where the notations $\JJ^k$ and $\HH^k$, for $k\>0$, 
are defined in \cite[\S2.4]{SeSt}.
We claim that the set in \eqref{horribleset}
is contained in $\HH^m(\g_i,\La)=\HH^m(\b_i,\La)$.
For then, adding
$\HH^{m+1}(\g_i,\La)=\HH^{m+1}(\b_i,\La)$, which is also in\-de\-pen\-dent of 
$i$, we see that:
\begin{equation*}
\HH^m(\b_i,\La)=\HH^m(\g_i,\La)=\PB_m^i+\HH^{m+1}(\g_i,\La)
\end{equation*}
is independent of $i$, as required.
We need the following lemma (see \cite[Lemma~3.11(i)]{St4}).

\begin{lemm}
\label{awkwardlemma}
Let $[\La,n,m,\b]$ be a simple stratum in $\A$ with $q=-k_0(\b,\La)$. 
For each integer $1\<k\<q-1$, we have:
\begin{equation*}
\left(\mathfrak n_{-k}(\b,\La)\cap \PP_{q-k}(\La)\right) 
\JJ^{\lceil k/2\rceil}(\b,\La) \subseteq 
\HH^{\lfloor k/2\rfloor+1}(\b,\La).
\end{equation*}
\end{lemm}

\begin{proof}
We write $[\widetilde\La,n,m,\b]$ for the simple stratum in 
$\widetilde\A=\End_\F(\V)$ associated with $[\La,n,m,\b]$
(see paragraph \ref{Split}).
Then we have:
\begin{equation*}
\left(\mathfrak{n}_{-k}(\b,\widetilde{\La})
\cap\PP_{q-k}(\widetilde{\La})\right) 
\JJ^{\lceil k/2\rceil}(\b,\widetilde{\La}) \\
\subseteq\HH^{\lfloor k/2\rfloor+1}(\b,\widetilde{\La})
\end{equation*}
by \cite[Lemma~3.11(i)]{St4}.
By taking the intersection with $\A$, we get the expected result. 
\end{proof}

We now see that:
\begin{equation*}
\left(\PP_{q_i-\nu}(\La)\cap\mathfrak n_{-\nu}(\g_i,\La)\right)\PB_m^i 
\subseteq\left(\PP_{q_i-\nu}(\La)\cap\mathfrak n_{-\nu}(\g_i,\La)\right)
\JJ^{\lceil \nu/2\rceil}(\g_i,\La) 
\subseteq\HH^m(\g_i,\La).
\end{equation*}
Similarly, we have: 
\begin{equation*}
\PB_m^i \JJ^{\lceil q_i/2\rceil}(\g_i,\La)\subseteq
\left(\PP_{q_i-(q_i-m)}(\La)\cap\mathfrak n_{m-q_i}(\g_i,\La)\right)
\JJ^{\lceil (q_i-m)/2\rceil}(\g_i,\La) 
\subseteq\HH^{\lfloor (q_i-m)/2\rfloor+1}(\g_i,\La).
\end{equation*}
Since the left hand side here is clearly also contained in $\PP_m(\La)$, 
we see that it is 
contained in $\HH^m(\g_i,\La)$ as required. 
This also completes the proof of Lemma~\ref{BK359}.
\end{proof}

For each $i$, write $\Theta_i$ for the ps-character defined by the pair 
$([\La,n,m,\b_i],\t_i)$, and recall that $\t_1$ and $\t_2$ are equal. 

\begin{lemm}
\label{NonNominatus}
We have $e_{\F}(\b_1)=e_{\F}(\b_2)$ and $f_{\F}(\b_1)=f_{\F}(\b_2)$.
\end{lemm}

\begin{proof}
By Proposition \ref{Hengist}, there is a simple 
central $\F$-algebra $\A$ together with realizations 
$[\La^{0},n,m,\h^{0}_i(\b_i^{})]$ of $(k,\b_i)$, with $i=1,2$, 
which are sound and have the same em\-bedding type, and such 
that $\Theta_1(\La^{0},m,\h^{0}_1)$ and 
$\Theta_2(\La^{0},m,\h^{0}_2)$ intertwine in $\A^{0\times}$.
Let us write $f$ for the greatest common divisor of 
$f_{\F}(\b_1)$ and $f_{\F}(\b_2)$ and $\K_i$ for the maximal 
unramified extension of $\F$ contained in $\F(\h^{0}_i(\b_i^{}))$.
Then Theorem \ref{Grabinoulor} gives us the ex\-pec\-ted equality.
\end{proof}

Thus the ps-characters $\Theta_1$ and $\Theta_2$ are endo-equivalent,
which allows us to use Lemma \ref{Hectare1}. 

\begin{lemm}
\label{Jaime}
The characters $\t^{\prime}_1$ and $\t^{\prime}_2$ are equal if
and only if we have:
\begin{equation}
\label{Mesclum}
\H^{m'+1}(\b_1,\La^{\prime})=\H^{m'+1}(\b_2,\La^{\prime}).
\end{equation}
\end{lemm}

\begin{proof}
This follows immediately from Lemma \ref{JacquesVerges}.
\end{proof}

Thus we are reduced to proving equality (\ref{Mesclum}),
and for this, we claim that it is enough to prove that:
\begin{equation}
\label{Rosso2}
\H^{q'}(\b_1,\La^{\prime})=\H^{q'}(\b_2,\La^{\prime}),
\end{equation}
where $q'=-k_0(\b_i,\La')$ is independent of $i$
by Lemma \ref{Hectare1}.
Indeed, assume that (\ref{Rosso2}) holds, and let $t'$
be the smallest integer in $\{m',\dots,q'-1\}$ such that:
\begin{equation}
\label{StServin}
\H^{t'+1}(\b_1,\La^{\prime})=\H^{t'+1}(\b_2,\La^{\prime}).
\end{equation} 
Suppose that $t'\neq m'$.
By Lemma \ref{Jaime}, the characters $\t^{\prime}_1$ and 
$\t^{\prime}_2$ agree on (\ref{StServin}), that is, the intersection 
$\Cc(\La',t',\b_{1})\cap\Cc(\La',t',\b_{2})$ is not empty. 
By Lemma \ref{BK359}, we get an equality which 
contradicts the minimality of $t'$.
Hence $t'=m'$ and we are thus reduced to proving (\ref{Rosso2}), 
which we do by induction on $\b_1$.
Assume first that $\b_1$ is minimal over $\F$. 
Then so is $\b_2$ by Lemma \ref{Hectare1}, so that we have:
\begin{equation*}
\H^{q'}(\b_1,\La^{\prime})=\U_{q'}(\La')=\H^{q'}(\b_2,\La^{\prime}).
\end{equation*}
Assume now that $\b_1$ is not minimal over $\F$, 
set $q=-k_0(\b_i,\La)$, 
which is independent of $i$ by Lemma \ref{Hectare1}, 
and choose a simple stratum $[\La,n,q,\g_i]$ in $\A$ 
equivalent to the stratum $[\La,n,q,\b_i]$, for each $i\in\{1,2\}$.
We then have:
\begin{equation*}
\H^{q'}(\b_i,\La^{\prime})=\H^{q'}(\g_i,\La'),
\end{equation*}
and the restriction
$\vartheta_i=\t_i\ |\ \H^{q+1}(\g_i,\La)$ 
belongs to $\Cc(\La,q,\g_i)$.
As $\b_i-\g_i\in\PP_{-q}(\La)$, the simple characters 
$\vartheta_1$ and $\vartheta_2$ are equal.
If we write $\vartheta'_i$ for the transfer of $\vartheta_i$ to the 
set $\Cc(\La',q',\g_i)$, then the inductive hypothesis implies that 
$\vartheta'_1=\vartheta'_2$. 
Therefore, the intersection 
$\Cc(\La',q',\g_{1})\cap\Cc(\La',q',\g_{2})$ is not empty, and 
Lemma \ref{BK359} gives us the required equality (\ref{Rosso2}). 
This ends the proof of Proposition \ref{Dolomites}.
\end{proof}

%%%%%%%%%%%%%%%%%%%%%%%%%%%%%%%%%%%%%%%%%%%%%%%%%%%%%%%%%%%%%%%%%%%%%%%%%%%

\subsection{}
\label{Mescalito}

Before closing this section, we prove
the following rigidity theorem for simple characters, which generalizes 
\cite[Theorem 3.5.8]{BK} and \cite[Proposition 9.9]{Gr}
to simple characters in non-necessarily split simple central $\F$-algebras 
with non-necessarily strict lattice sequences.

\begin{theo}
\label{Paraclet}
For $i=1,2$, let $[\La,n,m,\b_{i}]$ be a simple stratum in a simple 
central $\F$-alg\-ebra $\A$.
Assume that the intersection $\Cc(\La,m,\b_{1})\cap\Cc(\La,m,\b_{2})$ 
is not empty. 
Then we have $\Cc(\La,m,\b_{1})=\Cc(\La,m,\b_{2})$.
\end{theo}

\begin{proof}
For each $i\in\{1,2\}$, we fix a simple character
$\t_i\in\Cc(\La,m,\b_{i})$ and assume that $\t_1$ and $\t_2$ are equal.
In particular, we have:
\begin{equation}
\label{Meka}
\H^{m+1}(\b_1,\La)=\H^{m+1}(\b_2,\La).
\end{equation}
By choosing an integer $l$ as in Proposition 
\ref{LongJohnSilverVarianteSound}, 
we have sound simple strata $[\La^{\ddag},n,m,\b_{i}]$, $i=1,2$, in 
$\A^{\ddag}$. 
If we write $\t^{\ddag}_i$ for the transfer of $\t_i$ to 
$\Cc(\La^{\ddag},m,\b_i)$, then it follows from Proposition \ref{Dolomites}
that the simple characters $\t_1^{\ddag}$ and 
$\t_2^{\ddag}$ are equal, hence that the intersection
$\Cc(\La^{\ddag},m,\b_{1})\cap\Cc(\La^{\ddag},m,\b_{2})$ is not empty.
By Proposition \ref{Gr9199},
the sets $\Cc(\La^{\ddag},m,\b_{i})$, $i=1,2$, are equal.
As the transfer map from $\Cc(\La^{\ddag},m,\b_{i})$ to 
$\Cc(\La,m,\b_{i})$ is the restriction map from 
$\H^{m+1}(\b_i,\La^{\ddag})$ to $\H^{m+1}(\b_i,\La)$, the equality 
(\ref{Meka}) implies that $\Cc(\La,m,\b_{1})=\Cc(\La,m,\b_{2})$.
\end{proof}

It is natural to ask whether the simple strata $[\La,n,m,\b_{i}]$
in Theorem \ref{Paraclet} have the same embedding type. 
We have the following conjecture. 

\begin{conj}
\label{combo}
For $i=1,2$, let $[\La,n,m,\b_{i}]$ be a simple stratum in a simple 
central $\F$-alg\-ebra $\A$.
Assume that the intersection $\Cc(\La,m,\b_{1})\cap\Cc(\La,m,\b_{2})$
is not empty, and that $\La$ is strict.
Then these simple strata have the same embedding type.
\end{conj}

Note that we know from \cite[Lemma 5.2]{BG} that two equivalent simple 
strata (with respect to a strict lattice sequence) have the same embedding 
type. 

\medskip

In the case where the strata are sound, we will 
prove below that this conjecture is true. 
First we need a series of lemmas. 

\begin{lemm}
\label{PerRinc}
Let $\E/\F$ be a finite extension with ramification index $e$, 
contained in a simple central $\F$-algebra $\A$, 
and let $\mathfrak{B}$ be a principal $\Oo_\E$-order of 
period $r$ in the centralizer $\B$ of $\E$ in $\A$.
Write $\A\simeq\Mat_{k}(\D)$ for some $k\>1$ and some 
$\F$-division algebra $\D$, and write $d$ for the reduced 
degree of $\D$ over $\F$.
\begin{enumerate}
\item 
There exists a unique $\E$-pure hereditary $\Oo_\F$-order 
$\AA$ in $\A$ such that $\mathfrak{B}=\AA\cap\B$ and 
$\KK(\mathfrak{B})=\KK(\AA)\cap\mult\B$, and such an order
is principal.
\item
The period of $\AA$ is equal to
${re}/{(re,d)}$, where $(re,d)$ denotes the greatest common 
divisor of $re$ and $d$.
\end{enumerate}
\end{lemm}

\begin{proof}
The first part is given by \cite[Corollary 1.4(ii)]{Gr1}.
Part (2) follows for instance from the 
formula given in the proof of \cite[Th\'eor\`eme 1.7]{SeSt}.
\end{proof}

In other words, there exists a unique hereditary $\Oo_\F$-order 
$\AA$ in $\A$ such that $\AA\cap\B=\mathfrak{B}$ and that 
$(\E,\AA)$ is a sound embedding in $\A$.

\begin{lemm} 
\label{Boubinette}
For $i=1,2$, let $\E_i$ be an extension of $\F$ 
contained in $\A$ and let $\AA$ be a hereditary 
$\Oo_\F$-order in $\A$ such that $(\E_i,\AA)$ is 
a sound embedding in $\A$.
Write $\BB_i$ for the intersection of $\AA$ with 
the centralizer of $\E_i$ in $\A$.
Let $f$ be the greatest common divisor of 
$f(\E_1:\F)$ and $f(\E_2:\F)$, and for each $i$, 
let $\K_i$ be the un\-ramified extension of $\F$ 
of degree $f$ contained in $\E_i$. 
Assume $\E_1$ and $\E_2$ have the same 
ramification order $e$ and $\BB_1$ and $\BB_2$
have the same period $r$.
Then the embeddings $(\K_1,\AA)$ and $(\K_2,\AA)$ 
are equivalent in $\A$.
\end{lemm}

\begin{proof}
Let $\CC_i$ denote the intersection of $\AA$ 
with the centralizer $\C_i$ of $\K_i$ in $\A$.
If we write $\B_i$ for the centralizer of $\E_i$ 
in $\A$, then we have $\BB_i=\CC_i\cap\B_i$ and 
$\KK(\mathfrak{B_i})=\KK(\CC_i)\cap\mult\B_i$.
Using Lemma \ref{PerRinc}, the period of $\CC_i$ 
is equal to $re/(re,d_i)$, where $d_i$ is the 
reduced degree of the $\K_i$-division algebra 
$\D_i$ such that $\C_i$ is isomorphic to 
$\Mat_{k_i}(\D_i)$ for some $k_i\>1$.
Using for instance \cite{Zi1}, we have 
$d_i=d/(d,f)$, which does not depend on $i$.
By the Skolem-Noether theorem, there is 
$g\in\mult\A$ such that $g\K_1g^{-1}=\K_2$. 
Thus $g\CC_1g^{-1}$ and $\CC_2$ are two principal $\Oo_{\K_2}$-orders 
in $\C_2$ with the same period, which implies that there exists
$h\in\mult\C_2$ such that $g\CC_1g^{-1}=h\CC_2h^{-1}$. 
Let us write $u=h^{-1}g$. 
Using the unicity property (1) of Lemma \ref{PerRinc},
we get $\AA=u^{-1}\AA u$, that is $u\in\KK(\AA)$.
\end{proof}

We now prove Conjecture \ref{combo} in the case 
where the strata are sound.

\begin{prop}
\label{BrattleStreet}
For $i=1,2$, let $[\La,n,m,\b_{i}]$ be a sound simple 
stratum in a simple central $\F$-alg\-ebra $\A$.
Assume that the intersection $\Cc(\La,m,\b_{1})\cap\Cc(\La,m,\b_{2})$
is not empty.
Then these simple strata have the same embedding type.
\end{prop}

\begin{proof}
For each $i$, we fix a simple character
$\t_i\in\Cc(\La,m,\b_{i})$ and assume 
$\t_1$ and $\t_2$ are equal.
Thus we have $[\F(\b_1):\F]=[\F(\b_2):\F]$ 
by Proposition \ref{Gr9199}.
By Lemma \ref{Hectare1}, we also have 
$e_{\F}(\b_1)=e_{\F}(\b_2)$ and $f_{\F}(\b_1)=f_{\F}(\b_2)$.
If we write $\I_{\U(\La)}(\t_i)$ for the intertwining 
of $\t_i$ in $\U(\La)$, then \cite[Th\'eor\`eme 3.50]{VS1} 
gives us: 
\begin{equation*}
\I_{\U(\La)}(\t_i)\U^1(\La)/\U^1(\La)\simeq\U(\BB_i)/\U^1(\BB_i),
\end{equation*}
where $\BB_i$ is the intersection of $\AA=\AA(\La)$ 
with the centralizer of $\b_i$ in $\A$.
As the stratum $[\La,n,m,\b_{i}]$ is sound, 
$\BB_i$ is a principal $\Oo_{\F(\b_i)}$-order. 
Thus there are a finite extension $\mathfrak{k}_i$ 
of $\mathfrak{k}_{\F}$ and two positive integers 
$r_i,s_i\>1$ such that:
\begin{equation*}
\U(\BB_i)/\U^1(\BB_i)
\simeq
\GL_{s_i}(\mathfrak{k_i})^{r_i}.
\end{equation*}
Since it does not depend on $i$, we have $r_1=r_2$.
Now write $\K_i$ for the maximal unramified extension of $\F$ 
contained in $\F(\b_i)$. 
Using Lemma \ref{Boubinette} with $\E_i=\F(\b_i)$, 
we deduce that the embeddings $(\K_1,\AA)$ and $(\K_2,\AA)$ 
are equivalent in $\A$.
\end{proof}

%%%%%%%%%%%%%%%%%%%%%%%%%%%%%%%%%%%%%%%%%%%%%%%%%%%%%%%%%%%%%%%%%%%%%%%%%%%
%%%%%%%%%%%%%%%%%%%%%%%%%%%%%%%%%%%%%%%%%%%%%%%%%%%%%%%%%%%%%%%%%%%%%%%%%%%

\section{The interior lifting}
\label{PILSC}

In this section, we develop an interior lifting process for simple 
strata and characters with respect to a finite unramified extension 
$\K$ of $\F$, in a way  similar to \cite{BH} and \cite{Gr}.
The situation in \cite{BH} is somewhat more general than ours, since the 
authors only assume $\K/\F$ to be tamely ramified, but is only concerned 
with simple strata and characters in split sim\-ple central $\F$-algebras 
attached to strict lattice sequences.
The situation in \cite{Gr} deals with any simple central 
$\F$-algebra, but puts an unnecessarily restrictive condition on the 
simple strata (they are supposed to be sound).

%%%%%%%%%%%%%%%%%%%%%%%%%%%%%%%%%%%%%%%%%%%%%%%%%%%%%%%%%%%%%%%%%%%%%%%%%%%

\subsection{}
\label{BelleDeJour}

Let $\A$ be a simple central $\F$-algebra and $\K/\F$ be a finite unramified 
extension cont\-ain\-ed in $\A$.
Let $\C$ denote the centralizer of $\K$ in $\A$.
We fix a simple left $\A$-module $\V$ and a simple left $\C$-module $\W$. 
The following definition extends \cite[Definition 2.2]{BH} to strata
with non-necessarily strict lattice sequences.

\begin{defi}
\label{KPure}
A stratum $[\La,n,m,\b]$ in $\A$ is said to be \textit{$\K$-pure} if it 
is pure, if $\b$ centralizes $\K$ and if the algebra $\K[\b]$ is a field 
such that $\K[\b]^{\times}$ normalizes $\La$.
\end{defi}

Given a $\K$-pure stratum $[\La,n,m,\b]$ in $\A$, we can form the pure 
stratum $[\Ga,n,m,\b]$, where $\Ga$ is the unique (up to translation) 
lattice sequence on $\W$ defined by:
\begin{equation}
\label{Edescen}
\aa_{k}(\La)\cap\C=\aa_{k}(\Ga), 
\quad k\in\ZZ.
\end{equation}
Note that the $\mult\C$-normalizer of $\Ga$ is equal to 
$\KK(\La)\cap\mult\C$.
We then get a process:
\begin{equation}
\label{Dulilah}
[\La,n,m,\b]\mapsto[\Ga,n,m,\b]
\end{equation}
giving an injection, respecting equivalence, between the set of 
$\K$-pure strata of $\A$ and the set of pure strata of $\C$. 
The fact that $\Ga$ is defined only up to translation makes (\ref{Dulilah}) 
not well defined, but this will be of no importance in the sequel.
We now discuss the image of simple $\K$-pure strata of $\A$ by
(\ref{Dulilah}). 

\begin{prop}
\label{Alexandre1}
\begin{enumerate}
\item
Let $[\La,n,m,\b]$ be a $\K$-pure stratum in $\A$. 
Then:
\begin{equation}
\label{peccavi}
k_0(\b,\Ga)\<k_0(\b,\La).
\end{equation}
\item
Suppose moreover that $[\La,n,m,\b]$ is simple.
Then the stratum $[\Ga,n,m,\b]$ given by the map 
(\ref{Dulilah}) is simple. 
\end{enumerate}
\end{prop}

\begin{proof}
As $\K$ is unramified over $\F$, the lattice sequences $\La$ and $\Ga$
have the same period over $\Oo_\F$.
By (\ref{Lakme}) it is then enough to prove that $k_{\K}(\b)\<k_{\F}(\b)$.
Let $\mathfrak{L}$ denote the strict $\Oo_{\F}$-lattice sequence on
$\K(\b)$ defined by $i\mapsto\p_{\K(\b)}^{i}$.
By \cite[Theorem 2.4]{BH}, we have:
\begin{equation*}
k_{\K}(\b)\<k_0(\b,\mathfrak{L}).
\end{equation*}
On the other hand, we have $e_{\b}(\mathfrak{L})=1$ as $\K$ is
unramified over $\F$.
By (\ref{Lakme}) again, we get the expected result.
Suppose now that the stratum $[\La,n,m,\b]$ is simple.
Then the fact that $[\Ga,n,m,\b]$ is simple derives immediately from 
(\ref{peccavi}).
\end{proof}

\begin{rema}
For a case where the map (\ref{Dulilah}) is not surjective, see 
\cite[Exemple 1.6]{SeSt}.
Compare with the split case \cite[(2.3)]{BH}.
\end{rema}

%%%%%%%%%%%%%%%%%%%%%%%%%%%%%%%%%%%%%%%%%%%%%%%%%%%%%%%%%%%%%%%%%%%%%%%%%%%

\subsection{}

Given a simple stratum $[\Ga,n,m,\b]$ in $\C$ in the image of 
(\ref{Dulilah}), the $\K$-pure stratum of $\A$ corresponding 
to it may not be simple. 
However, we have the following result, which generalizes 
\cite[Corollary 3.8]{BH}.

\begin{prop}
\label{KaPureApprox}
Let $[\La,n,m,\b]$ be a $\K$-pure stratum in $\A$ such that 
$[\Ga,n,m,\b]$ is simple.
Then there exists a simple stratum $[\Ga,n,m,\b']$ in $\C$ equivalent 
to $[\Ga,n,m,\b]$ such that the stratum $[\La,n,m,\b']$ is simple.

Moreover, $\b'$ can be chosen such that the maximal unramified extension 
of $\F$ contained in $\F(\b')$ is contained in that of $\F(\b)$.
\end{prop}

\begin{proof}
Let $(k,\b)$ denote the simple pair over $\K$ of which $[\Ga,n,m,\b]$ 
is a realization, fix a simple right $\K(\b)\otimes_{\F}\D$-module $\SS$ 
and set $\A(\SS)=\End_\D(\SS)$.
Write $\rho$ for the natural $\K$-algebra homomorphism from $\K(\b)$ 
to $\A(\SS)$.
Let $\Ss$ denote the unique (up to translation) $\rho(\K(\b))$-pure 
strict $\Oo_\D$-lattice sequence on $\SS$ and $n_0$ the $\Ss$-valuation 
of $\rho(\b)$, and set: 
\begin{equation*}
m_0=e_{\rho(\b)}(\Ss)k,
\end{equation*}
so that $[\Ss,n_0,m_0,\rho(\b)]$ is a $\K$-pure stratum in $\A(\SS)$.
Write $\C(\SS)$ for the cent\-ralizer of $\K$ in $\A(\SS)$, fix a simple 
left $\C(\SS)$-module $\T$ and let $[\Tt,n_0,m_0,\rho(\b)]$
be the stratum in $\C(\SS)$ attached to $[\Ss,n_0,m_0,\rho(\b)]$ by 
(\ref{Dulilah}). 
This stratum is a realization of $(k,\b)$ in $\C(\SS)$, hence this is a 
simple stratum. 
According to \cite[Theorem 3.7]{BH}, the simple pair $(k,\b)$ is 
endo-equivalent to a simple pair $(k,\a)$ over $\K$ which is a 
$\K/\F$-lift of some simple pair over $\F$ in the sense of \cite{BH}
(see paragraph 3).
By \cite[Proposition 1.10]{BH}, the extensions $\K(\a)$ and $\K(\b)$ 
have the same ramification index and residue class degree over $\K$, 
which implies by \cite[Corollary 3.16]{BG} that there is a realization
$[\Tt,n_0,m_0,\h(\a)]$
of $(k,\a)$ in $\C(\SS)$, having the same embedding type as 
$[\Tt,n_0,m_0,\rho(\b)]$. 

\medskip

We now pass to the strata $[\widetilde\Tt,n_0,m_0,\h(\a)]$ and 
$[\widetilde\Tt,n_0,m_0,\rho(\b)]$ in the $\K$-algebra $\End_{\K}(\T)$ 
(see paragraph \ref{Split}).
By \cite{SeSt} (see Th\'eor\`eme 1.7 and Remarque 1.8), 
the lattice sequence $\Tt$ (and thus $\widetilde\Tt$) is in the affine 
class of a strict lattice sequence, so that, up to renormalization, 
one may consider it as being strict (see Lemma \ref{Bachir}). 
By \cite[Proposition 1.10]{BH}, these strata thus intertwine.
Hence, using Proposition \ref{BG413}, we can replace $\h$ by some 
$\KK(\Tt)$-con\-ju\-ga\-te and assume that the strata 
$[\Tt,n_0,m_0,\h(\a)]$ and $[\Tt,n_0,m_0,\rho(\b)]$ are equivalent,
and that the maximal unramified extension
of $\K$ contained in $\K(\h(\a))$ is equal to that of $\K(\rho(\b))$.
We check that the stratum $[\Ss,n_0,m_0,\h(\a)]$ is simple 
as in the proof of \cite[Proposition 4.3]{BH}.
We now fix a decomposition:
\begin{equation}
\label{Aron22}
\V=\V^{1}\oplus\cdots\oplus\V^{l}
\end{equation}
of $\V$ into simple right $\K(\b)\otimes_{\F}\D$-modules (which all are 
copies of $\SS$) such that $\La$ is de\-com\-posed by 
(\ref{Aron22}) in the sense of \cite[D\'efinition 1.13]{VS3}, that is, 
$\La$ is the direct sum of the lattice sequences 
$\La^{j}=\La\cap\V^{j}$, for $j\in\{1,\dots,l\}$.
By choosing, for each $j$, an iso\-mor\-phism of 
$\K(\b)\otimes_{\F}\D$-modules between $\SS$ and $\V^{j}$, 
this decomposition gives us an $\F$-algebra homomorphism:
\begin{equation*}
\iota:\A(\SS)\to\A.
\end{equation*}
Using Lemma \ref{Merimee2},
we may assume that this homomorphism 
satisfies $\iota(\rho(\b))=\b$.
If we set $\b'=\iota(\h(\a))$, then the stratum 
$[\Ga,n,m,\b']$ is simple and satisfies the conditions of the first part
of Proposition \ref{KaPureApprox}.

\medskip

In particular, the pure strata $[\La,n,m,\b]$ and $[\La,n,m,\b']$ 
are equivalent, and the second one is simple. 
By replacing the lattice sequence $\La$ by $\La^\dag$ (see paragraph 
\ref{Dague}), we can apply \cite[Theorem 5.1(ii)]{BG}
and thus get that $f_{\F}(\b')$ divides $f_{\F}(\b)$.
Moreover, the maximal unramified extension of $\K$ contained in $\K(\b')$
is equal to that of $\K(\b)$, denoted $\L$. 
As $\K/\F$ is unramified, the extension $\L/\F$ is unramified.
Thus the maximal unramified extension of $\F$ contained in $\F(\b')$ and 
that of $\F(\b)$ are two finite unramified extensions of $\F$ contained in 
$\L$. 
According to the condition on their degrees, it follows that the maximal 
unramified extension of $\F$ contained in $\F(\b')$ in contained in that 
of $\F(\b)$.
\end{proof}

%%%%%%%%%%%%%%%%%%%%%%%%%%%%%%%%%%%%%%%%%%%%%%%%%%%%%%%%%%%%%%%%%%%%%%%%%%%

\subsection{}

Let $[\La,n,m,\b]$ be a $\K$-pure simple stratum in $\A$, and let 
$[\Ga,n,m,\b]$ be the stratum in $\C$ given by the map 
(\ref{Dulilah}), which is simple by Proposition \ref{Alexandre1}.
Recall that one attaches to these simple strata compact 
open subgroups $\H^{m+1}(\b,\La)$ of $\mult\A$ and 
$\H^{m+1}(\b,\Ga)$ of $\mult\C$, respectively. 

\begin{prop}
\label{Alexandre}
Let $[\La,n,m,\b]$ be a $\K$-pure simple stratum in $\A$,
and let $[\Ga,n,m,\b]$ correspond to it by (\ref{Dulilah}).
Then we have:
\begin{equation*}
\H^{m+1}(\b,\La)\cap\mult\C=\H^{m+1}(\b,\Ga).
\end{equation*}
\end{prop}

\begin{proof}
It is enough to prove it when $m=0$.
The proof is by induction on $\b$.
Let $\R$ denote the centralizer of $\K(\b)$ in $\A$.
Assume first that $\b$ is minimal over $\F$, so that:
\begin{equation*}
\H^{1}(\b,\La)=(\U_{1}(\La)\cap\mult\B)\U_{\lfloor n/2\rfloor+1}(\La).
\end{equation*}
According to (\ref{Edescen}), we get:
\begin{equation*}
\H^{1}(\b,\La)\cap\mult\C=
(\U_{1}(\Ga)\cap\mult\R)\U_{\lfloor n/2\rfloor+1}(\Ga),
\end{equation*}
which is equal to $\H^{1}(\b,\Ga)$ as $\b$ is minimal over $\K$ 
by Proposition \ref{Alexandre1}.
Now assume that $\b$ is not minimal over $\F$, set $q=-k_0(\b,\La)$ 
and $r=\lfloor q/2\rfloor+1$, 
and choose a simple stratum $[\Ga,n,q,\g]$ equivalent to $[\Ga,n,q,\b]$ 
such that $[\La,n,q,\g]$ is simple and $\K$-pure, 
which is possible thanks to Proposition \ref{KaPureApprox}. 
We then have:
\begin{equation*}
\label{Pivoine}
\H^{1}(\b,\La)=(\U_{1}(\La)\cap\mult\B)\H^{r}(\g,\La)
\end{equation*}
and, if we set $q_1=-k_0(\b,\Ga)$ and $r_1=\lfloor q_1/2\rfloor+1$, we have:
\begin{equation*}
\label{Sterne}
\H^{1}(\b,\Ga)=(\U_{1}(\Ga)\cap\mult\R)\H^{r_1}(\g,\Ga).
\end{equation*}
As $-k_{0}(\g,\Ga)\>q_1\>q$, the group $\H^{r}(\g,\Ga)$ is equal to 
$(\U_{r}(\Ga)\cap\mult\R)\H^{r_1}(\g,\Ga)$.
It follows from (\ref{Edescen}) that the group $\H^{1}(\b,\Ga)$ is equal to 
the intersection $\H^{1}(\b,\La)\cap\mult\C$.
This ends the proof of Proposition \ref{Alexandre}.
\end{proof}

%%%%%%%%%%%%%%%%%%%%%%%%%%%%%%%%%%%%%%%%%%%%%%%%%%%%%%%%%%%%%%%%%%%%%%%%%%%

\subsection{}
\label{Fritouille}

We now want to lift simple characters. 
For this, given a simple stratum $[\La,n,m,\b]$ in $\A$, we will need 
a characterization of the set $\Cc(\La,m,\b)$ by induction on $\b$, 
generalizing \cite[Proposition 3.47]{VS1} to the case where $\La$ is
non-necessarily strict. 

\begin{lemm}
\label{RegrettableOmission}
Let $[\La,n,m,\b]$ be a simple stratum in $\A$ 
and $\t$ be a character of the group $\H^{m+1}(\b,\La)$, 
and set $q=-k_{0}(\b,\La)$ and $m'=\max\{m,\lfloor q/2\rfloor\}$. 
Then $\t\in\Cc(\La,m,\b)$ if and only if it is
normalized by $\KK(\La)\cap\mult\B$ and satisfies the following 
conditions: 
\begin{enumerate}
\item
if $\b$ is minimal over $\F$, then 
$\t\ |\ \U_{m'+1}(\La)=\psi_{\b}^{\A}$ and 
$\t\ |\ \U_{m+1}(\La)\cap\mult\B=\chi\circ\N_{\B/\E}$
for some character $\chi$ of $1+\p_{\E}$
(see (\ref{LesLapinsNAimentPasLesCarottes}) for the definition of 
$\psi_{\b}^{\A}$); 
\item
if $\b$ is not minimal over $\F$, and if $[\La,n,q,\g]$ is 
simple and equivalent to $[\La,n,q,\b]$ in $\A$, then
$\t\ |\ \H^{m'+1}(\b,\La)=\t_{0}^{}\psi_{\b-\g}^{\A}$ and 
$\t\ |\ \H^{m+1}(\b,\La)\cap\mult\B=\chi\circ\N_{\B/\E}$
for some simple character $\t_{0}\in\Cc(\La,m',\g)$ and 
some character $\chi$ of $1+\p_{\E}$.
\end{enumerate}
\end{lemm}

\begin{proof}
The proof is similar to that of \cite[Proposition 3.11]{VS1}, 
and we do not repeat it.
Note that \cite[Lemma 1.9]{Gr} is actually not needed in the proof, 
and that \cite[Corollary 5.3]{BK3} has to be replaced by 
\cite[Proposition 1.20]{SeSt} and \cite[Proposition 3.3.9]{BK} by 
\cite[Proposition 3.30]{VS1}.
\end{proof}

%%%%%%%%%%%%%%%%%%%%%%%%%%%%%%%%%%%%%%%%%%%%%%%%%%%%%%%%%%%%%%%%%%%%%%%%%%%

\label{Oahu}

Let $[\La,n,m,\b]$ be a simple $\K$-pure stratum in $\A$, 
and let $[\Ga,n,m,\b]$ correspond to it by (\ref{Dulilah}).
We write $\Cc(\Ga,m,\b)$
for the set of simple characters attached to $[\Ga,n,m,\b]$
with respect to the additive character:
\begin{equation}
\label{AddCharK}
\psi_{\K}=\psi\circ\tr_{\K/\F},
\end{equation}
which is trivial on $\p_{\K}$ but not on $\Oo_{\K}$,
as $\K$ is unramified over $\F$. 
Compare the following theorem with \cite[Theorem 7.7]{BH} and 
\cite[Proposition 7.1]{Gr}. 

\begin{theo}
\label{Syriana}
Let $[\La,n,m,\b]$ be a simple $\K$-pure stratum in $\A$, 
and let $[\Ga,n,m,\b]$ correspond to it by (\ref{Dulilah}).
Then, for any $\t\in\Cc(\La,m,\b)$, we have:
\begin{equation*}
\t\ |\ \H^{m+1}(\b,\Ga)\in\Cc(\Ga,m,\b).
\end{equation*}
\end{theo}

\begin{proof}
The proof is by induction on $\b$.
Let $\t^\K$ denote the restriction of $\t$ to the group 
$\H^{m+1}(\b,\Ga)$ and $\R$ the centralizer of $\K(\b)$ in $\A$.
Assume first that $\b$ is minimal over $\F$. 
By Proposition \ref{Alexandre1}, it is also minimal over $\K$. 
If $m\>\lfloor{n/2}\rfloor$, we have $\Cc(\La,m,\b)=\{\psi_{\b}^{\A}\}$ 
and $\Cc(\Ga,m,\b)=\{\psi_{\b}^{\C}\}$, where $\psi_{\b}^{\C}$ denotes 
the character of $\U_{m+1}(\Ga)$ defined by:
\begin{equation*}
\label{LesLapinsNAimentPasLesCarottesSurC}
\psi_\b^{\C}:x\mapsto\psi_{\K}\circ\tr_{\C/\K}(\b(x-1)).
\end{equation*}
So we just need to prove that:
\begin{equation}
\label{PsiRes}
\psi_{\b}^{\A}\ |\ \U_{m+1}(\Ga)=\psi_{\b}^{\C},
\end{equation}
which is given by \cite[Property (7.6)]{BH}.
If $m\<\lfloor{n/2}\rfloor$, then any 
$\t\in\Cc(\La,m,\b)$ extends the character 
$\psi_{\b}^{\A}\ |\ \U_{\lfloor n/2\rfloor+1}(\La)$ 
and its restriction to $\U_{m+1}(\La)\cap\mult\B$ has the form:
\begin{equation*}
\label{MirbeauK}
\t\ |\ \U_{m+1}(\La)\cap\mult\B=\chi\circ\N_{\B/\E}
\end{equation*}
for some character $\chi$ of $1+\p_{\E}$.
Therefore the character $\t^{\K}$ extends 
$\psi_{\b}^{\C}\ |\ \U_{\lfloor n/2\rfloor+1}(\Ga)$,
and its restriction to $\U_{m+1}(\Ga)\cap\mult\R$ 
has the form:
\begin{equation*}
\label{MirbeauP}
\t^{\K}\ |\ \U_{m+1}(\Ga)\cap\mult\R
=\chi\circ\N_{\K(\b)/\E}\circ\N_{\R/\K(\b)}.
\end{equation*}
Finally, the group $\KK(\Ga)\cap\mult\R$, which normalizes both $\t$ 
and the group $\H^{m+1}(\b,\Ga)$, normalizes $\t^{\K}$.
It follows from Lemma \ref{RegrettableOmission}
that $\t^{\K}\in\Cc(\Ga,m,\b)$.

Now assume that $\b$ is not minimal over $\F$. 
We set $q=-k_0(\b,\La)$ 
and $r=\lfloor q/2\rfloor+1$, and choose a simple stratum
$[\Ga,n,q,\g]$ equivalent to $[\Ga,n,q,\b]$ such that $[\La,n,q,\g]$ is 
simple and $\K$-pure. 
If $m\>\lfloor{q/2}\rfloor$, then 
any $\t\in\Cc(\La,m,\b)$ can be written as
$\t=\t_{0}^{}\psi_{\b-\g}^{\A}$ for some simple char\-ac\-ter 
$\t_{0}\in\Cc(\La,m,\g)$. 
Now we claim that:
\begin{equation}
\label{Col}
\H^{m+1}(\b,\Ga)=\H^{m+1}(\g,\Ga).
\end{equation}
We write $q_1=-k_0(\b,\Ga)$.
If $q_1=q$, then the equality (\ref{Col}) follows by definition. 
Other\-wise, we have $q_1>q$ by Proposition \ref{Alexandre1}.
The strata $[\Ga,n,q,\b]$ and $[\Ga,n,q,\g]$ are thus both simple, 
and (\ref{Col}) follows. 
We now restrict the character $\t$ to the group given by (\ref{Col}) and get 
$\t^{\K}=\t_{0}^{\K}\psi_{\b-\g}^{\C}$, 
where $\t_{0}^{\K}$ denotes the restriction $\t_0\ |\ \H^{m+1}(\g,\Ga)$, 
and this restriction is in $\Cc(\Ga,m,\g)$ by the inductive hypothesis.
If $q_1=q$, then $\t^{\K}$ is in $\Cc(\Ga,m,\b)$ by definition.  
Otherwise, $[\Ga,n,q,\b]$ is simple and the result follows from 
\cite[Proposition 2.15]{SeSt}. 
The case $m\<\lfloor{q/2}\rfloor$ reduces to the previous one exactly 
as in the minimal case. 
\end{proof}

%%%%%%%%%%%%%%%%%%%%%%%%%%%%%%%%%%%%%%%%%%%%%%%%%%%%%%%%%%%%%%%%%%%%%%%%%%%
%%%%%%%%%%%%%%%%%%%%%%%%%%%%%%%%%%%%%%%%%%%%%%%%%%%%%%%%%%%%%%%%%%%%%%%%%%%

\section{Interior lifting and transfer}
\label{ILT}

In this section, we define the interior lift of a ps-character.
This amounts to studying the behaviour of the interior lifting 
process with respect to trans\-fer. 

\subsection{}
\label{Jeu}

As in section \ref{PILSC}, we are given in this section 
a simple central $\F$-algebra
$\A$ and a finite unramified extension $\K/\F$ cont\-ain\-ed in $\A$.
We fix a finite unramified extension $\L$ of $\K$ such that the 
$\L$-algebra:
\begin{equation*}
\AL=\A\otimes_{\F}\L
\end{equation*}
is split. 
This $\L$-algebra inherits an action of the Galois group of $\L/\F$ 
in the obvious way, and we consider $\A$ as being naturally embedded 
in $\AL$ by $j_\A:a\mapsto a\otimes_{\F}1$.
We have a decomposition:
\begin{equation}
\label{KaiserFJ2}
\K\otimes_{\F}\L=\K^{1}\oplus\cdots\oplus\K^{\d}
\end{equation}
into simple $\K\otimes_{\F}\L$-modules, where $\d$ denotes the degree 
of $\K/\F$.
For each $i\in\{1,\ldots,\d\}$, we write $\e^{i}$ for the minimal 
idempotent in $\K\otimes_{\F}\L$ corresponding to $\K^{i}$.
The centralizer of $\K\otimes_{\F}\L$ in $\overline{\A}$, denoted $\UL$, 
is equal to $\C\otimes_{\F}\L$.
By identifying it with $\C\otimes_{\K}(\K\otimes_{\F}\L)$ and using 
(\ref{KaiserFJ2}), we get a decomposition:
\begin{equation*}
\UL=\UL^{1}\oplus\cdots\oplus\UL^{\d},
\end{equation*}
where the $\K^{i}$-algebra $\UL^{i}=\e^{i}\overline{\A}\e^{i}$ 
identifies with $\C\otimes_{\K}\K^{i}$ for each $i\in\{1,\ldots,\d\}$.

\medskip

In a similar way, we may consider the centralizer $\C$ of $\K$ in $\A$ as 
being embedded in the split $\L$-algebra $\CL=\C\otimes_\K\L$ by 
the $\K$-algebra homomorphism $j_\C:c\mapsto c\otimes_{\K}1$. 

\medskip

Similarly to the case of simple characters (see paragraph \ref{Oahu}), 
we will define the 
interior lift of a quasi-simple character by restriction from $\AL$ to 
$\CL$.
For this we need an embedding of $\CL$ in $\AL$ satisfying some conditions 
with respect to $j_\A$ and $j_\C$ (see below), 
but there is no canonical such em\-bed\-ding. 
We choose a set:
\begin{equation}
\label{Pipistrelle}
\SS=\{\s_1,\dots,\s_{\d}\}\subseteq\Gal(\L/\F)
\end{equation}
of representatives of $\Hom_{\F}(\K,\L)$ in $\Gal(\L/\F)$, that is a subset 
of $\Gal(\L/\F)$ such that the 
restriction map from $\L$ to $\K$ induces a bijection from $\SS$ to 
$\Hom_{\F}(\K,\L)$.
For simplicity, we assume that we have ordered the $\e^{i}$'s so that:
\begin{equation}
\label{PoupaPoupa}
\textit{$\K^1$ and $\L$ are isomorphic $\K\otimes\L$-modules
and $\s_{i}(\e^{1})=\e^{i}$ for any $i\in\{1,\ldots,\d\}$.}
\end{equation}
This gives us an $\F$-algebra homomorphism:
\begin{equation}
\label{MariusMarius}
\can:\CL\ {\smash{\mathop{\longrightarrow}\limits^{\simeq}}}\ 
\UL^{1}\subseteq\UL,
\end{equation}
and $\s_{i}\circ\can$ is an $\F$-algebra homomorphism from $\CL$
to $\UL^{i}$ for each integer $i\in\{1,\ldots,\d\}$.
The fol\-lowing lemma gives us a relationship between 
(\ref{MariusMarius}) and the embeddings $j_\A$ and $j_\C$.

\begin{lemm}
\label{Glossolalies}
Let $j_{\A,\C}$ denote the restriction of $j_\A$ to $\C$, with 
values in $\UL$.
Then the $\F$-algebra homomorphism from $\CL$ to $\UL$ defined by:
\begin{equation}
\label{Semprum}
\ii=\ii_{\SS}:x\mapsto \s_1\circ\can(x)+\dots+\s_{\d}\circ\can(x)
\end{equation}
satisfies the equality $\ii\circ j_{\C}=j_{\A,\C}$.
\end{lemm}

\begin{proof}
We have $\s_{i}(\e^1j_\A(x))=\e^ij_\A(x)$ for all 
$i\in\{1,\ldots,\d\}$ and $x\in\C$, which implies that 
$\ii\circ\e^1j_{\A,\C}=j_{\A,\C}$.
Note that $\e^1j_{\A,\C}=j_\C$, so that we 
get the expected equality.
\end{proof}

%%%%%%%%%%%%%%%%%%%%%%%%%%%%%%%%%%%%%%%%%%%%%%%%%%%%%%%%%%%%%%%%%%%%%%%%%%%

\subsection{}
\label{LarbaudFerminaMarquez}

Let $[\La,n,m,\b]$ be a simple stratum in $\A$, 
which is a realization of a simple pair $(k,\b)$ over $\F$. 
\textit{In this paragraph, we assume that $\La$ is a strict
lattice sequence.} 
If we fix a simple left $\AL$-module $\VL$,
then there is a unique (up to translation) $\Oo_{\L}$-lattice sequence 
$\overline{\La}$ on $\VL$ such that:
\begin{equation}
\label{Ldescen2}
\aa_k(\overline{\La})=\aa_k(\La)\otimes_{\Oo_{\F}}\Oo_{\L},
\quad k\in\ZZ
\end{equation}
(see \cite[\S2.2]{VS1}).
This provides us with a stratum $[\overline{\La},n,m,\b]$ in $\AL$, 
called the {\it quasi-simple $\L/\F$-lift} of the simple stratum 
$[\La,n,m,\b]$ ({\it ibid.}).
This quasi-simple lift is pure if and only if the residue class 
degree of $\E$ over $\F$ is prime to the degree of $\L$ over $\F$,
and in this case it is a simple stratum. 

\medskip

In \cite{VS1} (see paragraph 3.2.3), 
one attaches to the stratum $[\LaL,n,m,\b]$ 
a compact open subgroup $\H^{m+1}(\b,\LaL)$ of $\AL^{\times}$ and a set 
$\Qq(\LaL,m,\b)$ of characters of the group $\H^{m+1}(\b,\LaL)$, 
called {\it quasi-simple characters} of level $m$ and depending on an 
additive character:
\begin{equation}
\label{Versicolor}
\PSI:\L\to\mathbb{C}^{\times}
\end{equation}
extending the additive character 
(\ref{AddCharJ}), being trivial on $\p_{\L}$ but not on $\Oo_{\L}$. 
Recall that the restriction map from $\H^{m+1}(\b,\LaL)$ to 
$\H^{m+1}(\b,\La)$ induces a surjective map from $\Qq(\LaL,m,\b)$ to
$\Cc(\La,m,\b)$. 

\medskip

Let $[\La',n',m',\b]$ be another real\-ization of $(k,\b)$ in a 
simple central $\F$-algebra $\A'$, \textit{with $\La'$ strict}.
We assume that the extension $\L/\F$ is chosen such 
that the $\L$-algebras $\AL$ and $\AL^{\prime}$ are both split, 
and we set:
\begin{equation}
\label{Scelerat}
\V^{\up}=\VL\oplus\VL',
\quad
\La^{\up}=\LaL\oplus\LaL'.
\end{equation}
Then $\La^{\up}$ is a strict $\Oo_{\L}$-lattice sequence on the $\L$-vector 
space $\V^{\up}$, and $\A^{\up}=\End_{\L}(\V^{\up})$ is a split simple 
central $\L$-algebra in which $\E=\F(\b)$ is naturally embedded.
We write $\M$ for the product 
$\AL^{\times}\times\overline{\A}{}^{\prime\times}$
considered as a Levi sub\-group of $\A^{\up\times}$, so that we have
the de\-com\-po\-si\-tion: 
\begin{equation}
\label{Bershka}
\H^{m+1}(\b,\La^{\up})\cap\M=\H^{m+1}(\b,\LaL)\times\H^{m+1}(\b,\LaL').
\end{equation} 
We will need the following characterization of the transfer map.

\begin{prop}
\label{RapaNui}
Let $\t\in\Cc(\La,m,\b)$ and $\t'\in\Cc(\La',m',\b)$ be two 
simple characters.
Then $\t'$ is the transfer of $\t$ if and only if there exists 
$\boldsymbol{\t}^{\up}\in\Qq(\La^{\up},m,\b)$
such that we have:
\begin{equation}
\label{Boucherie}
\boldsymbol{\t}^{\up}\ |\ \H^{m+1}(\b,\La)\times\H^{m+1}(\b,\La')
=\t\otimes\t'.
\end{equation}
\end{prop}

\begin{proof}
Recall (see \cite[\S3.3]{VS1}) that $\t$ and $\t'$ are transfers 
of each others 
if and only if there exist two quasi-simple characters 
$\boldsymbol{\t}\in\Qq(\LaL,m,\b)$ and $\boldsymbol{\t}'\in\Qq(\LaL',m,\b)$,
extending $\t$ and $\t'$ respectively, which are transfers of each others.

\begin{lemm}
\label{Orion}
The map from $\Qq(\La^{\up},m,\b)$ to $\Qq(\LaL,m,\b)$ induced by 
the restriction from $\H^{m+1}(\b,\La^{\up})$ to $\H^{m+1}(\b,\LaL)$ is 
the transfer.
\end{lemm}

\begin{proof}
We have a decomposition of the $\L$-algebra $\E\otimes_{\F}\L$ 
into simple $\E\otimes_{\F}\L$-modules $\E^{j}$, 
for $j\in\{1,\ldots,\c\}$, where $\c$ denotes the greatest 
common divisor of the degree of $\L/\F$
and the residue class degree of $\E/\F$.
For each $j$, we write $\ce^{j}$ for the minimal 
idempotent in $\E\otimes_{\F}\L$ corresponding to $\E^{j}$,
as well as $\La^{\up j}$ for the projection of $\La^{\up}$
onto $\V^{\up j}=\ce^{j}\V^{\up}$ and $\b^{j}$ for $\ce^{j}\b$.
Thus we get a simple stratum $[\La^{\up j},n,m,\b^{j}]$ 
in the $\F$-algebra $\A^{\up j}=\ce^{j}\A^{\up}\ce^{j}$ and, 
similarly, we get a simple stratum $[\LaL^{j},n,m,\b^{j}]$ 
in $\AL^{j}$.
By \cite[Corollaire 3.34]{VS1}, there are bijections:
\begin{equation*}
\Qq(\La^{\up},m,\b)\to
\prod\limits_{j=1}^{\c}\Cc(\La^{\up j},m,\b^{j}),
\quad
\Qq(\LaL,m,\b)\to\prod\limits_{j=1}^{\c}\Cc(\LaL^{j},m,\b^{j}),
\end{equation*}
which are compatible with transfer.
Therefore, it is enough to prove that, for each $j$, 
the map from 
$\Cc(\La^{\up j},m,\b^{j})$ to $\Cc(\LaL^{j},m,\b^{j})$
induced by the restriction from $\H^{m+1}(\b^{j},\La^{\up j})$ to 
$\H^{m+1}(\b^{j},\LaL^{j})$ is the transfer. 
This is \cite[Th\'eor\`eme 2.17]{SeSt}.
\end{proof}

Assume first that there exists a quasi-simple character 
$\boldsymbol{\t}^{\up}\in\Qq(\La^{\up},m,\b)$ such that 
(\ref{Boucherie}) is satisfied, and write $\boldsymbol{\t}$
and $\boldsymbol{\t}'$ for the restrictions of $\boldsymbol{\t}^{\up}$
to $\H^{m+1}(\b,\LaL)$ and $\H^{m+1}(\b,\LaL')$, res\-pec\-tively.
By Lemma \ref{Orion}, these are quasi-simple characters which are 
transfers of each others. 
By (\ref{Boucherie}), they extend the 
simple characters $\t$ and $\t'$. 
It follows that $\t$ and $\t'$ are transfers of each others. 

Conversely, assume that $\t$ and $\t'$ are transfers of each others.
Let $\boldsymbol{\t}$ be a quasi-simple character in 
$\Qq(\LaL,m,\b)$ extending $\t$, 
and let $\boldsymbol{\t}^{\up}$ be its transfer to 
$\Qq(\La^{\up},m,\b)$.
By Lemma \ref{Orion}, the restriction of $\boldsymbol{\t}^{\up}$ 
to $\H^{m+1}(\b,\LaL')$ is the transfer 
of $\boldsymbol{\t}$, and thus extends $\t'$.
Therefore, the identity (\ref{Boucherie}) is satisfied. 
\end{proof}

\subsection{}
\label{Larug}

Let $[\La,n,m,\b]$ be a $\K$-pure simple stratum in $\A$,
and let $[\Ga,n,m,\b]$ denote the simple stratum in $\C$ 
associated with $[\La,n,m,\b]$ by (\ref{Dulilah}).
\textit{In this paragraph, we assume that $\La$ and $\Ga$ are 
strict lattice sequences.}

\medskip

If we fix a simple left $\CL$-module $\overline{\W}$,
we can form the quasi-simple lift $[\GaL,n,m,\b]$ of the 
simple stratum $[\Ga,n,m,\b]$ with respect to $\L/\K$.
One attaches to this quasi-simple lift a 
compact open sub\-group $\H^{m+1}(\b,\GaL)$ of $\CL^{\times}$ and 
a set $\Qq(\GaL,m,\b)$ of characters of $\H^{m+1}(\b,\GaL)$ with respect 
to the additive character:
\begin{equation}
\PSI_{\K}=\PSI\circ(\s_1+\dots+\s_{\d})
\end{equation}
of $\L$, depending on the choice of the set $\SS$ fixed in 
(\ref{Pipistrelle}). 
It is trivial on $\p_\L$ and, thanks to the condition on $\SS$, 
it extends the character $\psi_{\K}$ defined by 
(\ref{AddCharK}); hence it is not trivial on $\Oo_\L$. 
This comes with a surjective restriction map from $\Qq(\GaL,m,\b)$ to
$\Cc(\Ga,m,\b)$. 

\begin{lemm}
\label{MonaLoa}
The image of $\H^{m+1}(\b,\GaL)$ by the map $\ii$ 
is contained in $\H^{m+1}(\b,\LaL)$.
\end{lemm}

\begin{proof}
First we have to prove that:
\begin{equation*}
\can(\H^{m+1}(\b,\GaL))
=\H^{m+1}(\b,\LaL)\cap\UL^{1}
=\e^{1}\H^{m+1}(\b,\LaL)\e^{1}.
\end{equation*}
This follows from the definition of the groups 
$\H^{m+1}(\b,\GaL)$ and $\H^{m+1}(\b,\LaL)$ by induction on $\b$,
and from the fact that $\e^1$ commutes to $\b$.
According to (\ref{PoupaPoupa}), we get:
\begin{equation*}
\s_i\circ\can(\H^{m+1}(\b,\GaL))
=\H^{m+1}(\b,\LaL)\cap\UL^{i}
=\e^{i}\H^{m+1}(\b,\LaL)\e^{i}
\end{equation*}
for each $i\in\{1,\ldots,\d\}$,
and the result follows.
\end{proof}

This gives rise to the following result.

\begin{prop}
\label{Kyoto2}
Let $\t\in\Cc(\La,m,\b)$ be a simple character, 
let $\boldsymbol{\t}\in\Qq(\LaL,m,\b)$ be a
quasi-simple character extending $\t$, and set:
\begin{equation}
\label{MonaKea}
\boldsymbol{\t}^{\K}(x)=\boldsymbol{\t}(\ii(x)),
\quad
x\in\H^{m+1}(\b,\GaL).
\end{equation}
Then $\boldsymbol{\t}^{\K}$ is a quasi-simple character in 
$\Qq(\GaL,m,\b)$ extending $\t^{\K}=\t\ |\ \H^{m+1}(\b,\Ga)$.
\end{prop}

\begin{proof}
By Lemmas \ref{Glossolalies} and \ref{MonaLoa}, the character 
$\boldsymbol{\t}^{\K}$ is well defined and extends the simple 
character $\t^{\K}$.
It thus remains to prove that it is in $\Qq(\GaL,m,\b)$.
The proof is by induction on $\b$
(see \cite[D\'efinition 3.22]{VS1}).
Assume first that $\b$ is minimal over $\F$. 
Then it is minimal over $\K$ by Proposition \ref{Alexandre1}.
If $m\>\lfloor n/2\rfloor$, the set $\Qq(\LaL,m,\b)$ con\-sists of a single 
element $\PSI_{\b}^{\AL}$, which is the character of $\U_{m+1}(\LaL)$ 
defined by:
\begin{equation*}
\PSI_\b^{\AL}(x)=\PSI\circ\tr_{\AL/\L}(\b(x-1)),
\quad
x\in\U_{m+1}(\LaL),
\end{equation*}
and the set $\Qq(\GaL,m,\b)$ consists of a single 
element $\PSI_{\b}^{\CL}$, which is the character of $\U_{m+1}(\GaL)$ 
defined by:
\begin{equation*}
\PSI_\b^{\CL}(x)=\PSI\circ\tr_{\CL/\L}(\b(x-1)),
\quad
x\in\U_{m+1}(\GaL).
\end{equation*}
So we just need to prove that:
\begin{equation}
\label{Maca}
\PSI_\b^{\AL}\circ\ii(x)=\PSI_{\b}^{\CL}(x),
\quad
x\in\U_{m+1}(\GaL),
\end{equation}
which follows from the fact that:
\begin{eqnarray*}
\tr_{\AL/\L}\circ\ii&=&\sum\limits_{i=1}^{\d}\tr_{\AL/\L}\circ\s_i\circ\can\\
&=&(\s_1+\dots+\s_{\d})\circ\tr_{\AL/\L}\circ\can
=(\s_1+\dots+\s_{\d})\circ\tr_{\CL/\L}.
\end{eqnarray*}
If $m\<\lfloor n/2\rfloor$, then $\boldsymbol{\t}$ extends 
$\PSI_{\b}^{\AL}\ |\ \U_{\lfloor n/2\rfloor+1}(\LaL)$ 
and its restriction to $\U_{m+1}(\LaL)\cap{\BL}^{\times}$ 
has the form:
\begin{equation}
\label{Mirbeau2}
\boldsymbol{\t}\ |\ \U_{m+1}(\LaL)\cap{\BL}^{\times}=
\boldsymbol{\chi}\circ\N_{\BL/\E\otimes_\F\L},
\end{equation}
where we write $\BL$ for the centralizer of $\E$ in $\AL$ 
and where $\boldsymbol{\chi}$ 
denotes some character of the subgroup 
$1+\p_{\E}\otimes\Oo_{\L}$ of $(\E\otimes_\F\L)^\times$.
Then, if we write $\RL$ for the centralizer of $\K(\b)$ in $\AL$, 
the character $\boldsymbol{\t}^{\K}$ extends 
$\PSI_{\b}^{\CL}\ |\ \U_{\lfloor n/2\rfloor+1}(\GaL)$, and its
restriction to $\U_{m+1}(\GaL)\cap{\RL}^{\times}$ has the form: 
\begin{equation}
\label{Mirbeau'2}
\boldsymbol{\t}^{\K}\ |\ \U_{m+1}(\GaL)\cap{\RL}^{\times}
=\boldsymbol{\chi}^{\SS}\circ\N_{\RL/\K(\b)\otimes_{\K}\L}
\end{equation}
where $\boldsymbol{\chi}^{\SS}$ is the product of all 
the $\boldsymbol{\chi}\circ\s_i$'s for all $i\in\{1,\dots,\d\}$,
as required.

Assume now that $\b$ is not minimal over $\F$. 
We set $q=-k_0(\b,\La)$ and $r=\lfloor{q/2}\rfloor+1$, 
and choose a simple stratum $[\Ga,n,q,\g]$ equivalent 
to $[\Ga,n,q,\b]$ such that $[\La,n,q,\g]$ is simple 
and $\K$-pure. 
By \cite[Theorem 5.1]{BG}
and Proposition \ref{KaPureApprox} together, 
one may assume that the maximal unramified extension of $\F$ contained 
in $\F(\g)$ is contained in that of $\F(\b)$, which implies that the 
$\L$-canonical decomposition of $\g$ is finer than that of $\b$ 
(see paragraph 2.3.4 and the proof of Lemme 3.16 in \cite{VS1}).
If $m\>\lfloor q/2\rfloor$, then 
any $\boldsymbol{\t}\in\Qq(\LaL,m,\b)$ can be written as
$\boldsymbol{\t}=\boldsymbol{\t}_{0}^{}\PSI_{\b-\g}^{\AL}$ for some 
quasi-simple char\-ac\-ter $\boldsymbol{\t}_{0}\in\Qq(\LaL,m,\g)$. 
Now we claim that:
\begin{equation}
\label{ColL}
\H^{m+1}(\b,\GaL)=\H^{m+1}(\g,\GaL).
\end{equation}
We write $q_1=-k_0(\b,\Ga)$.
If $q_1=q$, then the equality (\ref{ColL}) follows by definition. 
Other\-wise, we have $q_1>q$ by Proposition \ref{Alexandre1}.
The strata $[\Ga,n,q,\b]$ and $[\Ga,n,q,\g]$ are thus simple, 
and (\ref{ColL}) follows. 
We now form the character 
$\boldsymbol{\t}^{\K}=\boldsymbol{\t}\circ\ii\ |\ \H^{m+1}(\b,\GaL)$ 
and get the equality 
$\boldsymbol{\t}^{\K}=\boldsymbol{\t}_{0}^{\K}\PSI_{\b-\g}^{\CL}$, 
where $\boldsymbol{\t}_{0}^{\K}$ denotes the character 
$\boldsymbol{\t}_0\circ\ii\ |\ \H^{m+1}(\g,\GaL)$, 
and this character is in $\Qq(\GaL,m,\g)$ by the inductive hypothesis.
If $q_1=q$, then $\boldsymbol{\t}^{\K}$ is in $\Qq(\GaL,m,\b)$ by definition.  
Otherwise, the strata $[\Ga,n,q,\b]$ and $[\Ga,n,q,\g]$ 
are simple and the result follows from 
\cite[Proposition 2.15]{SeSt}. 
The case $m\<\lfloor{q/2}\rfloor$ reduces to the previous one 
as in the minimal case. 

It remains to prove that the subgroup $\KK(\GaL)\cap{\RL}{}^{\times}$ 
normalizes $\boldsymbol{\t}^{\K}$.
If $g\in\KK(\GaL)\cap{\RL}{}^{\times}$, then we have:
\begin{eqnarray}
\label{GTMR}
\ii(g)\cdot\LaL_{k}
=\bigoplus\limits_{i=1}^{\d}\s_i(\can(g))\cdot\e^{i}\LaL_{k}
=\bigoplus\limits_{i=1}^{\d}\e^{i}\LaL_{k+\v(\s_i(\can(g)))}
\end{eqnarray}
where $\v$ denotes the valuation map associated with $\LaL$.
As all the $\s_i(\can(g)$'s have the same valuation, 
the equality (\ref{GTMR}) 
gives us 
$\ii(g)\in\KK(\LaL)\cap{\BL}{}^{\times}$.
Proposition \ref{Kyoto2} now follows from the fact that 
$\KK(\LaL)\cap{\BL}{}^{\times}$ normalizes $\boldsymbol{\t}$. 
\end{proof}

\begin{rema}
Note that the interior lifting map 
from $\Qq(\LaL,m,\b)$ to $\Qq(\GaL,m,\b)$ defined by Proposition \ref{Kyoto2}
depends on the choice of the set $\SS$ chosen in (\ref{Pipistrelle}).
\end{rema}

%%%%%%%%%%%%%%%%%%%%%%%%%%%%%%%%%%%%%%%%%%%%%%%%%%%%%%%%%%%%%%%%%%%%%%%%%%%

\subsection{}
\label{DeThou}

Let $[\La,n,m,\b]$ and $[\La',n',m',\b]$ be 
real\-izations of a simple pair $(k,\b)$ over
$\F$ in 
simple central $\F$-algebras $\A$ and $\A'$, 
respectively. 
Assume further that $\A$ and $\A'$ contain $\K$, 
that the strata $[\La,n,m,\b]$ and $[\La',n',m',\b]$ 
are $\K$-pure and that the strata $[\Ga,n,m,\b]$ 
and $[\Ga',n',m',\b]$ associated with them by (\ref{Dulilah}) 
are realizations of the same simple pair over $\K$. 
(This is equivalent to saying that the extensions of $\K$ 
generated by $\b$ in $\A$ and $\A'$ are $\K$-isomorphic.)
We have the following relation between the transfer maps 
and the interior lifting maps.

\begin{theo}
\label{ResComWithTransAndAEFit}
Let $\t\in\Cc(\La,m,\b)$ and $\t'\in\Cc(\La',m',\b)$ be 
transfers of each others.
Then the simple characters:
\begin{equation*}
\t\ |\ \H^{m+1}(\b,\Ga),
\quad
\t'\ |\ \H^{m'+1}(\b,\Ga')
\end{equation*}
are transfers of each others.
\end{theo}

\begin{proof}
The proof decomposes into two parts.
\begin{enumerate}
\item 
First we prove the theorem in the case where all the lattice sequences 
are strict, so that we can apply the results of paragraphs 
\ref{LarbaudFerminaMarquez} and \ref{Larug}.
We fix a quasi-simple character $\boldsymbol{\t}$ in $\Qq(\LaL,m,\b)$ 
extending $\t$ and write $\boldsymbol{\t}'$ for its transfer in 
$\Qq(\LaL',m',\b)$. 
The restriction of $\boldsymbol{\t}'$ 
to $\H^{m'+1}(\b,\La')$ is thus equal to $\t'$.
By Pro\-position \ref{RapaNui}, there exists a quasi-simple character
$\boldsymbol{\t}^{\up}$ in $\Qq(\La^{\up},m,\b)$ extending 
$\boldsymbol{\t}\otimes\boldsymbol{\t}'$.
We write $\CL$ and $\UL$ as in paragraph \ref{Jeu}, 
and use similar notations $\CL'$ and $\UL'$.
We have:
\begin{equation}
\label{AAA}
\H^{m+1}(\b,\La^{\up})\cap\(\CL^{\times}\times\CL^{\prime\times}\)
=\H^{m+1}(\b,\GaL)\times\H^{m+1}(\b,\GaL').
\end{equation}
We define $\ii$ by (\ref{Semprum}) and write $\boldsymbol{\t}^{\K}$
for the quasi-simple character defined by (\ref{MonaKea}).
We also 
define $\ii'$ and $\boldsymbol{\t}^{\prime\K}$ in a similar way. 
If we restrict the map $x\mapsto(\ii(x),\ii'(x))$ to the subgroup 
(\ref{AAA}) and then compose it with $\boldsymbol{\t}^{\up}$, 
then we get the character 
$\boldsymbol{\t}^{\K}\otimes\boldsymbol{\t}^{\prime\K}$.
This implies that $\boldsymbol{\t}^{\K}$ and $\boldsymbol{\t}^{\prime\K}$ 
are transfers of each others.
By Propositions \ref{Kyoto2} and \ref{RapaNui} together, their 
restrictions 
$\boldsymbol{\t}^{\K}\ |\ \H^{m+1}(\b,\Ga)=\t^\K$ and
$\boldsymbol{\t}^{\prime\K}\ |\ \H^{m'+1}(\b,\Ga')=\t^{\prime\K}$
are transfers of each others.
\item
We now reduce the general case to Case (1).
For this we fix a positive integer $l$ as in Lemma \ref{EveryBodyIsSound}, 
and form the sound simple strata $[\La^{\ddag},n,m,\b]$ and 
$[\La^{\prime\ddag},n',m',\b]$.
Write $\C^{\ddag}$ for the centralizer of $\K$ in $\A^{\ddag}$ and 
$[\Ga^{\ddag},n,m,\b]$ for the simple stratum in $\C^{\ddag}$ 
associated with $[\La^{\ddag},n,m,\b]$ by (\ref{Dulilah}).
In a similar way, we have a $\K$-algebra $\C^{\prime\ddag}$ and a 
simple stratum $[\Ga^{\prime\ddag},n,m,\b]$. 
Then the simple strata $[\Ga^{\ddag},n,m,\b]$ and $[\Ga^{\prime\ddag},n,m,\b]$ 
are realizations of the same simple pair over $\K$. 
Write $\t^{\ddag}$ for the transfer of $\t$ in $\Cc(\La^{\ddag},m,\b)$.
In a similar way, we have a simple character $\t^{\prime\ddag}$.
By Case (1), the simple characters:
\begin{equation*}
\t^{\ddag}\ |\ \H^{m+1}(\b,\Ga^{\ddag}),
\quad
\t^{\prime\ddag}\ |\ \H^{m'+1}(\b,\Ga^{\prime\ddag})
\end{equation*}
are transfers of each others.
Thus it remains to prove the following lemma.

\begin{lemm}
\label{Magritte}
The characters $\t\ |\ \H^{m+1}(\b,\Ga)$ and 
$\t^{\ddag}\ |\ \H^{m+1}(\b,\Ga^{\ddag})$ are 
transfers of each others.
\end{lemm}

\begin{proof}
Write $\M$ for the Levi subgroup of $\A^{\ddag\times}$ defined by 
the decomposition of $\V^{\ddag}$ into copies of $\V$. 
According to Lemma \ref{JacquesVerges}, the character $\t^\ddag$ is 
characterised by the identity:
\begin{equation*}
\t^\ddag\ |\ \H^{m+1}(\b,\La^\ddag)\cap\M=\t\otimes\dots\otimes\t.
\end{equation*}
Thus its restriction to 
$\H^{m+1}(\b,\Ga^\ddag)\cap\M
=\H^{m+1}(\b,\Ga)\times\dots\times\H^{m+1}(\b,\Ga)$ 
is equal to the 
tensor product of $l$ copies of $\t^\K$.
\end{proof}

\end{enumerate}
This ends the proof of Theorem \ref{ResComWithTransAndAEFit}. 
\end{proof}

\begin{rema}
In the case where $[\La,n,m,\b]$ and 
$[\La',n',m',\b]$ are sound,
this theorem implies that Grabitz's transfer \cite{Gr} is 
the same as the transfer defined in \cite{VS1}.
\end{rema}

%%%%%%%%%%%%%%%%%%%%%%%%%%%%%%%%%%%%%%%%%%%%%%%%%%%%%%%%%%%%%%%%%%%%%%%%%%%

\subsection{}

Before closing this section, we prove the following result. 
Let $[\La,n,m,\b]$ be a simple $\K$-pure stratum in $\A$, 
and write $[\Ga,n,m,\b]$ for the simple stratum in $\C$
which corresponds to it by (\ref{Dulilah}).
Theorem \ref{Syriana} gives us a map from $\Cc(\La,m,\b)$ 
to $\Cc(\Ga,m,\b)$, 
called the interior lift\-ing map, and denoted 
$\boldsymbol{l}_{\K/\F}:\t\mapsto\t^\K$. 
It has the following properties. 

\begin{prop}
\label{PropInjEquivInt}
The map $\boldsymbol{l}_{\K/\F}$ is injective and 
$\KK(\Ga)$-equivariant.
\end{prop}

\begin{proof}
Note that the second assertion is immediate. 
Let us fix a positive integer $l\>1$ as in Lemma \ref{EveryBodyIsSound}, 
and form the sound simple stratum $[\La^{\ddag},n,m,\b]$.
Write $\C^{\ddag}$ for the centralizer of $\K$ in $\A^{\ddag}$ and 
$[\Ga^{\ddag},n,m,\b]$ for the simple stratum in $\C^{\ddag}$ 
associated with the stratum $[\La^{\ddag},n,m,\b]$ by (\ref{Dulilah}). 
Now let $\t\in\Cc(\La,m,\b)$ be a simple character 
and write $\t^{\ddag}$ for its transfer in $\Cc(\La^{\ddag},m,\b)$. 
Then, by Lemma \ref{Magritte},
the transfer of $\t^\K$ to $\Cc(\Ga^{\ddag},m,\b)$ is equal to
$\t^{\ddag}\ \vert\ \H^{m+1}(\b,\Ga^{\ddag})$. 
As the transfer 
map from $\Cc(\La,m,\b)$ to $\Cc(\La^{\ddag},m,\b)$ is bijective, 
we may replace $\La$ by $\La^{\ddag}$ and assume that the stratum 
$[\La,n,m,\b]$ is sound.  
In this case, the injectivity of the map $\boldsymbol{l}_{\K/\F}$ 
follows from \cite[Proposition 7.1]{Gr}.
\end{proof}

Assume we are given two $\K$-pure 
simple strata $[\La,n,m,\b_i]$, $i=1,2$, in $\A$. 
For each $i$, let $\t_i$ be a simple character in $\Cc(\La,m,\b_i)$.

\begin{prop}
\label{MauiToho}
Assume that $\t_1$ and $\t_2$ are equal.
Then $\boldsymbol{l}_{\K/\F}(\t_1)$ and 
$\boldsymbol{l}_{\K/\F}(\t_2)$ are equal.
\end{prop}

\begin{proof}
It suffices to verify that the groups $\H^{m+1}(\b_i,\Ga)$, $i=1,2$, are 
equal.  
This follows from Proposition \ref{Alexandre} and the fact that the groups 
$\H^{m+1}(\b_i,\La)$, $i=1,2$, are equal.  
\end{proof}

%%%%%%%%%%%%%%%%%%%%%%%%%%%%%%%%%%%%%%%%%%%%%%%%%%%%%%%%%%%%%%%%%%%%%%%%%%%
%%%%%%%%%%%%%%%%%%%%%%%%%%%%%%%%%%%%%%%%%%%%%%%%%%%%%%%%%%%%%%%%%%%%%%%%%%%

\section{The base change}
\label{BCASDFSC}

In this section, we develop a base change process for simple 
strata and characters with respect to a finite unramified 
extension $\K$ of $\F$, in a way similar to \cite{BH}.

\subsection{}

Let $\K/\F$ be an unramified extension of degree $\d$. 
Given a simple central $\F$-alg\-ebra $\A$, we set:
\begin{equation*}
\qq{\A}=\A\otimes_{\F}\End_{\F}(\K).
\end{equation*}
Then $\K$ embeds naturally in $\qq{\A}$, and its centralizer, 
denoted $\AH$, is canonically isomorphic to $\A\otimes_{\F}\K$ 
as a $\K$-algebra.
Let $\V$ be a simple left $\A$-module.
Then $\qq{\V}=\V\otimes_{\F}\K$ is a simple left $\qq{\A}$-module 
and, if we fix an $\F$-basis of $\K$, we have a decomposition: 
\begin{equation}
\label{Gauvain}
\qq{\V}=\V\oplus\dots\oplus\V
\end{equation}
of $\qq{\V}$ into a sum of $\d$ copies of $\V$, so that we are in the 
situation of paragraph \ref{DerivedStratum}.

\medskip

Let $[\La,n,m,\b]$ be a simple stratum in $\A$ and set $\E=\F(\b)$.
Let us form the simple stra\-tum $[\qq{\La},n,m,\b]$ in $\qq{\A}$, where
$\qq{\La}=\La\oplus\dots\oplus\La$ is the direct sum of $\d$ copies of $\La$. 
This simple stratum is not $\K$-pure in general.
We have a decomposition:
\begin{equation*}
\label{KFJ}
\E\otimes_{\F}\K=\E^{1}\oplus\cdots\oplus\E^{\c}
\end{equation*}
into simple $\E\otimes_{\F}\K$-modules, where $\c$ denotes the greatest 
common divisor of $\d$ and the residue class degree of $\E$ over $\F$.
For each $j\in\{1,\ldots,\c\}$, we write $\ee^{j}$ for the minimal 
idempotent in $\E\otimes_{\F}\K$ corresponding to $\E^{j}$, and we set:
\begin{equation*}
\b^{j}=\ee^{j}\b,
\quad
j\in\{1,\ldots,\c\}.
\end{equation*}
These are the various $\K/\F$-lifts of $\b$.
If we write
$\qq{\La}^{j}$ for the projection of $\qq{\La}$ onto the space 
$\qq{\V}^{j}=\ee^{j}\qq{\V}$ for each $j$, 
we get a simple stratum $[\qq{\La}^{j},n,m,\b^j]$ in the
$\F$-algebra $\qq{\A}^{j}=\ee^{j}\qq{\A}\ee^{j}$,
which is $\K$-pure for the natural embedding of $\K$ in $\qq{\A}^{j}$.
Thus one can form the interior lift $[\qq{\Ga}^{j},n,m,\b^j]$ in the 
centralizer of $\K$ in $\qq{\A}^{j}$ (see paragraph \ref{BelleDeJour}).

\medskip

Given a simple character $\t\in\Cc(\La,m,\b)$, let $\qq{\t}$ denote 
its transfer to $\Cc(\qq{\La},m,\b)$ and write $\qq{\t}^{j}$ for the 
transfer of $\qq{\t}$ to $\Cc(\qq{\La}^{j},m,\b^j)$, that is the 
restriction of $\qq{\t}$ to $\H^{m+1}(\b^{j},\qq{\La}^{j})$. 
Let us denote by $\t_{\K}^{j}$ the restriction of $\qq{\t}^{j}$ to 
$\H^{m+1}(\b^{j},\qq{\Ga}^{j})$, which belongs to 
$\Cc(\qq{\Ga}^{j},m,\b^j)$ by Theorem \ref{Syriana}.
We have the following definition. 

\begin{defi}
The process:
\begin{equation*}
\boldsymbol{b}_{\K/\F}:\t\mapsto
\{\t_{\K}^{j},\ j=1,\dots,\c\}
\end{equation*}
is the $\K/\F$-{\it base change} for simple characters. 
For each $j$, the simple charac\-ter $\t_{\K}^{j}$ is called the 
$\K/\F$-lift of $\t$ corresponding to the 
$\K/\F$-lift $\b^j$ of $\b$.
\end{defi}

Now let $(\Theta,k,\b)$ be a ps-character over $\F$. 
Let $[\La,n,m,\h(\b)]$ be a realization of the pair 
$(k,\b)$ in a simple central $\F$-algebra $\A$, 
and let $\t$ denote the simple character $\Theta(\La,m,\h)$. 
Let $(k,\b^{j})$, for $j\in\{1,\dots,\c\}$, be the various 
$\K/\F$-lifts of the pair $(k,\b)$, and let $\h^{j}$ denote the 
homo\-morphism of $\K$-algebras 
from $\K(\b^{j})$ to the centralizer of $\K$ in $\qq{\A}^{j}$ 
induced by $\h$. 
Thus the sum of the $\h^{j}$'s is the $\K$-algebra 
homomorphism $\h\otimes{\rm id}_\K$ from $\E\otimes_\F\K$ to $\AH$.
For each $j$, let us denote by $(\Theta_{\K}^{j},k,\b^{j})$ 
the ps-character defined by $([\qq{\Ga}^{j},n,m,\b^j],\t_{\K}^{j})$.

\begin{defi}
The process:
\begin{equation*}
\boldsymbol{b}_{\K/\F}:(\Theta,k,\b)\mapsto
\{(\Theta_{\K}^{j},k,\b^{j}),\ j=1,\dots,\c\}
\end{equation*}
is the $\K/\F$-{\it base change} for ps-characters, and 
$\Theta_{\K}^{j}$ is called the $\K/\F$-lift of $\Theta$
corresponding to the $\K/\F$-lift $\b^j$ of $\b$.
\end{defi}

This definition does not depend on the choice of the realization 
$[\La,n,m,\h(\b)]$.
Indeed, let $[\La',n',m',\h'(\b)]$ be another realization of $(k,\b)$
in a simple central $\F$-algebra $\A'$, and let us write $\t'$ for 
the transfer of $\t$ to $\Cc(\La',m',\h'(\b))$.
Then it follows from Theorem \ref{ResComWithTransAndAEFit} that, 
for each $j$, the $\K/\F$-lifts $\t^{j}_\K$ and $\t^{\prime j}_\K$ 
are transfers of each others.

%%%%%%%%%%%%%%%%%%%%%%%%%%%%%%%%%%%%%%%%%%%%%%%%%%%%%%%%%%%%%%%%%%%%%%%%%%%

\subsection{}
\label{EhOui}

In this paragraph, we study in more details the case where $\c=1$,
that is the case where the residue class degree of $\F(\b)/\F$ 
is prime to $f$.
In this case, the simple pair $(k,\b)$ has exactly one 
$\K/\F$-lift. 
If we write $\La_\K$ for the $\Oo_\K$-lattice sequence defined by 
$\qq{\La}$, then the base change process gives rise to a map:
\begin{equation}
\label{HaleManoa}
\boldsymbol{b}_{\K/\F}:\Cc(\La,m,\b)\to\Cc(\La_{\K},m,\b)
\end{equation}
having the following properties.

\begin{prop}
\label{PropInjEquiv}
The map $\boldsymbol{b}_{\K/\F}$ is injective and $\KK(\La)$-equivariant.
\end{prop}

\begin{proof}
As $\boldsymbol{b}_{\K/\F}$ is the composite of the transfer map from 
$\Cc(\La,m,\b)$ to $\Cc(\qq{\La},m,\b)$ and the interior lifting from 
$\Cc(\qq{\La},m,\b)$ to $\Cc(\La_\K,m,\b)$, this follows from Proposition 
\ref{PropInjEquivInt}.
\end{proof}

Assume we are given two simple strata $[\La,n_i,m_i,\b_i]$, $i=1,2$, in $\A$, 
such that $f_\F(\b_1)$ and $f_\F(\b_2)$ are prime to $\d$.
For each $i$, let $\t_i$ be a simple character in $\Cc(\La,m_i,\b_i)$.

\begin{prop}
\label{Maui}
Assume $\t_1$ and $\t_2$ intertwine in $\mult\A$.
Then $\bc_{\K/\F}(\t_1)$ and $\bc_{\K/\F}(\t_2)$ 
inter\-twine in $\A_{\K}^{\times}$.
\end{prop}

\begin{proof}
Assume $\t_1$ and $\t_2$ are intertwined by $g\in\A^{\times}$. 
By the proof of 
Proposition \ref{LongJohnSilverGeneralCarSimVariante}, 
the characters $\qq{\t}_1$ and $\qq{\t}_2$ are intertwined 
by $\iota(g)$, where $\iota$ denotes the diagonal 
embedding of $\A$ in $\qq{\A}=\Mat_f(\A)$.
As $\iota(g)$ is actually in $\A_{\K}^{\times}$, we deduce that 
the characters 
$\bc_{\K/\F}(\t_1)$ and $\bc_{\K/\F}(\t_2)$ inter\-twine in 
$\A_{\K}^{\times}$.  
\end{proof}

We now suppose that $n_1=n_2$ and $m_1=m_2$.

\begin{prop}
\label{Maui35}
Assume that $\t_1$ and $\t_2$ are equal.
Then $\bc_{\K/\F}(\t_1)$ and $\bc_{\K/\F}(\t_2)$ are equal.
\end{prop}

\begin{proof}
If $\t_1$ and $\t_2$ are equal, then Proposition \ref{Dolomites}
gives us $\qq{\t}_1=\qq{\t}_2$ and Proposition \ref{MauiToho} 
gives us the expected equality.
\end{proof}

Let $[\La,n,m,\b]$ and $[\La',n',m',\b]$ be two realizations 
of the simple pair $(k,\b)$, let $\t$ be a simple character in 
$\Cc(\La,m,\b)$
and let $\t'$ be its transfer in $\Cc(\La',m',\b)$.
The following proposition is a special case of 
Theorem \ref{ResComWithTransAndAEFit}.

\begin{prop}
\label{AbbeC}
The character $\bc_{\K/\F}(\t')$ is the transfer of $\bc_{\K/\F}(\t)$ in 
$\Cc(\La'_\K,m',\b)$.
\end{prop}

Finally, we will need the following result. 
Note that $\Gal(\K/\F)$ acts naturally on $\AH$.

\begin{prop}
\label{GalInv}
Let $\t\in\Cc(\La_\K,m,\b)$ be a simple character.
For any $\s\in\Gal(\K/\F)$,
we have $\t\circ\s\in\Cc(\La_\K,m,\b)$.
\end{prop}

\begin{proof}
One checks by induction on $\b$ that the image of $\Cc(\La_\K,m,\b)$ by 
$\t\mapsto\t\circ\s$ is the set of simple characters attached to the 
image of $[\La_\K,n,m,\b]$ by $\s^{-1}$ with respect to the additive 
character $\psi_{\K}\circ\s$. 
The result follows from the fact that this stratum and the additive 
character $\psi_\K$ are invariant by $\s$. 
\end{proof}

\subsection{}

We prove the following theorem, 
which generalizes \cite[Corollary 3.6.3]{BK}. 

\begin{theo}
\label{Pleurotes}
For $i=1,2$, let $(k_i,\b_i)$ be a simple pair over $\F$.
Let us fix two realizations 
$[\La,n,m,\b_{i}]$ and $[\La',n',m',\b_{i}]$ of $(k_i,\b_i)$.  
Assume $\Cc(\La,m,\b_{i})$ and $\Cc(\La',m',\b_{i})$ do not 
depend on $i$.
Then the transfer map
$\boldsymbol\tau_i:\Cc(\La,m,\b_{i})\to\Cc(\La',m',\b_{i})$
does not depend on $i$.
\end{theo}

\begin{proof}
The proof decomposes into three steps.
\begin{enumerate}
\item 
In the first step, we reduce to the case where the strata are all sound.  
For this, we fix an integer $l$ as in Proposition 
\ref{LongJohnSilverVarianteSound} which is large enough for $\La$ and 
$\La'$. 
Write $\boldsymbol{a}_i$ for the transfer map from $\Cc(\La,m,\b_i)$ to 
$\Cc(\La^{\ddag},m,\b_i)$. 
There is also a map $\boldsymbol{a}_i'$ for $\La'$. 
Thus we have a commutative diagram: 
\begin{equation*}
\diagram 
\Cc(\La^{\ddag},m,\b_i)\rto^{\boldsymbol\tau_i^{\ddag}}
&\Cc(\La^{\prime\ddag},m',\b_i)\\
\Cc(\La,m,\b_i)\rto_{\boldsymbol\tau_i}\uto^{\boldsymbol{a}_i}
&\Cc(\La',m',\b_i)\uto_{\boldsymbol{a}_i'}
\enddiagram
\end{equation*}
where $\tau_i^{\ddag}$ denotes the transfer map from 
$\Cc(\La^{\ddag},m,\b_{i})$ to $\Cc(\La^{\prime\ddag},m',\b_{i})$. 
By Proposition \ref{Dolomites}, the vertical maps $\boldsymbol{a}_i$ 
and $\boldsymbol{a}_i'$ do not depend on $i$, 
and Proposition \ref{Gr9199} implies that the sets 
$\Cc(\La^{\ddag},m,\b_{i})$ and $\Cc(\La^{\prime\ddag},m',\b_{i})$ 
do not depend on $i$.
Since $\boldsymbol{a}_i'$ is bijective, 
the equality $\boldsymbol\tau_1^{\ddag}=\boldsymbol\tau_2^{\ddag}$ 
implies that $\boldsymbol\tau_1=\boldsymbol\tau_2$.  
We thus may replace $\La$ by $\La^{\ddag}$ and $\La'$ by 
$\La ^{\prime\ddag}$ and assume that all the strata are sound.  
\item
We now assume that all the strata are sound, and we reduce to the 
case where the extensions $\F(\b_i)/\F$ are totally ramified.
By Proposition \ref{BrattleStreet}, for each $i$, the simple strata 
$[\La,n,m,\b_{i}]$ and $[\La',n',m',\b_{i}]$ have the same 
embedding type.  
Write $\K_i$ for the maximal unramified extension of $\F$ 
contained in $\F(\b_i)$, and fix $\t_i\in\Cc(\La,m,\b_{i})$. 
Assume that the characters $\t_1$ and $\t_2$ are equal.
Using the ``intertwining implies conjugacy'' theorem 
\cite[Corollary 10.15]{Gr}, one may assume that $\K_1=\K_2$, 
denoted $\K$.
Write $\boldsymbol{l}_i$ for the interior lifting map from 
$\Cc(\La,m,\b_i)$ to $\Cc(\Ga,m,\b_i)$. 
There is also a map $\boldsymbol{l}_i'$ for $\La'$.  
By Theorem \ref{ResComWithTransAndAEFit}, we have a commutative diagram: 
\begin{equation*}
\diagram 
\Cc(\Ga,m,\b_i)\rto^{\boldsymbol\tau_i^{\K}}&\Cc(\Ga^{\prime},m',\b_i)\\
\Cc(\La,m,\b_i)\rto_{\boldsymbol\tau_i}\uto^{\boldsymbol{l}_i}&\Cc(\La',m',\b_i)
\uto_{\boldsymbol{l}_i'}
\enddiagram
\end{equation*}
where $\boldsymbol\tau_i^{\K}$ denotes the transfer map from 
$\Cc(\Ga,m,\b_{i})$ to $\Cc(\Ga^{\prime},m',\b_{i})$. 
By Proposition \ref{MauiToho}, the vertical maps $\boldsymbol{l}_i$ 
and $\boldsymbol{l}_i'$ do not depend on $i$, and 
Theorem \ref{Paraclet} implies that the sets 
$\Cc(\Ga,m,\b_i)$ and $\Cc(\Ga',m',\b_i)$ do not depend on $i$.  
By the same argument as above, using that the map $\boldsymbol{l}_i'$
is injective (see Proposition \ref{PropInjEquivInt}), 
we may assume that $\F(\b_i)$ is totally ramified over $ \F$. 
\item
We now assume that $f_{\F}(\b_1)=f_{\F}(\b_2)=1$, 
and reduce to the split case.  
Let us fix a finite unramified extension $\L/\F$ such 
that the $\L$-algebras $\AL$ and $\AL'$ are split.  
Write $\boldsymbol{b}_i$ for the base change map from 
$\Cc(\La,m,\b_i)$ to $\Cc(\LaL,m,\b_i)$. 
There is also a map $\boldsymbol{b}_i'$ for $\La'$. 
By Proposition \ref{AbbeC}, we have a commutative diagram: 
\begin{equation*}
\diagram 
\Cc(\LaL,m,\b_i)\rto^{\overline{\boldsymbol\tau}_i}
&\Cc(\LaL^{\prime},m',\b_i)\\
\Cc(\La,m,\b_i)\rto_{\boldsymbol\tau_i}\uto^{\boldsymbol{b}_i}
&\Cc(\La',m',\b_i)\uto_{\boldsymbol{b}_i'}
\enddiagram
\end{equation*}
where $\overline{\boldsymbol\tau}_i$ denotes the transfer map from 
$\Cc(\LaL,m,\b_{i})$ to $\Cc(\LaL^{\prime},m',\b_{i})$. 
By Proposition \ref{Maui}, the maps $\boldsymbol{b}_i$ 
and $\boldsymbol{b}_i'$ do not depend on $i$.  
Thus \cite[Theorem 3.5.8]{BK} (the rigidity theorem for simple 
characters in the split case)
implies that the sets of simple characters 
$\Cc(\LaL,m,\b_i)$ and $\Cc(\LaL',m',\b_i)$ do not depend on $i$.  
By the same argument as above, using that the map $\boldsymbol{b}_i'$
is injective (see Proposition \ref{PropInjEquiv}), 
we may assume that $\A$ is split and $\La$ is strict. 
\end{enumerate}
The result then follows from \cite[Corollary 3.6.3]{BK}.
\end{proof}

%%%%%%%%%%%%%%%%%%%%%%%%%%%%%%%%%%%%%%%%%%%%%%%%%%%%%%%%%%%%%%%%%%%%%%%%%%%
%%%%%%%%%%%%%%%%%%%%%%%%%%%%%%%%%%%%%%%%%%%%%%%%%%%%%%%%%%%%%%%%%%%%%%%%%%%

\section{Endo-equivalence of simple characters}
\label{Sec8}

\subsection{}
\label{TRC}

In this paragraph, we prove Theorem \ref{Gata} in the totally ramified 
case. 
For $i=1,2$, let $(\Theta_{i},k,\b_i)$ be a ps-character over $\F$ 
with $f_{\F}(\b_{i})=1$, and suppose that $\Theta_{1}$ and $\Theta_{2}$
are endo-equivalent.
Let $\A$ be a simple central $\F$-algebra and let 
$[\La,n,m,\h_i(\b_i)]$ be realizations of 
$(k,\b_i)$ in $\A$, with $i=1,2$.
Write $\t_i$ for the simple character $\Theta_i(\La,m,\h_i)$. 
We have to prove that $\t_1$ and $\t_2$ are conjugate under $\KK(\La)$.

\medskip

For each $i$, we write $\E_i$ for the $\F$-algebra $\F(\b_i)$, 
which is a totally ramified finite extension of $\F$. 
By assumption, we have $[\E_1:\F]=[\E_2:\F]$.
Using 
Proposition \ref{Hengist6}, there exists a simple central $\F$-algebra 
$\A'$ together with sound real\-iza\-tions $[\La',n',m',\h'_i(\b^{}_i)]$ 
of $(k,\b_i)$, 
with $i=1,2$, such that $k$ divides $m'$ and $\t_1'=\t_2'$, where we
write $\t_i'=\Theta^{}_i(\La',m'_i,\h'_i)$. 

\medskip

Now let $\A$ be a simple central $\F$-algebra and 
$[\La,n,m,\h_i(\b_i)]$ be realizations of $(k,\b_i)$ in 
$\A$, for $i=1,2$.
Let $\V$ denote the simple left $\A$-module on which $\La$ is a lattice 
sequence and write $\D$ for the $\F$-algebra opposite to $\End_{\A}(\V)$.
Let us fix a finite unramified extension $\L$ of $\F$ such that the 
$\L$-algebra 
$\AL=\A\otimes_\F\L$ is split and a simple left $\AL$-module $\VL$.
As $\E_i$ is totally ramified over $\F$, 
the quasi-simple lift $[\overline{\La},n,m,\b_i]$ is a simple stratum in $\AL$
(see \cite[Th\'eor\`eme 2.30]{VS1} and \cite[Re\-mar\-que 2.9]{SeSt}).
We denote by $\Cc(\LaL,m,\b_i)$ the set of simple characters 
attached to this quasi-simple lift 
with respect to the character $\psi\circ\tr_{\L/\F}$.
The base change process developed in paragraph \ref{EhOui} gives 
rise to an injective and $\KK(\La)$-equivariant map:
\begin{equation*}
\boldsymbol{b}_{\L/\F}:\Cc(\La,m,\b_i)\to\Cc(\LaL,m,\b_i),
\end{equation*}
simply denoted $\boldsymbol{b}$.  
We use similar notations for $\A'$.
For each $i$, we write $\t_i$ for the simple character 
$\Theta_i(\La,m,\h_i)$. 
By Proposition \ref{Maui}, we have $\bc(\t'_1)=\bc(\t'_2)$.
By Proposition \ref{AbbeC},
for each $i$, the lifts $\bc(\t_i^{})$ 
and $\bc(\t'_i)$ are transfers of each others.
At this point, we cannot apply \cite{BH,BK} to deduce 
that $\bc(\t_1)$ and $\bc(\t_2)$ are $\KK(\LaL)$-conjugate, because the 
lattice sequence $\La$ is not necessarily strict. 

\medskip

Let us fix a simple right $\E_1\otimes_{\F}\D$-module $\SS$. 
We set $\A(\SS)=\End_{\D}(\SS)$, and denote by $\rho_1$ the natural 
$\F$-algebra homomorphism $\E_1\to\A(\SS)$.
Let $\Ss$ denote the unique (up to translation) $\E_1$-pure 
strict $\Oo_{\D}$-lattice sequence on $\SS$, and 
let us fix an $\F$-algebra homomorphism $\rho_2:\E_2\to\A(\SS)$ such 
that $\Ss$ is $\rho_2(\E_2)$-pure. 
Write $n_0$ for the $\Ss$-valuation of $\rho_i(\b_i)$ and:
\begin{equation*}
m_0=e_{\rho_i(\b_i)}(\Ss)k,
\end{equation*}
which do not depend on $i$. 
We thus can form the 
stratum $[\Ss,n_0,m_0,\rho_i(\b_i)]$, 
which is a realization of $(k,\b_i)$ in $\A(\SS)$.
Write $\vartheta_{i}$ for the simple character $\Theta_i(\Ss,m_0,\rho_i)$.
We now form the simple stratum $[\overline\Ss,n_0,m_0,\rho_i(\b_i)]$ 
in the split simple central $\L$-algebra $\A(\SS)\otimes_\F\L$. 
It is a realization of $(k,\b_i)$ over $\L$, and the $\Oo_{\L}$-lattice
sequence $\overline\Ss$ is strict.
We thus can apply \cite[Theorem 8.7]{BH} and \cite[Theorem 3.5.11]{BK}, 
which imply together that there exists $u\in\KK(\overline\Ss)$ such that:
\begin{equation*}
\bc(\vartheta_{2})(x)=\bc(\vartheta_{1})(uxu^{-1}),
\quad
x\in\H^{m+1}(\rho_2(\b_2),\overline\Ss)=
u^{-1}\H^{m+1}(\rho_1(\b_1),\overline\Ss)u.
\end{equation*}
We need the following lemma. 

\begin{lemm}
\label{CohoArg}
We may assume that $u\in\KK(\Ss)$.
\end{lemm}

\begin{proof}
By Proposition \ref{GalInv}, 
the map $\s\mapsto u^{-1}\s(u)$ is a $1$-cocycle 
on $\Gal(\L/\F)$ with values 
in the $\U(\overline\Ss)$-normalizer of $\bc(\vartheta_2)$, 
which is equal to $\J(\rho_2(\b_2),\overline\Ss)$ according 
to \cite{BK}. 
This cocycle defines a class in the cohomology set:
\begin{equation*}
\H^1(\Gal(\L/\F),\J(\rho_2(\b_2),\overline\Ss)).
\end{equation*}
We claim this cohomology set is trivial. 
According to \cite[Proposition 2.39]{VS1}, it is enough to 
prove that:
\begin{equation*}
\H^1(\Gal(\L/\F),\J(\rho_2(\b_2),\overline\Ss)/\J^1(\rho_2(\b_2),\overline\Ss))
\end{equation*}
is trivial, which is given by a standard filtration argument
(see \cite[\S6]{BG}).
\end{proof}

Using Proposition \ref{PropInjEquiv}, we thus may replace $\rho_2$ by a 
$\KK(\Ss)$-conjugate and assume that the characters 
$\vartheta_{1}$ and $\vartheta_{2}$ are equal. 
We now fix a decomposition:
\begin{equation*}
\V=\V^{1}\oplus\cdots\oplus\V^{l}
\end{equation*}
of $\V$ into simple right $\E_{1}\otimes_{\F}\D$-modules (which all are 
copies of $\SS$) such that the lattice sequence $\La$ 
decomposes into the direct sum of the 
$\La^{j}=\La\cap\V^{j}$, for $j\in\{1,\dots,l\}$.
By choosing, for each $j$, an iso\-mor\-phism of 
$\K(\b)\otimes_{\F}\D$-modules between $\SS$ and $\V^{j}$, 
this gives us an $\F$-algebra homomorphism:
\begin{equation*}
\iota:\A(\SS)\to\A. 
\end{equation*}
Using Lemma \ref{Merimee2}, we may assume that $\iota\circ\rho_1=\h_1$, 
and, by Lemma \ref{TotRam}, on may replace $\h_2$ by a $\KK(\La)$-conjugate 
and assume that $\iota\circ\rho_2=\h_2$.
We now remark that, for each $i$, the map $\vartheta_i\mapsto\t_i$ 
corresponds to the process described in paragraph \ref{Miskatonic}. 
The equality $\t_1=\t_2$ thus follows from Proposition \ref{Dolomites}.

\subsection{}
\label{sectionEESC}
\label{Livesey}

In this paragraph, we reduce the proof of Theorem \ref{Gata}
to the totally ra\-mi\-fied case, which has been treated in paragraph 
\ref{TRC}. 
For $i=1,2$, let $(\Theta_{i},k,\b_{i})$ be a ps-character over $\F$, 
set $\E_i=\F(\b_i)$ and write $\K_i$ for the maximal unramified 
extension of $\F$ contained in $\E_i$, 
and suppose that $\Theta_{1}\thickapprox\Theta_{2}$.
Then we have $[\E_1:\F]=[\E_2:\F]$ and, using 
Proposition \ref{Hengist6}, there is a simple central $\F$-algebra 
$\A$ together with realizations $[\La,n,m,\h_i(\b_i)]$ of $(k,\b_i)$, 
with $i=1,2$, which are sound and have the same embedding type, with 
$k$ dividing $m$ and such that $\h_1(\K_1)=\h_2(\K_2)$, denoted $\K$, 
and $\t_1=\t_2$, where $\t_i=\Theta_i(\La,m_i,\h_i)$.
Let $\C$ denote the centralizer of $\K$ in $\A$ 
and write $[\Ga,n,m,\b_i]$ for the stratum 
in $\C$ associated with $[\La,n,m,\b_i]$ by (\ref{Dulilah}).  
By Proposition \ref{MauiToho}, the $\K/\F$-lifts 
$\t_1^{\K}$ and $\t_2^{\K}$ are equal.

\medskip

Now let $\A'$ be a simple central $\F$-algebra and 
$[\La',n',m',\h'_i(\b^{}_i)]$ be realizations of 
$(k,\b_i)$ in $\A$, with $i=1,2$,
having the same embedding type.
By Remark \ref{Pisot}, we may conjugate $\h'_2$ by $\KK(\La')$ 
and assume that the maximal 
un\-ramified extensions of $\F$ contained in $\h'_1(\E^{}_1)$ and 
$\h'_2(\E^{}_2)$ 
are equal to a common extension $\K'$ of $\F$, say. 
Moreover, by Lemma \ref{ConjPiD}, 
we may conjugate again $\h'_2$ by $\KK(\La')$ 
and assume that the $\F$-algebra isomorphisms 
$\h'_1\circ\h^{-1}_1$ and $\h'_2\circ\h^{-1}_2$ agree on $\K$
(and thus identify $\K$ and $\K'$).
Let $\C'$ denote the centralizer of $\K'$ in $\A'$ and write 
$[\Ga',n',m',\h'_i(\b^{}_i)]$ for the stratum in $\C$ 
associated with $[\La',n',m',\h'_i(\b^{}_i)]$ 
by (\ref{Dulilah}). 
Thus the simple strata $[\Ga,n,m,\h_i(\b_i)]$ and 
$[\Ga',n',m',\h'_i(\b^{}_i)]$ are realizations 
of the same simple pair over $\K$. 
For each $i$, we write $\t'_{i}$ for the character 
$\Theta^{}_i(\La',m',\h'_i)$.
By Theorem \ref{ResComWithTransAndAEFit}, for each $i$, 
the $\K/\F$-lifts $\t_i^{\K}$ and $\t_i^{\prime\K}$ are 
transfers of each others.
Therefore, by paragraph \ref{TRC}, 
there exists $u\in\KK(\Ga')$ such that:
\begin{equation*}
\t_2^{\prime\K}(x)=\t_1^{\prime\K}(uxu^{-1}), 
\quad
x\in\H^{m+1}(\h'_2(\b^{}_2),\Ga')=u^{-1}\H^{m+1}(\h'_1(\b^{}_1),\Ga')u.
\end{equation*}
The equality ${\t'}_1^u=\t'_2$ follows from Proposition \ref{PropInjEquivInt}.

\begin{coro}
\label{Vendries2}
Definition \ref{Petrone2} is equivalent to \cite[Definition 8.6]{BH}.
\end{coro}

\begin{proof}
Assume we are given two ps-characters $(\Theta_i,k,\b_i)$, $i=1,2$, 
which are endo-equivalent in the sense of Definition \ref{Petrone2},
and let $\A$ be a simple central split $\F$-algebra to\-gether with 
realizations $[\La,n_i,m_i,\h_i(\b_i)]$ of $(k,\b_i)$ in $\A$, with $i=1,2$, 
such that $\La$ is strict.
By Theorem \ref{EndoClasChar}, the simple characters 
$\Theta_i(\La,m_i,\h_i)$ intertwine in $\mult\A$, that is, 
the ps-characters $(\Theta_i,k,\b_i)$ are endo-equivalent 
in the sense of \cite[Definition 8.6]{BH}. 
Conversely, two simple pairs which are endo-equivalent in this 
sense are clearly endo-equivalent in the sense of Definition 
\ref{Petrone2}. 
\end{proof}

\begin{coro}
The relation $\thickapprox$ on ps-characters is an equivalence relation.
\end{coro}

\begin{proof}
This comes from \cite[Corollary 8.10]{BH} together with Corollary 
\ref{Vendries2}.
\end{proof}

\section{The endo-class of a discrete series representation}

\subsection{}
\label{data}

Let $\A$ be a simple central $\F$-algebra, and let $\V$ be 
a simple left $\A$-module.
Associated with it, there is an $\F$-division algebra $\D$.
We write $d$ for the reduced degree of $\D$ over $\F$ and
$m$ for the dimension of $\V$ as a right $\D$-vector space. 
We set $\G=\mult\A$, identified with $\GL_{m}(\D)$.

Let $\pi$ be an irreducible smooth representation of $\G$, 
and assume that its inertial class (in the sense of Bushnell 
and Kutzko's theory of types \cite{BK2}), denoted $\ss(\pi)$,
is homogeneous.
Thus there is a positive integer $r$ dividing $m$, 
an irreducible cuspidal representation $\rho$ of the
group $\G_0=\GL_{m/r}(\D)$ and unramified characters $\chi_i$ of 
$\G_0$, with $i\in\{1,\dots,r\}$, such that $\pi$ is isomorphic
to a quotient of the normalized parabolically induced 
representation $\rho\chi_1\times\dots\times\rho\chi_r$
(see for instance \cite{BHLS} for the notation).

In this section, we associate with $\pi$ an endo-class 
$\boldsymbol\Theta(\pi)$ over $\F$, and show that it depends 
only on the inertial class $\ss=\ss(\pi)$. 

\subsection{}
\label{daga}

Let $\pi$ be a representation of $\G$ as above, and write $\ss=\ss(\pi)$ 
for its inertial class.
According to \cite[Th\'eor\`eme 5.23]{SeSt},
this inertial class possesses a type in the sense of \cite{BK2}.
Such a type is a pair $(\J,\l)$ formed of a compact open 
subgroup $\J$ of $\G$ and of an irreducible smooth representation 
$\lambda$ of $\J$ such that an irreducible smooth representation
of $\G$ has inertial class $\ss$ if and only if $\l$ occurs in its 
restriction to $\J$.
More precisely, $(\J,\l)$ can be chosen to be a {\it simple type} 
in the sense of \cite{VS3}. 
We won't give a precise description of simple types;
the only property of interest for us is the following fact, which 
is a weak form of \cite[Th\'eor\`eme 5.23]{SeSt}. 

\begin{enonce}{Fact}
\label{Blig} 
There is a simple stratum $[\AA,n,0,\b]$ in $\A$ together with 
a simple character $\t\in\Cc(\AA,0,\b)$ such that the order
$\AA\cap\B$ (with $\B$ the centralizer of $\F(\b)$ in $\A$)
is principal of period $r$ and the character $\t$ occurs
in the restriction of $\pi$ to $\H^1(\b,\AA)$.
\end{enonce}

Neither $[\AA,n,0,\b]$ 
nor the character $\t$ are uniquely determined. 
We let $(\Theta,0,\b)$ be the ps-character 
defined by the pair $([\AA,n,0,\b],\t)$ and 
we denote by ${\boldsymbol\Theta}$ its endo-class. 

\begin{theo}
\label{Endo-Class}
The endo-class $\boldsymbol\Theta$ depends only on the 
inertial class $\ss$. 
\end{theo}

\begin{proof}
We have to prove that $\boldsymbol\Theta$ does not depend 
on the choice of the simple stratum $[\AA,n,0,\b]$ 
and the simple character $\t$ satisfying the conditions of 
Fact \ref{Blig}. 
For $i=1,2$, let $[\AA_i,n_i,0,\b_i]$ be a simple stratum 
and $\t_i$ be a simple character satisfying the 
conditions of Fact \ref{Blig}, and let $(\Theta_i,0,\beta_i)$ 
denote the ps-character that it defines.
Let $\AA_i'$ denote the unique principal $\Oo_\F$-order in $\A$
such that the pair $(\E^{}_i,\AA_i')$ is a sound embedding in 
$\A$ (see Lem\-ma \ref{PerRinc}) and let $\t_i'$ denote the 
transfer of $\t_i$ in $\Cc(\AA_i',0,\b^{}_i)$. 
By a standard argument using \cite[Th\'eor\`eme 2.13]{SeSt}, the 
character $\t_i'$ occurs in the restriction of 
$\pi$ to $\H^1(\b^{}_i,\AA_i')$.
Therefore, we can assume without changing $\Theta_i$ 
that $(\E_i,\AA_i)$ is sound. 
By Lemma \ref{PerRinc} again, the principal orders $\AA_1$ 
and $\AA_2$ have the same period (as $\AA_i\cap\B_i$ has period $r$).
Thus one may conjugate $([\AA_1,n_1,0,\b_1],\t_1)$ by an 
element of $\G$ and assume that $\AA_1=\AA_2$, denoted $\AA$. 
For each $i$, we have $\t_i\in\Cc(\AA,0,\b_i)$ and $\t_i$ occurs 
in the restriction of $\pi$ to the subgroup $\H^1(\b_i,\AA)$.
Thus the characters $\t_1$ and $\t_2$ intertwine in $\mult\A$. 
To prove that $\Theta_1$ and $\Theta_2$ are endo-equivalent, 
it remains to prove that $\F(\b_1)$
and $\F(\b_2)$ have the same degree over $\F$.
By copying the beginning of the proof of Lemma \ref{Hectare1},
we get $n_1=n_2$.
We now write $f$ for the greatest common divisor of $f_{\F}(\b_1)$ and 
$f_{\F}(\b_2)$ and $\K_i$ for the maximal un\-ramified extension of $\F$ 
contained in $\F(\b_i)$.
Then Theorem \ref{Grabinoulor} gives us the expected equality.
\end{proof}

We call the class $\boldsymbol\Theta$ 
{\it the endo-class} of $\pi$ (or of $\ss$). 
We have actually obtained more.

\begin{theo}
Let $\pi$ be an irreducible representation with inertial class $\ss$
as above, and let $[\AA,n,0,\b]$ and $\t$ satisfy the conditions of 
Fact \ref{Blig}.
Assume moreover $[\AA,n,0,\b]$ is sound. 
The following objects are invariants of the inertial class $\ss$:
\begin{enumerate}
\item 
the ramification index $e_\F(\b)$ and the residue class degree
$f_F(\b)$; 
\item
the $\G$-conjugacy class of the order $\AA$;
\item
the embedding type of $(\F(\b),\AA)$.
\end{enumerate}
\end{theo}

\begin{proof}
Assertions (1) and (2) have already been proved. 
Assertion (3) follows immediately from Lemma \ref{Boubinette}.
\end{proof}

\subsection{} 

Recall that an irreducible smooth representation $\pi$ of $\G$ is 
{\it essentially squa\-re integrable} if there is a
character $\chi$ of $\G$ such that $\pi\chi$ is unitary and has 
a non-zero coefficient which is square integrable on $\G/\Z$,
where $\Z$ denotes the centre of $\G$.
We write $\EuScript{D}(\G)$ for the set of isomorphism classes of 
essentially squa\-re integrable representation of $\G$.
Ac\-cor\-ding to \cite[\S2.2]{BHLS}, any essentially squa\-re integrable 
representation of $\G$ has an inertial class which is homogeneous 
in the sense of paragraph \ref{data}.
Thus the construction of paragraph \ref{daga} gives us a map:
\begin{equation}
\boldsymbol\Theta_\G:\EuScript{D}(\G)\to\Ee(\F)
\end{equation}
from $\EuScript{D}(\G)$ to the set of endo-classes
of ps-characters over $\F$.

\medskip

We now write $\H=\GL_{md}(\F)$, and let $\JL$ denote 
the Jacquet-Langlands correspondence (see \cite{Ba,DKV})
from $\EuScript{D}(\G)$ to $\EuScript{D}(\H)$.
We have the following conjecture. 

\begin{conj}
\label{EndoClassJL1}
For any $\pi$ in $\EuScript{D}(G)$, we have:
\begin{equation}
\label{EndoClassJL}
\Theta_{\H}(\JL(\pi))=\Theta_{\G}(\pi).
\end{equation}
\end{conj}

This conjecture generalizes the fact that, for any level zero 
representation $\pi$ in $\EuScript{D}(\G)$, the representation 
$\JL(\pi)$ has level zero. 
It allows one to refine the correspondence $\JL$ by fixing the 
endo-class: given ${\bf\Theta}$ an endo-class over $\F$, 
Conjecture \ref{EndoClassJL1} implies that we have a bijective 
map:
\begin{equation*}
\JL_{{\bf\Theta}}:\EuScript{D}(\G,{\bf\Theta})\to
\EuScript{D}(\H,{\bf\Theta})
\end{equation*}
where we write 
$\EuScript{D}(\G,{\bf\Theta})$ for the set of isomorphism classes 
of essentially square integrable representations of 
$\G$ of endo-class ${\bf\Theta}$. 

%\bibliography{Biblio}

\providecommand{\bysame}{\leavevmode ---\ }
\providecommand{\og}{``}
\providecommand{\fg}{''}
\providecommand{\smfandname}{\&}
\providecommand{\smfedsname}{\'eds.}
\providecommand{\smfedname}{\'ed.}
\providecommand{\smfmastersthesisname}{M\'emoire}
\providecommand{\smfphdthesisname}{Th\`ese}

\end{document}